\documentclass[a4paper,11pt,reqno]{amsart}

\usepackage[utf8]{inputenc}
\usepackage[T1]{fontenc}
\usepackage{hyperref}
\usepackage{xcolor}
\hypersetup{
    colorlinks,
    linkcolor={red!50!black},
    citecolor={blue!50!black},
    urlcolor={blue!80!black}
}

\usepackage{microtype}
\usepackage{MnSymbol}
\usepackage{textcomp}

\usepackage[mathcal]{euscript}
\let\mathcaltmp\mathcal
\usepackage{mathrsfs}
\let\mathcal\mathscr
\let\mathscr\mathcaltmp

\newtheoremstyle{plain}% name
  {.5\baselineskip\@plus.2\baselineskip\@minus.2\baselineskip}% Space above
  {.5\baselineskip\@plus.2\baselineskip\@minus.2\baselineskip\@plus.5em}% Space below
  {\slshape}% Body font
  {}%Indent amount (empty = no indent, \parindent = para indent)
  {\bfseries}%  Thm head font
  {.}%       Punctuation after thm head
  { }%      Space after thm head: " " = normal interword space;
  {}%       Thm head spec

\newtheoremstyle{definition}% name
  {.5\baselineskip\@plus.2\baselineskip\@minus.2\baselineskip}% Space above
  {0.8\baselineskip\@plus.2\baselineskip\@minus.2\baselineskip\@plus.5em}% Space below
  {}
  {}
  {\bfseries}
  {.}
  { }
  {}

\makeatletter
\newcommand{\eqnum}{\refstepcounter{equation}\textup{\tagform@{\theequation}}}
\makeatother

\numberwithin{equation}{section}

\newtheorem{thm}{Theorem}[section]
\newtheorem{thmX}{Theorem}
\newtheorem{corX}[thmX]{Corollary}
\newtheorem{conjX}[thmX]{Conjecture}

\newtheorem*{thm*}{Theorem}
\newtheorem{lem}[thm]{Lemma}
\newtheorem{cor}[thm]{Corollary}
\newtheorem{prop}[thm]{Proposition}
\newtheorem*{defthm*}{Definition/Theorem}

\theoremstyle{definition}
\newtheorem{defn}[thm]{Definition}
\newtheorem{rem}[thm]{Remark}
\newtheorem{exam}[thm]{Example}
\newtheorem{constr}[thm]{Construction}

\newtheorem*{exam*}{Example}

\newcommand\arXiv[1]{\href{http://arxiv.org/abs/#1}{arXiv:#1}}
\newcommand\mathAG[1]{\href{http://arxiv.org/abs/math/#1}{math.AG/#1}}

\usepackage{xspace}

\newcommand{\nc}{\newcommand}
\nc{\renc}{\renewcommand}
\nc{\ssec}{\subsection}
\nc{\sssec}{\subsubsection}
\nc{\on}{\operatorname}
\nc{\term}[1]{#1\xspace}

\newcommand{\changelocaltocdepth}[1]{%
  \addtocontents{toc}{\protect\setcounter{tocdepth}{#1}}}

\usepackage{xpatch}
\makeatletter   
\xpatchcmd{\@tocline}
  {\hfil\hbox to\@pnumwidth{\@tocpagenum{#7}}\par}
  {\ifnum#1<0\hfill\else\dotfill\fi\hbox to\@pnumwidth{\@tocpagenum{#7}}\par}
  {}{}
\def\l@subsection{\@tocline{2}{0pt}{4pc}{6pc}{}}
\def\l@subsubsection{\@tocline{3}{0pt}{8pc}{8pc}{}}
\makeatother

\makeatletter
\def\paragraph{\@startsection{paragraph}{4}%
  \z@\z@{-\fontdimen2\font}%
  {\normalfont\bfseries}}
\makeatother

\usepackage{tikz}
\usetikzlibrary{matrix}
\usepackage{tikz-cd}

\tikzset{
  commutative diagrams/.cd,
  arrow style=tikz,
  diagrams={>=latex}}
\makeatletter
\tikzset{
  column sep/.code=\def\pgfmatrixcolumnsep{\pgf@matrix@xscale*(#1)},
  row sep/.code   =\def\pgfmatrixrowsep{\pgf@matrix@yscale*(#1)},
  matrix xscale/.code=%
    \pgfmathsetmacro\pgf@matrix@xscale{\pgf@matrix@xscale*(#1)},
  matrix yscale/.code=%
    \pgfmathsetmacro\pgf@matrix@yscale{\pgf@matrix@yscale*(#1)},
  matrix scale/.style={/tikz/matrix xscale={#1},/tikz/matrix yscale={#1}}}
\def\pgf@matrix@xscale{1}
\def\pgf@matrix@yscale{1}
\makeatother

\usepackage[inline]{enumitem}
\setlist[enumerate,1]{label={(\alph*)},itemsep=\parskip}

\newlist{enumcompress}{enumerate}{1}
\setlist[enumcompress,1]{label={(\roman*)},itemsep=0.2em}

\newlist{thmlist}{enumerate}{1}
\setlist[thmlist,1]{
  label={\em(\roman*)}, ref={(\roman*)},
  itemsep=0.5em,
  leftmargin=*,
  align=right,widest=vi)}

\newlist{thmlistbis}{enumerate}{1}
\setlist[thmlistbis,1]{
  label={\em(\roman*~\textit{bis})},
  ref={(\roman*}~\textit{bis}\upshape{)},
  itemsep=0.5em,
  leftmargin=0pt, align=right, widest=vi)}

\newlist{defnlist}{enumerate}{2}
\setlist[defnlist,1]{
  label={(\roman*)}, ref={(\roman*)},
  itemsep=0.5em,
  leftmargin=*,
  align=right, widest=vi)}

\setlist[defnlist,2]{
  label={(\alph*)}, ref={(\alph*)},
  itemsep=0.75em,
  labelsep=0em,labelindent=0em,leftmargin=*,align=left,widest=vi),
  topsep=0.75em}

\newlist{inlinelist}{enumerate*}{1}
\setlist[inlinelist,1]{label={(\alph*)}}

\newlist{inlinedefnlist}{enumerate*}{1}
\definecolor{green}{HTML}{38550C}
\setlist[inlinedefnlist,1]{label={\color{green}(\roman*)}}

\nc{\cA}{\ensuremath{\mathcal{A}}\xspace}
\nc{\cB}{\ensuremath{\mathcal{B}}\xspace}
\nc{\cC}{\ensuremath{\mathcal{C}}\xspace}
\nc{\cD}{\ensuremath{\mathcal{D}}\xspace}
\nc{\cE}{\ensuremath{\mathcal{E}}\xspace}
\nc{\cF}{\ensuremath{\mathcal{F}}\xspace}
\nc{\cG}{\ensuremath{\mathcal{G}}\xspace}
\nc{\cH}{\ensuremath{\mathcal{H}}\xspace}
\nc{\cI}{\ensuremath{\mathcal{I}}\xspace}
\nc{\cJ}{\ensuremath{\mathcal{J}}\xspace}
\nc{\cK}{\ensuremath{\mathcal{K}}\xspace}
\nc{\cL}{\ensuremath{\mathcal{L}}\xspace}
\nc{\cM}{\ensuremath{\mathcal{M}}\xspace}
\nc{\cN}{\ensuremath{\mathcal{N}}\xspace}
\nc{\cO}{\ensuremath{\mathcal{O}}\xspace}
\nc{\cP}{\ensuremath{\mathcal{P}}\xspace}
\nc{\cQ}{\ensuremath{\mathcal{Q}}\xspace}
\nc{\cR}{\ensuremath{\mathcal{R}}\xspace}
\nc{\cS}{\ensuremath{\mathcal{S}}\xspace}
\nc{\cT}{\ensuremath{\mathcal{T}}\xspace}
\nc{\cU}{\ensuremath{\mathcal{U}}\xspace}
\nc{\cV}{\ensuremath{\mathcal{V}}\xspace}
\nc{\cW}{\ensuremath{\mathcal{W}}\xspace}
\nc{\cX}{\ensuremath{\mathcal{X}}\xspace}
\nc{\cY}{\ensuremath{\mathcal{Y}}\xspace}
\nc{\cZ}{\ensuremath{\mathcal{Z}}\xspace}

\nc{\sA}{\ensuremath{\mathscr{A}}\xspace}
\nc{\sB}{\ensuremath{\mathscr{B}}\xspace}
\nc{\sC}{\ensuremath{\mathscr{C}}\xspace}
\nc{\sD}{\ensuremath{\mathscr{D}}\xspace}
\nc{\sE}{\ensuremath{\mathscr{E}}\xspace}
\nc{\sF}{\ensuremath{\mathscr{F}}\xspace}
\nc{\sG}{\ensuremath{\mathscr{G}}\xspace}
\nc{\sH}{\ensuremath{\mathscr{H}}\xspace}
\nc{\sI}{\ensuremath{\mathscr{I}}\xspace}
\nc{\sJ}{\ensuremath{\mathscr{J}}\xspace}
\nc{\sK}{\ensuremath{\mathscr{K}}\xspace}
\nc{\sL}{\ensuremath{\mathscr{L}}\xspace}
\nc{\sM}{\ensuremath{\mathscr{M}}\xspace}
\nc{\sN}{\ensuremath{\mathscr{N}}\xspace}
\nc{\sO}{\ensuremath{\mathscr{O}}\xspace}
\nc{\sP}{\ensuremath{\mathscr{P}}\xspace}
\nc{\sQ}{\ensuremath{\mathscr{Q}}\xspace}
\nc{\sR}{\ensuremath{\mathscr{R}}\xspace}
\nc{\sS}{\ensuremath{\mathscr{S}}\xspace}
\nc{\sT}{\ensuremath{\mathscr{T}}\xspace}
\nc{\sU}{\ensuremath{\mathscr{U}}\xspace}
\nc{\sV}{\ensuremath{\mathscr{V}}\xspace}
\nc{\sW}{\ensuremath{\mathscr{W}}\xspace}
\nc{\sX}{\ensuremath{\mathscr{X}}\xspace}
\nc{\sY}{\ensuremath{\mathscr{Y}}\xspace}
\nc{\sZ}{\ensuremath{\mathscr{Z}}\xspace}

\nc{\bA}{\ensuremath{\mathbf{A}}\xspace}
\nc{\bB}{\ensuremath{\mathbf{B}}\xspace}
\nc{\bC}{\ensuremath{\mathbf{C}}\xspace}
\nc{\bD}{\ensuremath{\mathbf{D}}\xspace}
\nc{\bE}{\ensuremath{\mathbf{E}}\xspace}
\nc{\bF}{\ensuremath{\mathbf{F}}\xspace}
\nc{\bG}{\ensuremath{\mathbf{G}}\xspace}
\nc{\bH}{\ensuremath{\mathbf{H}}\xspace}
\nc{\bI}{\ensuremath{\mathbf{I}}\xspace}
\nc{\bJ}{\ensuremath{\mathbf{J}}\xspace}
\nc{\bK}{\ensuremath{\mathbf{K}}\xspace}
\nc{\bL}{\ensuremath{\mathbf{L}}\xspace}
\nc{\bM}{\ensuremath{\mathbf{M}}\xspace}
\nc{\bN}{\ensuremath{\mathbf{N}}\xspace}
\nc{\bO}{\ensuremath{\mathbf{O}}\xspace}
\nc{\bP}{\ensuremath{\mathbf{P}}\xspace}
\nc{\bQ}{\ensuremath{\mathbf{Q}}\xspace}
\nc{\bR}{\ensuremath{\mathbf{R}}\xspace}
\nc{\bS}{\ensuremath{\mathbf{S}}\xspace}
\nc{\bT}{\ensuremath{\mathbf{T}}\xspace}
\nc{\bU}{\ensuremath{\mathbf{U}}\xspace}
\nc{\bV}{\ensuremath{\mathbf{V}}\xspace}
\nc{\bW}{\ensuremath{\mathbf{W}}\xspace}
\nc{\bX}{\ensuremath{\mathbf{X}}\xspace}
\nc{\bY}{\ensuremath{\mathbf{Y}}\xspace}
\nc{\bZ}{\ensuremath{\mathbf{Z}}\xspace}

\nc{\bbA}{\ensuremath{\mathbb{A}}\xspace}
\nc{\bbB}{\ensuremath{\mathbb{B}}\xspace}
\nc{\bbC}{\ensuremath{\mathbb{C}}\xspace}
\nc{\bbD}{\ensuremath{\mathbb{D}}\xspace}
\nc{\bbE}{\ensuremath{\mathbb{E}}\xspace}
\nc{\bbF}{\ensuremath{\mathbb{F}}\xspace}
\nc{\bbG}{\ensuremath{\mathbb{G}}\xspace}
\nc{\bbH}{\ensuremath{\mathbb{H}}\xspace}
\nc{\bbI}{\ensuremath{\mathbb{I}}\xspace}
\nc{\bbJ}{\ensuremath{\mathbb{J}}\xspace}
\nc{\bbK}{\ensuremath{\mathbb{K}}\xspace}
\nc{\bbL}{\ensuremath{\mathbb{L}}\xspace}
\nc{\bbM}{\ensuremath{\mathbb{M}}\xspace}
\nc{\bbN}{\ensuremath{\mathbb{N}}\xspace}
\nc{\bbO}{\ensuremath{\mathbb{O}}\xspace}
\nc{\bbP}{\ensuremath{\mathbb{P}}\xspace}
\nc{\bbQ}{\ensuremath{\mathbb{Q}}\xspace}
\nc{\bbR}{\ensuremath{\mathbb{R}}\xspace}
\nc{\bbS}{\ensuremath{\mathbb{S}}\xspace}
\nc{\bbT}{\ensuremath{\mathbb{T}}\xspace}
\nc{\bbU}{\ensuremath{\mathbb{U}}\xspace}
\nc{\bbV}{\ensuremath{\mathbb{V}}\xspace}
\nc{\bbW}{\ensuremath{\mathbb{W}}\xspace}
\nc{\bbX}{\ensuremath{\mathbb{X}}\xspace}
\nc{\bbY}{\ensuremath{\mathbb{Y}}\xspace}
\nc{\bbZ}{\ensuremath{\mathbb{Z}}\xspace}

\nc{\sfC}{\ensuremath{\mathsf{C}}\xspace}
\nc{\sfD}{\ensuremath{\mathsf{D}}\xspace}
\nc{\sfE}{\ensuremath{\mathsf{E}}\xspace}
\nc{\sfX}{\ensuremath{\mathsf{X}}\xspace}
\nc{\sfY}{\ensuremath{\mathsf{Y}}\xspace}

\nc{\mrm}[1]{\ensuremath{\mathrm{#1}}\xspace}
\nc{\mfr}[1]{\ensuremath{\mathfrak{#1}}\xspace}
\nc{\mit}[1]{\ensuremath{\mathit{#1}}\xspace}
\nc{\mbf}[1]{\ensuremath{\mathbf{#1}}\xspace}
\nc{\mcal}[1]{\ensuremath{\mathcal{#1}}\xspace}
\nc{\msc}[1]{\ensuremath{\mathscr{#1}}\xspace}
\nc{\msf}[1]{\ensuremath{\mathsf{#1}}\xspace}

\nc{\sub}{\subseteq}
\nc{\too}{\longrightarrow}
\nc{\hook}{\hookrightarrow}
\nc{\hooklongrightarrow}{\lhook\joinrel\longrightarrow}
\nc{\hooklong}{\hooklongrightarrow}
\nc{\hooklongleftarrow}{\longleftarrow\joinrel\rhook}
\nc{\twoheadlongrightarrow}{\relbar\joinrel\twoheadrightarrow}
\nc{\longrightleftarrows}{\ \raisebox{0.3ex}{\(\mathrel{\substack{\xrightarrow{\rule{1em}{0em}} \\[-1ex] \xleftarrow{\rule{1em}{0em}}}}\)}\ }

\renc{\ge}{\geqslant}
\renc{\geq}{\geqslant}
\renc{\le}{\leqslant}
\renc{\leq}{\leqslant}

\nc{\id}{\mathrm{id}}

\DeclareMathOperator{\rk}{\mathrm{rk}}
\DeclareMathOperator{\Hom}{\on{Hom}}
\nc{\uHom}{\underline{\smash{\Hom}}}
\DeclareMathOperator{\Maps}{\on{Maps}}
\DeclareMathOperator{\Aut}{\on{Aut}}
\DeclareMathOperator{\End}{\on{End}}
\DeclareMathOperator{\Sym}{\on{Sym}}
\nc{\uEnd}{\underline{\smash{\End}}}

\nc{\colim}{\varinjlim}
\renc{\lim}{\varprojlim}
\nc{\Cofib}{\on{Cofib}}
\nc{\Fib}{\on{Fib}}
\nc{\initial}{\varnothing}
\nc{\op}{\mathrm{op}}

\DeclareMathOperator*{\fibprod}{\times}

\renc{\setminus}{\smallsetminus}

\usepackage{mathtools}
\DeclarePairedDelimiter\abs{\lvert}{\rvert}%

\newcommand{\thmref}[1]{Theorem~\ref{#1}}

\newcommand{\secref}[1]{Sect.~\ref{#1}}
\newcommand{\ssecref}[1]{Subsect. ~\ref{#1}}
\newcommand{\sssecref}[1]{(\ref{#1})}
\newcommand{\lemref}[1]{Lemma~\ref{#1}}
\newcommand{\propref}[1]{Proposition~\ref{#1}}
\newcommand{\corref}[1]{Corollary~\ref{#1}}
\newcommand{\conjref}[1]{Conjecture~\ref{#1}}
\newcommand{\remref}[1]{Remark~\ref{#1}}
\newcommand{\defnref}[1]{Definition~\ref{#1}}

\renewcommand{\eqref}[1]{(\ref{#1})}

\newcommand{\examref}[1]{Example~\ref{#1}}

\newcommand{\itemref}[1]{\ref{#1}}

\nc{\A}{\bA}
\renc{\P}{\bP}
\nc{\Spec}{\on{Spec}}
\nc{\uSpec}{\underline{\smash{\Spec}}}
\nc{\Proj}{\on{Proj}}
\nc{\uProj}{\underline{\smash{\Proj}}}
\nc{\bDelta}{\mathbf{\Delta}}
\nc{\Cech}{\textnormal{\v{C}}}
\nc{\QCoh}{\on{QCoh}}
\nc{\Perf}{\on{Perf}}
\nc{\cl}{{\mrm{cl}}}
\nc{\Bl}{\on{Bl}}
\nc{\Nl}{\on{N}}
\nc{\Dl}{\on{D}}
\nc{\El}{\on{E}}
\nc{\Tl}{\on{T}}
\nc{\et}{\mrm{\acute{e}t}}
\nc{\rightrightrightarrows}
  {\mathrel{\substack{\textstyle\rightarrow\\[-0.6ex]
   \textstyle\rightarrow \\[-0.6ex]
   \textstyle\rightarrow}}}
\nc{\rightrightrightrightarrows}
  {\mathrel{\substack{\textstyle\rightarrow\\[-0.6ex]
   \textstyle\rightarrow \\[-0.6ex]
   \textstyle\rightarrow \\[-0.6ex]
   \textstyle\rightarrow}}}
\nc{\Einfty}{{\sE_\infty}}
\renc{\sp}{\mrm{sp}}
\nc{\cosp}{\mrm{cosp}}
\nc{\Td}{\on{Td}}
\nc{\ch}{\on{ch}}
\nc{\RGamma}{R\Gamma}
\nc{\red}{\mrm{red}}
\nc{\der}{{\mrm{der}}}
\nc{\Mod}{{\mrm{Mod}}}
\nc{\Gr}{{\on{Gr}}}
\nc{\Ind}{\on{Ind}}
\nc{\form}{\widehat}
\renc{\L}{\bL}
\nc{\otimesL}{\mathchoice{\overset{\bL}{\otimes}}{\otimes^\bL}{\otimes^\bL}{\otimes^\bL}}
\nc{\fibprodR}{\fibprod^R}
\nc{\uRHom}{\bR\uHom}
\nc{\uMaps}{\underline{\smash{\Maps}}}
\nc{\GL}{\mrm{GL}}
\nc{\Vect}{\on{Vect}}
\nc{\Fun}{\on{Fun}}
\nc{\vb}[1]{\langle #1\rangle}
\nc{\V}{\bV}
\nc{\Gm}{{\bG_m}}
\nc{\pt}{\mrm{pt}}
\nc{\pr}{\mrm{pr}}
\nc{\uAut}{\underline{\Aut}}
\nc{\Pic}{{\on{Pic}}}
\nc{\uPic}{{\underline{\on{Pic}}}}
\nc{\dash}{\textnormal{-}}
\nc{\Spt}{\mrm{Spt}}
\nc{\Stk}{\mrm{Stk}}
\nc{\dStk}{\mrm{dStk}}
\nc{\dSch}{\mrm{dSch}}
\nc{\Spc}{\mrm{Spc}}
\nc{\dSpc}{\mrm{dSpc}}
\nc{\Catoo}{\mrm{Cat}_\infty}
\nc{\Grpdoo}{\mrm{Grpd}_\infty}
\nc{\counit}{\mrm{counit}}
\nc{\tr}{\mrm{tr}}
\nc{\cotr}{\mrm{cotr}}
\nc{\eul}{\mrm{eul}}
\nc{\gys}{\mrm{gys}}
\nc{\Shv}{\on{Shv}}
\nc{\h}{\on{h}}
\nc{\modmod}{/\!\!/}
\nc{\VCart}{\on{VCart}}
\nc{\BGm}{B\bG_m}
\nc{\AGm}{\Theta}
\nc{\quo}[1]{{[{#1}]}}
\nc{\C}{\mrm{C}}
\nc{\Arr}{\on{Arr}}
\nc{\R}{\sR}
\nc{\Rext}{\sR^{\mrm{ext}}}
\nc{\rmR}{\mrm{R}}
\nc{\rmRext}{\mrm{R}^{\mrm{ext}}}
\nc{\Ban}{\mrm{Ban}}
\nc{\Aff}{\mrm{Aff}}
\nc{\Supp}{\on{Supp}}
\nc{\Fil}{\on{Fil}}
\nc{\Z}{\mathbf{Z}}
\nc{\N}{\mathbf{N}}
\nc{\HBM}{\on{H}^{\mrm{BM}}}
\nc{\vd}{\mrm{vd}}
\nc{\vir}{\mrm{vir}}
\nc{\sNv}{\sN}

\nc{\scr}{\term{derived commutative ring}}
\nc{\scrs}{\term{derived commutative rings}}

\nc{\inftyCat}{\term{$\infty$-category}}
\nc{\inftyCats}{\term{$\infty$-categories}}

\nc{\inftyGrpd}{\term{$\infty$-groupoid}}
\nc{\inftyGrpds}{\term{$\infty$-groupoids}}

\nc{\dA}{\term{derived Artin}}

\nc{\spref}[1]{\href{http://stacks.math.columbia.edu/tag/#1}{#1}}

\title[Derived Weil restrictions]{Deformation to the normal bundle and blow-ups via derived Weil restrictions\vspace{-2mm}}
\author[J. Hekking]{Jeroen Hekking}
\address{Fakult\"at f\"ur Mathematik, Universit\"at Regensburg, 93051 Regensburg, Germany}
\email{jeroen.hekking@ur.de}

\author[A.\,A. Khan]{Adeel A. Khan}
\address{Institute of Mathematics, Academia Sinica, 10617 Taipei, Taiwan}
\address{National Center for Theoretical Sciences, National Taiwan University, 106 Taipei, Taiwan}
\email{adeelkhan@as.edu.tw}

\author[D. Rydh]{David Rydh}
\address{KTH Royal Institute of Technology, Department of Mathematics, SE-100 44 Stockholm, Sweden}
\email{dary@math.kth.se}
\date{2025-11-24}



\begin{document}

\begin{abstract}
  We develop an analogue of the deformation to the normal cone in the context of derived algebraic geometry. This provides \emph{any} given morphism of derived stacks with a degeneration to the zero section of its normal bundle (i.e., its $1$-shifted relative tangent bundle).
  The construction is realized via the derived Weil restriction along the zero section of the affine line.
  We prove a general algebraicity theorem for derived Weil restrictions along finite but possibly non-flat morphisms.
  As an application of the theory, we study derived blow-ups along arbitrary closed centres, generalizing previous works of the authors in the quasi-smooth case.
  \vspace{-5mm}
\end{abstract}

\maketitle

\renewcommand\contentsname{\vspace{-1cm}}
\tableofcontents

\parskip 0.75em

\thispagestyle{empty}

%%%%%%%%%%%%%%%%%%%%%%%%%%%%%%%%%%%%%%%%%%%%%%%%%%%%%%%%%%%%%%%%%%%%%%%%%%%

\changelocaltocdepth{1}
%!TEX root = ../weilres.tex

\section*{Introduction}

A central motivation for derived algebraic geometry arises from the study of singular moduli spaces and stacks appearing in modern enumerative geometry, geometric representation theory, and arithmetic geometry.
The \emph{hidden smoothness} principle expressed by M.~Kontsevich in \cite{KontsevichEnumeration} predicted that such moduli problems should carry natural derived enhancements that are smooth in some derived sense.
With the subsequent development of derived algebraic geometry \cite{CFKQuot,HAG2,LurieDAG}, this principle was promoted to rigorous theorems: see for instance \cite{CFKQuot}, \cite{CFKHilb}, \cite[\S 2]{SchurgToenVezzosi}, and \cite[Thm.~0.0.12]{KhanNCTS}.

This work is guided by the idea that hidden smoothness should moreover provide a conceptual explanation for the various \emph{virtual structures} that are ubiquitous in modern geometry.
These include the virtual fundamental classes \cite{BehrendFantechi,LiTian} underlying Gromov--Witten theory, the Behrend function in Donaldson--Thomas theory and its motivic and cohomological refinements \cite{BehrendDT,KontsevichSoibelmanCoHA,BBBJ,BBDJS}, the analogues of virtual classes underlying the Donaldson--Thomas theory of Calabi--Yau fourfolds \cite{BorisovJoyce}, and the virtual cycles underlying invariants of gauged linear sigma models \cite{FaveroKim,PolishchukVaintrob}.
Although these constructions originate from seemingly disparate considerations, the second-named author has proposed that they are all manifestations of a single geometric mechanism: a derived analogue of \emph{deformation to the normal cone}.

This paper provides the foundations for this program by developing the technique of deformation to the normal cone in the context of derived algebraic geometry.
Applications to the above-mentioned virtual structures will be surveyed in \ssecref{ssec:appl} below.

\subsection{Deformations to the normal cone}

For a closed immersion of schemes $f : X \hook Y$, the \emph{deformation to the normal cone} is a family of closed immersions degenerating $f$ to the zero section of its normal cone.
In this paper we introduce the following derived version:

\begin{thmX}\label{thm:intro/def}
  Let $f : X \to Y$ be a morphism of derived stacks.
  Then we have:
  \begin{thmlist}
    \item
    There exists a canonical diagram of derived stacks
    \begin{equation*}
      \begin{tikzcd}
        X \ar{r}{0}\ar{d}{0}
        & X \times \A^1 \ar[leftarrow]{r}\ar{d}{Df}
        & X \times \Gm \ar{d}{f\times \id}
        \\
        \Nl_{X/Y} \ar{r}\ar{d}
        & \Dl_{X/Y} \ar[leftarrow]{r}\ar{d}
        & Y \times \Gm \ar[equals]{d}
        \\
        Y \ar{r}{0}
        & Y \times \A^1 \ar[leftarrow]{r}
        & Y \times \Gm
      \end{tikzcd}
    \end{equation*}
    where each square is cartesian.
    The morphism $0 : X \to \Nl_{X/Y}$ is the zero section of the derived normal bundle, and $\Nl_{X/Y} \to Y$ factors through the projection $\Nl_{X/Y} \to X$.

    \item
    There is a canonical $\Gm$-action on $\Dl_{X/Y}$.
    Every morphism in the diagram is $\Gm$-equivariant, with respect to the scaling by weight $-1$ actions on $\A^1$ and $\Gm$, the scaling by weight $1$ action on $\Nl_{X/Y}$, and the trivial actions on $X$ and $Y$.

    \item
    Let $\cN_{X/Y} := [\Nl_{X/Y}/\Gm]$ and $\cD_{X/Y} := [\Dl_{X/Y}/\Gm]$ denote the quotient stacks.
    The commutative square
    \begin{equation*}
      \begin{tikzcd}
        \cN_{X/Y} \ar{r}\ar{d}
        & \cD_{X/Y} \ar{d}
        \\
        X \ar{r}{f}
        & Y
      \end{tikzcd}
    \end{equation*}
    is the universal virtual Cartier divisor over $f : X \to Y$.

    \item\label{item:intro/def/alg}
    If $X$ and $Y$ are Artin and $f$ is locally of finite type, then $\Dl_{X/Y}$ is Artin.
  \end{thmlist}
\end{thmX}

We call $\Dl_{X/Y}$ the \emph{deformation to the (derived) normal bundle}, or simply the \emph{normal deformation} for short.
We will see that $\Dl_{X/Y}$ is nothing but the \emph{derived Weil restriction} of the morphism $X \to Y$ along $0 : Y \to Y \times \A^1$.
By definition, the functor of derived Weil restriction
\begin{equation*}
  0_* : \dStk_{/Y} \to \dStk_{/Y \times \A^1}
\end{equation*}
is right adjoint to the functor
\begin{equation*}
  0^* : \dStk_{/Y\times\A^1} \to \dStk_{/Y},
  \qquad (S \to \A^1) \mapsto S \fibprod_{\A^1} \{0\}.
\end{equation*}

In order to be able to use this construction in practice, one needs the \emph{algebraicity} asserted in \thmref{thm:intro/def}\itemref{item:intro/def/alg}.
The proof of algebraicity constitutes the technical heart of this paper; see \thmref{thm:Dart}.\footnote{%
  We will see more precisely that if $X$ and $Y$ are $n$-Artin, then the normal deformation $\Dl_{X/Y}$ is $(n+1)$-Artin.
  For example, even if $X$ and $Y$ are schemes, $\Dl_{X/Y}$ is typically a $1$-Artin stack (unless $f : X \to Y$ is a closed immersion).
}
Let us also note that in \emph{formal} derived algebraic geometry over characteristic zero, a variant of the normal deformation was constructed by Gaitsgory and Rozenblyum.
While their construction is comparatively rather involved, one can show that it is just the formal neighbourhood of $Df : X \times \A^1 \to \Dl_{X/Y}$ \sssecref{sssec:GR}; in particular, by \thmref{thm:intro/def}\itemref{item:intro/def/alg} it is a formal Artin stack (i.e., an ind-Artin stack).

When $f : X \hook Y$ is a closed immersion, the deformation $\Dl_{X/Y}$ can alternatively be described using the derived blow-ups of \cite{blowups,Hekking}: we have
\[\Dl_{X/Y} = \Bl_{X\times\{0\}/Y\times\A^1} \setminus \Bl_{X\times \{0\}/Y\times \{0\}}\]
in that case (see \cite[Thm.~4.1.13]{blowups}, \cite[\S 7.6]{Hekking}).
It follows that for a regular closed immersion of classical schemes, $\Dl_{X/Y}$ is the deformation to the normal cone of Verdier \cite[\S 2]{Verdier}.
So, seen through the lens of derived algebraic geometry, Verdier's construction is just a right adjoint to the derived zero-fibre functor.

For a non-regular closed immersion $X \hook Y$, the classical deformation to the normal cone can be recovered from (the classical truncation of) $\Dl_{X/Y}$ as follows: it is the schematic closure of the open subscheme $Y \times \bG_m$.
In fact, for any morphism of Artin stacks $X \to Y$, locally of finite type, the same construction gives a degeneration to the (relative) \emph{intrinsic normal cone} of Behrend--Fantechi \cite{BehrendFantechi}; see \ssecref{ssec:defcl}.

\subsection{Blow-ups}

Let $X \hook Y$ be a closed immersion.
While the normal deformation $\Dl_{X/Y}$ in that case can be described using derived blow-ups, we will see conversely that the derived blow-up $\Bl_{X/Y}$ sits inside $\cD_{X/Y} = [\Dl_{X/Y}/\Gm]$ as an open substack.
This leads to the following functor-of-points description of $\Bl_{X/Y}$, generalizing \cite[\S 4.1]{blowups} to the non-quasi-smooth case:

\begin{corX}\label{cor:intro/bl}
  Let $i : X \hook Y$ be a closed immersion of derived stacks.
  Then the commutative square
  \begin{equation*}
    \begin{tikzcd}
      \P_X(\Nl_{X/Y}) \ar{r}\ar{d}
      & \Bl_{X/Y} \ar{d}
      \\
      X \ar{r}{i}
      & Y,
    \end{tikzcd}
  \end{equation*}
  where $\P_X(\Nl_{X/Y})$ is the projectivized derived normal bundle, is the universal \emph{excessive} virtual Cartier divisor over $i$
  (\ssecref{ssec:blups}).
\end{corX}

That is, for every derived scheme $S$ over $Y$, the \inftyGrpd of morphisms $S \to \Bl_{X/Y}$ over $Y$ is equivalent to the \inftyGrpd of \emph{excessive} virtual Cartier divisors over $i$, i.e., commutative squares
\begin{equation*}
  \begin{tikzcd}
    D \ar{r}\ar{d}
    & S \ar{d}
    \\
    X \ar{r}{i}
    & Y
  \end{tikzcd}
\end{equation*}
where
\begin{enumerate}
  \item $D \to S$ is a virtual Cartier divisor.
  \item The square is cartesian on classical truncations.
  \item The induced morphism of normal bundles $\Nl_{D/S} \hook \Nl_{X/Y} \fibprod_X D$ is a closed immersion.
\end{enumerate}

\subsection{Algebraicity of Weil restrictions}

The most nontrivial part of \thmref{thm:intro/def} is the algebraicity result for $\Dl_{X/Y}$, i.e., that it is Artin when $f : X \to Y$ is locally of finite type.
We deduce this from the following general algebraicity result for derived Weil restrictions:

\begin{thmX}\label{thm:intro/weil}
  Suppose given a diagram
  \begin{equation*}
    \begin{tikzcd}
      X \ar{d}
      &
      \\
      S \ar{r}{h}
      & T
    \end{tikzcd}
  \end{equation*}
  of derived Artin stacks.
  Then the derived Weil restriction $h_*(X \to S)$ is Artin in either of the following cases:
  \begin{thmlist}
    \item\label{item:intro/weil/hfp}
    The morphism $h$ is finite, of finite Tor-amplitude, and almost of finite presentation\footnote{%
      Under the assumption that $h$ is finite, it is almost of finite presentation if and only if $h_*\sO_S$ is a pseudo-coherent complex (see \cite[Cor.~5.2.2.2]{LurieSAG}).
      This condition is superfluous when $S$ and $T$ are locally noetherian (see e.g. \cite[Rem.~4.2.0.4]{LurieSAG}).
    }
    and $X$ is locally homotopically of finite presentation over $S$ (\thmref{thm:alg}).

    \item\label{item:intro/weil/vcd}
    The morphism $h$ is a virtual Cartier divisor, and $X$ is locally of finite type over $S$ (\thmref{thm:algvcd}).
  \end{thmlist}
\end{thmX}

Recall that the classical Weil restriction along a morphism of schemes $h : S \to T$ preserves algebraicity when $h$ is finite and flat.
In that case, the flatness implies that the classical Weil restriction agrees with the derived version.
\thmref{thm:intro/weil} shows that the \emph{derived} Weil restriction is still algebraic when we relax the flatness assumption.
In fact, we will show that if $h$ is of Tor-amplitude $\le d$ and $X$ is $n$-representable over $S$, then $h_*(X)$ is $(n+d)$-representable over $T$.

As in classical algebraic geometry, there is a close relationship between Weil restrictions and mapping stacks (\ssecref{ssec:umaps}).
\thmref{thm:intro/weil} implies (and is equivalent to) the following algebraicity result for derived mapping stacks:

\begin{corX}\label{cor:intro/map}
  Let $X \to S$ and $Y \to S$ be morphisms of derived Artin stacks.
  Then the derived mapping stack $\uMaps_S(X,Y)$ is Artin in either of the following cases:
  \begin{thmlist}
    \item\label{item:intro/map/hfp}
    The morphism $X \to S$ is almost of finite presentation, finite, and of finite Tor-amplitude, and $Y \to S$ is locally homotopically of finite presentation.

    \item\label{item:intro/map/vcd}
    The morphism $X \to S$ is a virtual Cartier divisor, and $Y \to S$ is locally of finite type.
  \end{thmlist}
\end{corX}

See Corollaries~\ref{cor:algmap} and \ref{cor:algmapvcd}.
It would be interesting to relax the condition that $X \to S$ is finite.
We expect that properness, or even formal properness in the sense of Halpern-Leistner--Preygel, is sufficient; see \conjref{conj:map} below.

\subsection{Applications}
\label{ssec:appl}

We now return to our motivating discussion on virtual structures of derived stacks.
As we will see, a general approach to extracting virtual structures is by considering incarnations of the normal deformation in the form of \emph{specialization maps} in various contexts.
We will also see some applications of derived blow-ups.

\subsubsection{Virtual fundamental classes and virtual pull-backs}
\label{sssec:vfc}

In \cite{virtual}, the normal deformation is used to give a conceptual construction of the virtual fundamental class of Kontsevich \cite{KontsevichEnumeration,BehrendFantechi,LiTian}.
Unlike the classical approaches, this construction is robust enough to apply in a variety of contexts, yielding Gromov--Witten-type invariants e.g. in complex analytic geometry \cite{PardonMNOP}, non-archimedean analytic geometry \cite{PortaYu} and symplectic geometry \cite{PardonRepr,Steffens}.

Let $f : X \to Y$ be a morphism between derived Artin stacks, locally of finite type over some base.
The morphism $Df : X \times \A^1 \to \Dl_{X/Y}$, which degenerates $f$ to the zero section $0 : X \to \Nl_{X/Y}$, manifests at the level of Borel--Moore homology\footnote{%
  For simplicity, assume $X$ and $Y$ are locally of finite type over $\bC$ and work in singular Borel--Moore homology with rational coefficients.
  If $X$ and $Y$ are locally of finite type over a field $k$ of characteristic coprime to $\ell$, we may alternatively work in $\ell$-adic étale cohomology.
  If $X$ and $Y$ are locally of finite type over an arbitrary field $k$, we may also work with motivic Borel--Moore homology ($\approx$ higher Chow groups).
} as a \emph{specialization} map
\begin{equation}\label{eq:sp}
  \sp_{X/Y} : \HBM_*(Y; \bQ) \to \HBM_*(\Nl_{X/Y}; \bQ).
\end{equation}
See \cite[Constr.~3.1]{virtual} and \cite[\S 2.1]{virloc}.
When $f$ is \emph{quasi-smooth}, i.e., the relative cotangent complex $\sL_{X/Y}$ is perfect of Tor-amplitude $\le 1$\footnote{
  If $X$ and $Y$ are Deligne--Mumford, or more generally $f$ is representable by DM stacks, then $\sL_{X/Y}$ is necessarily connective and quasi-smoothness is equivalent to $\sL_{X/Y}$ being of Tor-amplitude $[0,1]$.
  We refer to \cite[\S 2]{blowups} for some background on quasi-smoothness.
}, the normal bundle $\pi : \Nl_{X/Y} \to X$ is a vector bundle stack and we have the Thom isomorphism
\begin{equation}
  \pi^! : \HBM_{*}(X; \bQ) \to \HBM_{*-2\vd}(\Nl_{X/Y}; \bQ(-\vd)),
\end{equation}
where $\vd = \rk(\sL_{X/Y})$ is the relative virtual dimension of $f$.
The \emph{virtual pull-back} along $f$ is the canonical map
\begin{equation}
	f_\vir^! : \HBM_*(Y; \bQ)
	\xrightarrow{\sp_{X/Y}} \HBM_*(\Nl_{X/Y}; \bQ)
	\simeq \HBM_{*+2\vd}(X; \bQ(\vd)),
\end{equation}
see \cite[Constr.~3.3]{virtual}.
If $X$ is a quasi-smooth derived Artin stack over $Y=\pt$, the \emph{virtual fundamental class} of $X$ is the class\footnote{%
  Since Borel--Moore homology is insensitive to derived structures, this may be regarded as a class on the classical truncation $X_\cl$.
}
\begin{equation}
  [X]^\vir := f_\vir^!(1) \in \HBM_{2\vd}(X; \bQ(-\vd)).
\end{equation}

Our algebraicity result for the normal deformation is required here, since it is only for Artin stacks that one has the full six operations formalism.
In particular, this work fills in a gap in \cite{virtual}, where only a sketch of the proof of algebraicity (in the quasi-smooth case) was given.

\subsubsection{Virtual localization formula}

The normal deformation also allows a conceptual proof (and generalization) of the virtual localization formula of \cite{GraberPandharipande}, explained in \cite{virloc}.
Let $X$ be a quasi-smooth derived Deligne--Mumford stack with an action of a split torus $T$.
The fixed locus $X^T$ is then also quasi-smooth, but the inclusion $i : X^T \hook X$ is typically not; in particular, there is no virtual pull-back $i_\vir^!$.
Nevertheless, one can still use the derived specialization map $\sp_{X^T/X}$ \eqref{eq:sp} to construct a virtual pull-back $i_\vir^!$ in localized $T$-equivariant cohomology.
This turns out to be inverse to the push-forward $i_*$ up to multiplication by the Euler class of the normal bundle, and this immediately gives rise to the virtual localization formula computing integrals on $X$ in terms of integrals on $X^T$ (when $X$ is proper).

\subsubsection{Coherent duality}

Let $f : X \to Y$ be a quasi-smooth morphism of derived Artin stacks.
In \cite{cohdual} the normal deformation is used to construct a canonical isomorphism of quasi-coherent complexes
\begin{equation}
  f^!(\sO_Y) \simeq \det(\sL_{X/Y})[\vd]
\end{equation}
between the coherent dualizing complex of $f$ and the graded determinant of the relative cotangent complex.
This is a well-known folklore conjecture.
In the case where $f$ is a quasi-smooth closed immersion or a smooth separated representable morphism, it can be proven using the \emph{formal} counterpart to the normal deformation constructed by Gaitsgory and Rozenblyum (see \cite[Chap.~9, 7.3.2]{GaitsgoryRozenblyumII}, \cite[App.~B]{HalpernLeistnerDEquiv}).

Our construction is essentially a coherent analogue of \sssecref{sssec:vfc}.
We use the $\Gm$-equivariance of the degeneration $Df : X\times\A^1 \to \Dl_{X/Y}$, which yields a $\Gm$-equivariant invertible sheaf $(Df)^!(\sO_{\Dl_{X/Y}})$ on $X\times\A^1$ degenerating $f^!(\sO_Y)$ to $0^!(\sO_{\Nl_{X/Y}})$, and we then compute $0^!(\sO_{\Nl_{X/Y}}) \simeq \det(\sL_{X/Y})[\vd]$.

\subsubsection{Derived microlocalization}

Let $f : X \to Y$ be a morphism of locally of finite type derived Artin stacks over $\bC$.
On \inftyCats of sheaves of $\bQ$-vector spaces, the normal deformation manifests as a \emph{specialization} functor
\begin{equation}\label{eq:sp2}
  \sp_{X/Y} : \Shv(Y) \to \Shv^{\Gm}(\Nl_{X/Y}),
\end{equation}
categorifying \eqref{eq:sp}.
This is defined by nearby cycles along the $\Gm$-equivariant morphism $t : \Dl_{X/Y} \to \A^1$ (see \cite{dimredcoha,Schefers}) and lands in $\Gm$-equivariant sheaves on the normal bundle.
Applying a derived Fourier--Sato transform yields the \emph{microlocalization} functor of \cite{dimredcoha},
\begin{equation}\label{eq:mu}
  \mu_{X/Y} : \Shv(Y) \to \Shv^{\Gm}(\Nl^*_{X/Y}),
\end{equation}
landing in $\Gm$-equivariant sheaves on the conormal bundle $\Nl^*_{X/Y}$.
In the case of closed immersions of smooth schemes, these constructions were developed by Kashiwara--Schapira in \cite{KashiwaraSchapira}.

Microlocalization is shown in \cite{dimredcoha} to recover the cohomological Donaldson--Thomas theory of local surfaces (discussed in \sssecref{sssec:intro/CohDT} below), and moreover to upgrade 2d cohomological Hall algebras of surfaces \cite{KapranovVasserot} to 3d cohomological Hall algebras of local surfaces.

\subsubsection{\texorpdfstring{$(-1)$}{(-1)}--shifted microlocal sheaf theory}
\label{sssec:intro/CohDT}

Let $X$ be a derived Artin stack locally of finite type over $\bC$.
Suppose that $X$ admits a \emph{$(-1)$-shifted symplectic structure} in the sense of \cite{PTVV}, inducing in particular an isomorphism $\Tl_X \simeq \Tl^*_X[-1]$ of derived vector bundles over $X$.
One introduces in this situation the perverse sheaf of vanishing cycles $\phi_X$ on $X$, assuming the existence of certain orientation data for $X$: by the $(-1)$-shifted Darboux theorem, $X$ admits a local presentation as the derived critical locus of a regular function $f$ on some ambient smooth stack, and locally $\phi_X$ is defined as (a twist of) the vanishing cycles sheaf $\phi_f$ (see \cite{BehrendDT,KontsevichSoibelmanCoHA,BBBJ, BBDJS,KiemLiPerverse}).
For example, if $X$ is the derived moduli stack of compactly supported coherent sheaves on a Calabi--Yau threefold $Y$, then $\phi_X$ is a categorification of the Donaldson--Thomas invariants of $Y$.
We refer to \cite{Kinjo} for an introduction to some aspects of this subject.

A fundamental example is the $(-1)$-shifted cotangent bundle $X = \Tl^*_M[-1]$ for a locally hfp derived Artin stack $M$, equipped with its canonical $(-1)$-shifted symplectic structure and orientation.
In \cite{pervpull2}, normal deformation is used to construct a \emph{dimensional reduction} isomorphism
\begin{equation}\label{eq:dimred}
  \pi_!(\phi_{\Tl^*_M[-1]}) \simeq \bQ_M[\vd(M)],
\end{equation}
through which $\phi_M$ may be regarded as a $(-1)$-shifted microlocalization of the constant sheaf $\bQ_M$.

In forthcoming work of the second-named author with T.~Kinjo, H.~Park, and P.~Safronov, the normal deformation is also used to construct \emph{Lagrangian virtual classes} in cohomological DT theory, conjectured by D.~Joyce.
Let $i : L \to X$ be a $(-1)$-shifted Lagrangian in the sense of \cite{PTVV}, inducing in particular an isomorphism $\Nl_{L/X} \simeq \Tl^*_L[-1]$ of derived vector bundles over $L$.
When $i$ is appropriately oriented, there exists a canonical morphism of the form
\begin{equation}\label{eq:Joyce}
  i_!(\bQ_L[\vd(L)]) \to \phi_X
\end{equation}
in $\Shv(X)$.
This construction generalizes Konstevich's virtual classes \sssecref{sssec:vfc}, the virtual classes of \cite{BorisovJoyce}, the virtual cycles underlying invariants of gauged linear sigma models as in \cite{FaveroKim,PolishchukVaintrob}, and the virtual pull-backs used in the construction of 3d cohomological Hall algebras in \cite{dimredcoha,KinjoParkSafronov}.
Conceptually, one may regard \eqref{eq:Joyce} as a $(-1)$-shifted microlocalization of the virtual pull-back for quasi-smooth morphisms \sssecref{sssec:vfc}; the latter can be regarded as a morphism $f_!(\bQ_X[2\vd(X/Y)]) \to \bQ_Y$ in $\Shv(Y)$ for $f : X \to Y$ quasi-smooth.

The construction of \eqref{eq:Joyce} is formally analogous to that of \sssecref{sssec:vfc}.
According to \cite{CalaqueSafronov}, the normal deformation $\Dl_{L/X}\to\A^1$ inherits a relative exact $(-1)$-shifted symplectic structure, and $Di : L\times\A^1 \to \Dl_{L/X}$ is a relative Lagrangian which is a degeneration of the Lagrangian $i : L \to X$ to the Lagrangian $0 : L \to \Nl_{L/X} \simeq \Tl^*_L[-1]$.
With some work, building on \cite{pervpull,pervpull2}, one extracts from this geometry a cospecialization morphism $\phi_{X} \to i_!\pi_!(\phi_{\Tl^*_L[-1]})$, where the target is identified with $\bQ_L[\vd(L)]$ by dimensional reduction \eqref{eq:dimred}.

The second-named author plans to give a more conceptual description of the vanishing cycles sheaf $\phi_X$ by using the normal deformation to define an analogue of specialization \eqref{eq:sp}, \eqref{eq:sp2} at the level of perverse microsheaves on $0$-shifted symplectic stacks (see \cite[Conj.~B]{pervpull}).

\subsubsection{Reduction of stabilizers for derived Artin stacks}

Given an Artin stack $\msf{X}$ satisfying mild technical hypotheses, there are two natural approaches to constructing a canonical \emph{stabilizer reduction} for $\msf{X}$, i.e., a birational morphism $\widetilde{\msf{X}} \to \msf{X}$ such that the maximal stabilizer dimension of the points of $\widetilde{\msf{X}}$ is strictly less than those of $\msf{X}$.
\begin{enumerate}
  \item\label{item:xmaxER}
  The construction of \cite{EdidinRydh} proceeds by blowing up along the locus of points of maximal stabilizer dimension, discarding the unstable locus, and iterating repeatedly.

  \item\label{item:xmaxKLS}
  The construction of \cite{KiemLiSavvas,SavvasDT} proceeds by taking a so-called \emph{intrinsic blow-up} along the locus of points of maximal stabilizer dimension, discarding the unstable locus, and iterating repeatedly.
\end{enumerate}

The constructions agree when $\msf{X}$ is smooth.
In general, it is shown in \cite{HekkingRydhSavvas} that the discrepancy can be understood as follows.
If $X$ is a \emph{derived} Artin stack, one may imitate construction~\itemref{item:xmaxER} using \emph{derived} blow-ups.
This yields a derived stack $\widetilde{X}$ whose classical truncation $\widetilde{X}_\cl$ is identified with construction~\itemref{item:xmaxKLS}.
The relation between \itemref{item:xmaxER} and \itemref{item:xmaxKLS} is then entirely explained by \propref{prop:compbl}, which explicitly identifies the classical blow-up as a closed substack of the derived one.

\sssec{Further applications of derived blow-ups}

In the quasi-smooth case, derived blow-ups have been applied to great effect in the context of algebraic K-theory \cite{KerzStrunkTamme,kblow,milnor} and algebraic bordism theory \cite{AnnalaChern,AnnalaSpivak}, for example.
The non-quasi-smooth case we consider here has already seen surprising applications in recent work of Yu~Zhao \cite{ZhaoNest,ZhaoVan}.

\subsection{Related works}

\subsubsection{Algebraicity results of Halpern-Leistner--Preygel}

Our algebraicity results \thmref{thm:intro/weil} and \corref{cor:intro/map} have some overlap with work of D.~Halpern-Leistner and A.~Preygel \cite{HalpernLeistnerPreygel}.
Namely, they prove in \cite[Thm.~5.1.1]{HalpernLeistnerPreygel} a result similar to \corref{cor:intro/map} but under different assumptions.
Instead of imposing finiteness on $X \to S$, they impose the much more general notion of formal properness (which includes non-representable morphisms such as $B\bG_{m,S} \to S$).
However, in exchange, their proof only applies to $1$-Artin stacks (whereas ours works for \emph{higher} Artin stacks).
They also need to impose certain technical hypotheses: $Y$ has affine diagonal, and $S$ is locally noetherian with a smooth cover by an affine $\Spec(R)$ where $\pi_0(R)$ is a G-ring admitting a dualizing complex (see \cite[Tags~\spref{07GH}, \spref{0A7A}]{Stacks}).

The proofs are also quite different.
Halpern-Leistner and Preygel appeal to the Artin--Lurie representability criteria, while our argument is much more low-tech, using local structural results to construct explicit atlases.

We expect the following statement, generalizing \corref{cor:intro/map} in one direction and \cite[Thm.~5.1.1]{HalpernLeistnerPreygel} in another:

\begin{conjX}\label{conj:map}
  Let $S$ be as above and let $X$ and $Y$ be derived Artin stacks over $S$.
  If $X \to S$ is formally proper and of finite Tor-amplitude, and $Y \to S$ is locally homotopically of finite presentation, then $\uMaps_S(X, Y)$ is derived Artin.
\end{conjX}

\subsubsection{Normal deformations}
\label{sssec:GR}

The first incarnation of deformation to the normal cone in derived algebraic geometry was in work of Ciocan-Fontantine--Kapranov (see \cite[\S 1]{CFKdg}), in the more restrictive context of dg-schemes\footnote{%
  The notion of dg-scheme is defined over a field $k$ of characteristic zero.
  Any dg-scheme over $k$ gives rise to a derived $k$-scheme equipped with a closed embedding into an ambient smooth $k$-scheme.
  Morphisms between dg-schemes are more rigid than morphisms between their associated derived schemes.
  In other words, there is a functor from dg-schemes over $k$ to derived $k$-schemes, which is neither essentially surjective nor fully faithful.
}.
They applied their deformation to the normal cone in the quasi-smooth case to construct virtual fundamental classes and prove virtual Grothendieck--Riemann--Roch theorems, in particular comparing the virtual cycles of \cite{KontsevichEnumeration} and \cite{BehrendFantechi}.

For quasi-smooth closed immersions of derived stacks, the second- and third-named authors defined normal deformations using derived blow-ups (see \cite[Thm.~4.1.13]{blowups}).
An extension to the case of general quasi-smooth morphisms was sketched in \cite[\S 1.4]{virtual}, using derived Weil restrictions, and applied to develop a theory of virtual fundamental classes for derived Artin stacks and extend the Kontsevich formula to that generality (see \cite[(0.4)]{virtual}).
Our work here in particular provides the necessary details on normal deformations and their algebraicity.

In the first-named author's thesis \cite{Hekking}, he extended derived blow-ups to the non-quasi-smooth case and used this to define normal deformations for arbitrary closed immersions, also using derived Weil restrictions.
We will generalize this here to the case of arbitrary morphisms of derived stacks.
We will also see that, even though concepts like ideals and adic filtrations are \emph{a priori} subtle in derived algebraic geometry, one can now reverse the logic and define adic filtrations and infinitesimal neighbourhoods using the normal deformation (see \ssecref{ssec:inf}).
In forthcoming work of the first-named author with Z.~Gardner, this approach will be linked back to the theory of Smith ideals. It will be shown that all known approaches to adic filtrations in derived geometry agree, and likewise for all known notions of completeness under suitable finiteness assumptions.

Another type of deformation to the normal cone in derived algebraic geometry was constructed previously by D.~Gaitsgory and N.~Rozenblyum in \cite{GaitsgoryRozenblyumII}.
We adopt the terminology of \emph{op.\ cit}.\ and in particular work in characteristic zero.\footnote{%
  In the case of schemes, T.~Moulinos has extended the Gaitsgory--Rozenblyum construction to general base fields in \cite[\S 5]{Moulinos} (cf. \sssecref{sssec:general}).
}
Given a nil-isomorphism $\mfr{f} : \mfr{X} \to \mfr{Y}$ of derived prestacks which are laft and admit deformation theories, they construct in \cite[Chap.~9, \S 2]{GaitsgoryRozenblyumII} a derived stack $\mfr{Y}_{\mrm{scaled}}$.
This may be regarded as a formal analogue of our normal deformation, in the following sense: when $\mfr{X} = X$ is a derived Artin stack and $\mfr{f} : \mfr{X} \to \mfr{Y}$ is the formal completion $X \to Y^\wedge_X$ of a derived Artin stack $Y$ along a morphism $f : X \to Y$, then $\mfr{Y}_{\mrm{scaled}}$ is the formal completion of $\Dl_{X/Y}$ along the canonical morphism $D f : X \times \A^1 \to \Dl_{X/Y}$.
This is proven in forthcoming work of Christopher Brav and Nick Rozenblyum, who also give an alternative construction of $\mfr{Y}_{\mrm{scaled}}$ via Weil restriction.
See also \cite[Rem.~5.9]{CalaqueSafronov} for the relationship between our construction and the groupoid used to define $\mfr{Y}_{\mrm{scaled}}$ in \cite[Chap.~9, Eq.~(2.2)]{GaitsgoryRozenblyumII}.
For $Y=\pt$ it is well-known that the latter agrees with Simpson's construction $X_{\mrm{Hod}}$ (see \cite[\S\S 4--5]{Simpson}, \cite[\S 2.3]{BhattLect}).
See also \remref{rem:GadR} for some discussion of the ring stack $\bG_a^{\mrm{dR},+}$ of \cite[Constr.~2.3.4]{BhattLect}.
We also note that our construction of infinitesimal neighbourhoods using normal deformations in \ssecref{ssec:inf} is similar to the construction by Gaitsgory and Rozenblyum in \cite[Chap.~9, Sec.~5]{GaitsgoryRozenblyumII}).

Recently, H.~Park and Calaque--Safronov have studied normal deformations in the context of shifted symplectic geometry \cite{PTVV}:
\begin{enumerate}
  \item 
  If $\pi : M \to B$ is an $n$-shifted symplectic fibration equipped with a \emph{locked} structure, then the morphism $D\pi : M\times\A^1 \to \Dl_{M/B}$ is again a locked $n$-shifted symplectic fibration.
  See \cite[Cor.~5.1.2]{ParkLocked}.
  
  \item\label{item:Noseren}
  If $M$ is an $n$-shifted symplectic derived stack and $i : L \to M$ is a morphism, then the morphism $\Dl_{L/M} \to \A^1$ is an $n$-shifted symplectic fibration.
  If $i$ is furthermore Lagrangian, then also the morphism $Di : L\times\A^1 \to \Dl_{L/M}$ is Lagrangian relative to $\A^1$.
  See \cite[Cor.~5.13]{CalaqueSafronov}.
\end{enumerate}
The notion of \emph{locked form}, defined by Park using regular functions on the normal deformation, is a stronger version of the notion of \emph{closed form} introduced in \cite{PTVV} which is required in the formulation of the relative shifted Darboux theorem (see \cite[Thm.~B]{ParkLocked}).

\subsubsection{The HKR filtration}
\label{sssec:HKR}

Let $A$ be an animated commutative ring and $B$ an animated $A$-algebra.
Write $S = \Spec(A)$, $X = \Spec(B)$, and denote by
\[ L_S(X) := \uMaps_S(S^1,X) \simeq X \fibprod_{X \fibprod_S X} X \]
the relative derived loop space.
The normal deformation along the ``constant loops'' map $X \to L_S(X)$ is a derived stack $\cD := \cD_{X/L_S(X)}$ over $[\A^1/\bG_m]$ which degenerates $L_S(X)$ to the $(-1)$-shifted tangent bundle $\Tl_{X/S}[-1]$.
Pushing forward the structure sheaf $\sO_\cD$ along $v : \cD \to S \times [\A^1/\bG_m]$ yields a filtered quasi-coherent complex on $S$, i.e., a filtered animated $A$-algebra, whose underlying animated $A$-algebra is the Hochschild homology $\on{HH}(B/A)$ (= global sections of $L_S(X)$) and whose associated graded is $\Sym^*_B(\sL_{B/A}[1]) \simeq \bigoplus_{n\ge 0} \Lambda_B^n \sL_{B/A}[n]$ (= global sections of $\Tl_{X/S}[-1]$).
In particular, the normal deformation $\cD$ may be regarded as a geometric incarnation of the HKR filtration on Hochschild homology, which makes sense even for an arbitrary morphism $X \to S$ of derived schemes or Artin stacks.
The precise relationship with the HKR filtration will be investigated in forthcoming work of the first-named author.
See also \cite{Moulinos}.

\subsubsection{Generalizations}
\label{sssec:general}

The fact that the normal deformation can be realized via Weil restriction along the zero section in $\A^1$ implies that it makes sense in any reasonable geometry.
Following this observation, Porta--Yu and P.~Steffens have studied analogues of the normal deformation in the context of derived $k$-analytic geometry (where $k$ is the field of complex numbers or a non-archimedean field) and derived $C^\infty$-geometry (see \cite[Thm.~4.3]{PortaYu} and \cite[Ex.~4.2.4.3]{Steffens}).
P.~Scholze has also studied this construction in the context of liquid mathematics (with an intended application to virtual fundamental classes of moduli spaces of pseudo-holomorphic curves in symplectic geometry), see \cite{ScholzeLect24}.
As another example, let us note that the normal deformation also makes sense in the context of spectral algebraic geometry, which is built out of \emph{connective $\Einfty$-ring spectra} rather than animated commutative rings (see \cite{LurieSAG}), \footnote{%
  In the context of spectral algebraic geometry, there are two versions of $\A^1$ (and $\Gm$): the intrinsic or ``smooth'' affine line and the ``flat'' affine line, respectively (see e.g. \cite[\S 1.1]{localization}).
  The normal deformation formed using the smooth version defines a degeneration whose special fibre is the $1$-shifted relative tangent bundle, while it is the flat version for which quasi-coherent sheaves on $\A^1/\bG_m$ can be interpreted in terms of filtrations (see \cite[Thm.~4.8]{Moulinos}).
  The normal deformation is thus most powerful in the setting of derived algebraic geometry, where the smooth affine line has been collapsed into the flat (= classical) one.
}
The first-named author has studied normal deformations and derived blow-ups for stacks internal to arbitrary derived algebraic contexts, in joint work with O.~Ben-Bassat (see \cite{BenBassatHekking}).
He intends to pursue a generalization of the construction of virtual fundamental classes described in \sssecref{sssec:vfc} to the abstract setup; this requires an algebraicity result for the normal deformation as in \thmref{thm:Dart}.
In forthcoming joint work with O.~Ben-Bassat and J.~Kelly, the ideas discussed in \sssecref{sssec:HKR} will be used to give a uniform treatment of HKR results in generalized geometries.

\subsection{Acknowledgments}

AAK would like to thank Tasuki Kinjo, Hyeonjun Park, and Pavel Safronov for comments, questions, and suggestions related to this paper.
He also thanks Bhargav Bhatt for suggesting the description via ring stacks in \ssecref{ssec:ringstk}, Chris Brav and Nick Rozenblyum for discussions about their forthcoming work mentioned in \sssecref{sssec:GR}, and Tasos Moulinos for a discussion about his related work on formal deformation to the normal cone and the HKR filtration.

AAK acknowledges support from the grants AS-CDA-112-M01 (Academia Sinica), NSTC 110-2115-M-001-016-MY3, and NSTC 112-2628-M-001-0062030.
JH acknowledges support from the grant 2021.0287 (Knut and Alice Wallenberg Foundation) and 
the SFB 1085: Higher Invariants (project number 224262486).
JH and DR acknowledge support from the G\"oran Gustafsson Foundation (KTH/UU) and
the G\"oran Gustafsson Foundation for Research in Natural Sciences and Medicine.

%!TEX root = ../weilres.tex

\section{Conventions and notation}

\subsection{Derived stacks}
\label{ssec:convent/dstk}

We refer to \cite[\S 5.1]{LurieDAG}, \cite[\S 5]{ToenSimp}, \cite[Chap.~2]{GaitsgoryRozenblyumI}, or \cite[\S 8]{KhanNCTS} for background on derived algebraic geometry.

A \emph{derived stack} is an étale sheaf of \inftyGrpds on the \inftyCat of derived schemes.
Any derived stack $X$ is determined by its restriction to affines, i.e., by the functor $R \mapsto X(R) := X(\Spec(R))$ on animated commutative rings.
Unless otherwise specified, colimits and limits (in particular fibred products) will always be taken in the \inftyCat $\dStk$ of derived stacks.
We write $\pt := \Spec(\Z)$ for the terminal object.
We write $\dStk_{/X}$ for the \inftyCat of derived stacks over a fixed derived stack $X$.

Our specific conventions on algebraicity follow \cite[\S 8]{KhanNCTS}.
We give a brief summary for the reader's convenience:

A derived stack $X$ is \emph{0-Artin}, or a \emph{derived algebraic space}, if it has schematic $(-1)$-truncated diagonal and there exists a derived scheme $U$ and a morphism $U \twoheadrightarrow X$ which is étale and surjective (i.e., for any derived scheme $V$ over $X$, the fibre $U \fibprod_X V \twoheadrightarrow V$ is étale and surjective).

For $n>0$, a morphism $f : X \to Y$ is \emph{$(n-1)$-representable} if for every derived scheme $V$ and every morphism $V \to Y$, the fibred product $X \fibprod_Y V$ is $(n-1)$-Artin.
An $(n-1)$-representable morphism $f : X \to Y$ is \emph{smooth} or \emph{smooth surjective}, if for every derived scheme $V$ and morphism $V \to Y$, and every derived scheme $U$ and smooth surjection $U \to X \fibprod_Y V$, the composite $U \to X \fibprod_Y V \to V$ has the respective property.\footnote{For $n>1$ this works because $U \to X \fibprod_Y V$ is automatically $(n-2)$-representable. For $n=1$ one defines the notion separately in the expected way.}

A derived stack $X$ is \emph{$n$-Artin} if it has $(n-1)$-representable diagonal and there exists a derived scheme $U$ and a smooth surjective (automatically $(n-1)$-representable) morphism $U \twoheadrightarrow X$.
A derived stack is \emph{Artin} if it is $n$-Artin for some $n$, and a morphism is \emph{eventually representable} if it is $n$-representable for some $n$.

\subsection{Quasi-coherent complexes}
\label{ssec:convent/qcoh}

For background on quasi-coherent, pseudo-coherent, and perfect complexes on derived stacks, see \cite[\S 5.2]{LurieDAG}, \cite[Chap.~3]{GaitsgoryRozenblyumI}, \cite[\S 8]{KhanNCTS}, or \cite[\S 1]{kstack}.

We recall that for $X = \Spec(R)$ affine, $\QCoh(X)$ is equivalent to the derived \inftyCat of $R$-modules.
An $R$-module $M$ is \emph{perfect} if it can be built out of $R$ under finite co/limits and direct summands.
It is \emph{$n$-pseudo-coherent} if there exists a perfect $R$-module $M'$ and a morphism $M' \to M$ with $n$-connective fibre; moreover, $M'$ can be taken to be of Tor-amplitude $\le n$ (see \cite[Cor.~2.7.2.2]{LurieSAG}).
We say that $M$ is \emph{pseudo-coherent} if it is $n$-pseudo-coherent for every integer $n$.\footnote{%
	In \cite{LurieHA}, pseudo-coherent is called \emph{almost perfect}.
	See \cite[Rem.~2.7.0.2, Cor.~2.7.2.2]{LurieSAG} for the comparison.
}
A pseudo-coherent $R$-module is perfect if and only if it is of finite Tor-amplitude (see \cite[Prop.~7.2.4.23]{LurieHA}).

For a derived stack $X$, the stable \inftyCat $\QCoh(X)$ of quasi-coherent complexes on $X$ is the limit
\begin{equation*}
	\QCoh(X) \simeq \lim_{(T,t)} \on{D}(\Gamma(T,\sO_T))
\end{equation*}
over pairs $(T,t : T \to X)$ where $T$ is an affine derived scheme and $\on{D}(\Gamma(T,\sO_T))$ is the unbounded derived $\infty$-category of the animated ring $\Gamma(T,\sO_T)$.
A quasi-coherent complex $\sF$ on $X$ is perfect, resp. pseudo-coherent, if each restriction $\sF|_T = t^*(\sF)$ corresponds to a perfect, resp. pseudo-coherent $\sO_T$-module.
We write $\Perf(X) \sub \QCoh(X)$ for the full subcategory of perfect complexes.

We say $\sF \in \QCoh(X)$ is \emph{$n$-connective}, resp. \emph{locally eventually connective} if each $\sF|_T$ is $n$-connective (resp. eventually connective) as an $\sO_T$-module for every $(T,t)$ as above.
We say it is \emph{eventually connective} if it is $n$-connective for some integer $n$.
We write $\QCoh(X)_{\ge 0} \sub \QCoh(X)$ for the full subcategory of connective (= $0$-connective) quasi-coherent complexes.

The \emph{rank} of a perfect complex is its Euler characteristic.
An \emph{invertible sheaf} on $X$ is a locally free sheaf of rank one, i.e., a perfect complex of Tor-amplitude $[0,0]$ and of rank one.
These are the $\otimes$-invertible objects of $\QCoh(X)_{\ge 0}$ as in \cite[Prop.~2.9.4.2]{LurieSAG}.
(Contrast with \cite[Def.~2.9.5.1]{LurieSAG}.)

\subsection{Finiteness conditions}

We say that a morphism of derived stacks $f : X \to Y$ is \emph{locally of finite type} if the induced morphism on classical truncations $f_\cl : X_\cl \to Y_\cl$ is locally of finite type.
It is \emph{locally homotopically of finite presentation} if for any cofiltered system of affines $\{T_\alpha\}$ over $Y$ with limit $T$, the canonical map
\begin{equation}\label{eq:Ovawuoh}
	\colim_\alpha \Maps_Y(T_\alpha, X) \to \Maps_Y(T, X)
\end{equation}
is invertible.\footnote{%
	In \cite{LurieSAG}, the term \emph{locally of finite presentation} is used instead.
	We prefer not to use this terminology because, when restricted to classical schemes or stacks, this condition is much stronger than the classical notion of finite presentation.
	See \cite[Rem.~8.7.5, Thm.~8.7.6]{KhanNCTS}.
}
Equivalently, for every affine $V$ over $Y$, the base change $X \fibprod_Y V$ preserves filtered colimits when regarded as a functor from animated commutative $\sO_V$-algebras to \inftyGrpds.
For $m\geq 0$, it is \emph{locally of finite presentation to order $m$} if \eqref{eq:Ovawuoh} is invertible for every cofiltered system of affines $\{T_\alpha\}$ over $Y$ where each $\sO_{T_\alpha}$ is $m$-truncated and the transition maps induce monomorphisms on $\pi_m$. In particular,
locally of finite type is equivalent to locally of finite presentation to order $0$.
It is \emph{locally almost of finite presentation} if it is locally of finite presentation to order $m$ for all $m \geq 0$,
or equivalently when every base change $X \fibprod_Y V$ to an affine $V$ preserves filtered colimits when regarded as a functor from $m$-truncated animated commutative rings for every $m$ \cite[\S 17.4.1]{LurieSAG}.
We will usually use the abbreviations \emph{locally hfp}, \emph{locally fp to order $m$}, and \emph{locally afp} for the last three conditions, respectively.

\subsection{Coconnective and almost truncated stacks}
\label{ssec:convent/trunc}

Following \cite[Chap.~2, \S 1.3]{GaitsgoryRozenblyumI}, we define coconnectivity for derived stacks as follows.
An affine derived scheme $X = \Spec(R)$ is \emph{$k$-coconnective} if the animated commutative ring $R$ is $k$-truncated.
A derived stack $X$ is \emph{$k$-coconnective} if it belongs to the full subcategory generated under colimits by $k$-coconnective affines.
It is \emph{eventually coconnective} if it is $k$-coconnective for some $k\ge 0$.

We say that a derived stack $X$ is \emph{almost $n$-truncated}\footnote{%
	We warn the reader that derived stacks are rarely $n$-truncated as objects of the \inftyCat $\dStk$ (in the sense of \cite[\S 5.5.6]{LurieHTT}), even when they are almost $n$-truncated.
} if it sends $k$-truncated animated rings $R$ to $(n+k)$-truncated \inftyGrpds $X(R)$.\footnote{%
	This implies more generally that the \inftyGrpd $\Maps(Y, X)$ is $(n+k)$-truncated for any $k$-coconnective derived stack $Y$.
	Indeed, write $Y$ as a colimit of $k$-coconnective affines and observe that $(n+k)$-truncatedness of \inftyGrpds is stable under limits.
}
We say that a morphism $f : X \to Y$ is \emph{almost $n$-truncated} if for every derived scheme $V$ over $Y$, the fibre $X\fibprod_Y V$ is almost $n$-truncated.

Any $n$-representable morphism is almost $n$-truncated; in particular, any $n$-Artin $X$ is almost $n$-truncated (see \cite[Cor.~4.3.4]{GaitsgoryRozenblyumI}).
Conversely, if $X$ is $m$-Artin for some $m\ge n$, then it is $n$-Artin if and only if it is almost $n$-truncated.\footnote{%
	One reduces inductively to the fact that a derived $1$-Artin stack is $0$-Artin (a derived algebraic space) if and only if its classical truncation is, hence if and only if it is almost $0$-truncated (see e.g. \cite[Tag~\spref{04SZ}]{Stacks}).
}

\subsection{Cotangent complexes}
\label{ssec:convent/cot}

A \emph{cotangent complex} for a morphism $f : X \to Y$ of derived stacks is a locally eventually connective quasi-coherent complex $\sL_{X/Y} \in \QCoh(X)$, determined uniquely by the following universal property.

For any morphism $u : U \to X$ and connective quasi-coherent complex $\sF \in \QCoh(U)_{\ge 0}$, let $\on{Der}_U(X/Y; \sF)$ denote the \inftyGrpd of derivations of $f$ at $U$ with values in $\sF$, i.e., dashed arrows making the solid arrow diagram commute:
\begin{equation*}
	\begin{tikzcd}
		U \ar[d] \ar[r, "u"] & X \ar[d, "f"] \\
		U[\sF] \ar[r] \ar[ru, dashed] & Y
	\end{tikzcd}
\end{equation*}
where $U \hook U[\sF] := \uSpec_U(\sO_U \oplus \sF)$ denotes the trivial square-zero extension.
Then there are isomorphisms
\begin{equation}\label{eq:derivations}
	\Maps_{\QCoh(U)}(u^*\sL_{X/Y}, \sF)
	\simeq \on{Der}_U(X/Y; \sF),
\end{equation}
functorial in $u : U \to X$ and $\sF \in \QCoh(U)_{\ge 0}$.

Equivalently, consider the cartesian fibration
\[ \QCoh_{/Y} \to \dStk_{/Y} \]
corresponding to the presheaf $\QCoh(-) : \dStk_{/Y}^\op \to \Catoo$.
Objects of $\QCoh_{/Y}$ are pairs $(U, \sF)$ where $U \in \dStk_{/Y}$ and $\sF \in \QCoh(U)$, and morphisms $(U',\sF') \to (U,\sF)$ are pairs
\[
\big(u : U' \to U~\text{in}~\dStk,
\quad \phi : u^*\sF \to \sF' ~\text{in}~\QCoh(U') \big).
\]
Then we have isomorphisms
\begin{equation}
	\Maps_{\QCoh_{/Y}}\big((U, \sF), (X, \sL_{X/Y})\big)
	\simeq \Maps_{Y}(U[\sF], X),
\end{equation}
functorial in pairs $(U,\sF)$ with $\sF$ connective.

An $n$-representable morphism of derived stacks admits a $(-n)$-connective and
$(-1)$-pseudo-coherent cotangent complex.
A morphism of derived Artin stacks $f : X \to Y$ is locally hfp if and only if
\begin{inlinelist}
	\item the induced morphism on classical truncations $f_\cl : X_\cl \to Y_\cl$ is locally of finite presentation, and
	\item $\sL_{X/Y}$ is perfect.
\end{inlinelist}

\subsection{Properties of morphisms}

A morphism of derived stacks $f : X \to Y$ is \emph{étale}, \emph{smooth}, or \emph{quasi-smooth} if it is locally hfp and admits a relative cotangent complex $\sL_{X/Y}$ which vanishes, is perfect of Tor-amplitude $\le 0$, or perfect of Tor-amplitude $\le 1$, respectively.
This notion of smoothness agrees with the one for eventually representable morphisms used in \ssecref{ssec:convent/dstk}.

An eventually representable morphism $f : X \to Y$ is \emph{of Tor-amplitude $\le n$} if for every derived scheme $V$ over $Y$ and every derived scheme $U$ and smooth morphism $U \to X \fibprod_Y V$, the composite $f_0 : U \to V$ is of Tor-amplitude $\le n$ as a morphism of derived schemes; that is, $f_0^* : \QCoh(V) \to \QCoh(U)$ sends $k$-coconnective complexes to $(k+n)$-coconnective complexes.
We say that $f$ is \emph{of finite Tor-amplitude} if it is of Tor-amplitude $\le n$ for some $n$, and \emph{flat} if it is of Tor-amplitude $\le 0$.

A $1$-representable morphism is \emph{proper} if the induced morphism on classical truncations is proper.
Similarly, a representable morphism is \emph{finite} or a \emph{closed immersion} if it is so on classical truncations.

\subsection{\texorpdfstring{$\Gm$-actions}{Gm-actions}}

We explain our conventions for $\Gm$-actions on $\A^1$ and derived vector bundles.

Consider the derived stack classifying invertible sheaves
\[ (\sL \in \uPic), \]
i.e., the functor sending a derived scheme $S$ to the \inftyGrpd $\uPic(S)$ of invertible sheaves on $S$.
This is the classifying stack $\BGm = [\pt/\Gm]$ of the multiplicative group $\Gm$.
We write $\sO(1) \in \QCoh(\BGm)$ for the tautological (universal) invertible sheaf.
Under the equivalence between quasi-coherent complexes on $\BGm$ and $\Z$-graded quasi-coherent complexes on $\pt$, this corresponds (by convention) to $\sO \in \QCoh(\pt)$ in degree $-1$.
See e.g. \cite[2.2.1]{BhattLect}.

Next consider the derived stack classifying pairs
\[ (\sL \in \uPic, ~\sigma : \sL \to \sO) \]
where $\sL$ is an invertible sheaf and $\sigma : \sL \to \sO$ is an $\sO$-linear homomorphism.
We denote it by $\AGm$; it can be thought of as the total space of the tautological line bundle on $\BGm$, i.e.,
\begin{equation}\label{eq:Theta}
	\AGm \simeq \V_{\BGm}(\sO(1)) \simeq \Spec_{\BGm}(\Sym^*_{\sO_{\BGm}}(\sO(1))).
\end{equation}
We often write $\quo{0} : \BGm \hook \AGm$ for the zero section, which is the universal virtual Cartier divisor (see \ssecref{ssec:VCD}).
It is identified with the zero locus of the tautological cosection
\begin{equation}\label{eq:sigma}
	\sigma : \sO_\AGm(1) \to \sO_\AGm
\end{equation}
on $\AGm$, which is the universal cosection classified by the identity on $\AGm$.
Equivalently, $\AGm$ is the stack quotient of $\A^1$ by the action of $\Gm$ by scaling with weight $-1$.
We always write $t^{-1}$ for the coordinate of $\A^1 = \Spec(\Z[t^{-1}])$, which is in homogeneous degree $-1$.
Under the equivalence between quasi-coherent complexes on $\AGm$ and filtered quasi-coherent complexes on $\pt$, $\sigma$ corresponds to the inclusion $(t^{-1})\Z[t^{-1}] \sub \Z[t^{-1}]$.
See e.g. \cite[2.2.5]{BhattLect}.

Note that there is a canonical isomorphism of cotangent complexes
\begin{equation}\label{eq:LAGm}
	\sL_{\AGm} \simeq \quo{0}_*(\sL_{\BGm}),
\end{equation}
where $\quo{0} : \BGm \to \AGm$ is the zero section.
Indeed, under the standard formula for the cotangent complex of a quotient stack (see e.g. \cite[Ex.~8.6.5]{KhanNCTS}), \eqref{eq:LAGm} becomes the isomorphism
\[ \Fib(\sO_{\AGm}(1) \xrightarrow{\sigma} \sO_{\AGm}) \simeq [0]_*\sO_{\BGm}[-1]\]
coming from the exact triangle $\sO_{\AGm}(1) \to \sO_{\AGm} \to \quo{0}_*(\sO_{\BGm})$.

Finally, given a quasi-coherent complex $\sE \in \QCoh(X)$ on a derived stack $X$, consider the derived stack over $X$ classifying pairs
\[ (\sL \in \uPic, ~\phi : \sE \to \sL) \]
where $\sL$ is an invertible sheaf and $\phi : \sE \to \sL$ is an $\sO$-linear homomorphism.
In terms of the derived vector bundle $E := \V_X(\sE)$ classifying cosections $\sE \to \sO$ (see App.~\ref{sec:vb}), this is the quotient stack
\begin{equation}
	[E/\Gm] = [\V_X(\sE)/\Gm] \simeq \V_{X\times\BGm}(\sE(-1)).
\end{equation}
where $\Gm$ acts by scaling with weight $1$.
Here we have written $\sE(n) := \sE \boxtimes \sO(n) \in \QCoh(X\times\BGm)$ for $n\in\Z$; under the equivalence between $\QCoh(X\times\BGm)$ and $\Z$-graded objects of $\QCoh(X)$, this corresponds to viewing $\sE$ in homogeneous degree $-n$.
Unless otherwise specified, any derived vector bundle $E$ over a derived stack $X$ is regarded with the weight $1$ scaling action.\footnote{%
	The exception is the vector bundle $\A^1$ over $\pt$ discussed above, which we always regard with weight $-1$ scaling action.
	We always write the quotient stack as $\Theta$ (following Halpern--Leistner) rather than $[\A^1/\Gm]$ though, to avoid confusion.
}

\changelocaltocdepth{2}
%!TEX root = ../weilres.tex

\section{Weil restrictions}

In this section, we introduce the functor of (derived) Weil restriction along a morphism of derived stacks $h : S \to T$.

\subsection{Definitions}

The base change functor
\begin{equation}
  h^* : \dStk_{/T} \to \dStk_{/S},
  \quad (Y \to T) \mapsto (Y \fibprod_T S \to S)
\end{equation}
admits both left and right adjoints.
The left adjoint $h_\sharp$ is given by
\begin{equation}\label{eq:hsharp}
  h_\sharp(X \to S) = (X \to S \xrightarrow{h} T).
\end{equation}
The right adjoint $h_*$, called \emph{Weil restriction}, sends a derived stack $X$ over $S$ to the derived stack $h_*(X\to S)$ over $T$ given by the functor of points
\begin{equation}
  h_*(X\to S) : (U \to T) \mapsto X(U \fibprod_T S \to S).
\end{equation}
We will usually write $h_*(X) := h_*(X\to S)$ when there is no risk of confusion.
Note that $h_*(S) \simeq T$ and that, more generally, $h_*$ preserves limits by virtue of being a right adjoint.

The counit of the adjunction is a canonical $S$-morphism
\begin{equation}
  \varepsilon : h_*(X) \fibprod_T S \to X.
\end{equation}

When $h : S \to T$ is a \emph{flat} morphism of classical stacks, and $X$ is a classical stack over $S$, $h_*(X)$ is the classical Weil restriction.
Otherwise, $h_*(X)$ will typically not be classical, nor will its classical truncation be the classical Weil restriction. This is because, in general,
$U\times_T S \neq U\times_T^{\cl} S$ even if $U$ is classical.

\subsection{Mapping stacks}
\label{ssec:umaps}
Given derived stacks $X$ and $Y$ over a derived stack $S$, the derived mapping stack $\uMaps_S(X,Y)$ over $S$ is determined by the isomorphisms
\begin{equation}\label{eq:rebellike}
  \Maps_S(U, \uMaps_S(X, Y))
  \simeq \Maps_S(U \fibprod_S X, Y),
\end{equation}
functorial in $U\in\dStk_{/S}$.
We have the following tautological description of the Weil restriction:

\begin{prop}\label{prop:wame}
  Let $h : S \to T$ be a morphism of derived stacks.
  Then for every $X \in \dStk_{/S}$, there is a canonical cartesian square
  \begin{equation*}
    \begin{tikzcd}
      h_*(X) \ar{r}\ar{d}
      & \uMaps_T(S, X) \ar{d}
      \\
      h_*(S) \simeq T \ar{r}{\id_S}
      & \uMaps_T(S, S),
    \end{tikzcd}
  \end{equation*}
  where the lower and upper horizontal arrows correspond under \eqref{eq:rebellike} to the identity $\id_S : S \to S$ and the counit $\varepsilon : h_*(X) \fibprod_T S \to X$, respectively.
\end{prop}

Conversely, mapping stacks can also be described as Weil restrictions:

\begin{prop}\label{prop:map=weil}
  Let $X$ and $Y$ be derived stacks over a derived stack $S$.
  Then there is a canonical isomorphism
  \begin{equation*}
    \uMaps_S(X, Y)
    \simeq h_*(X \fibprod_S Y \to X)
  \end{equation*}
  of derived stacks over $S$, where $h : X \to S$ is the structural morphism.
\end{prop}
\begin{proof}
  Consider the commutative diagram
  \begin{equation*}
    \begin{tikzcd}
      h_*(X \fibprod_S Y) \ar{r}\ar{d}
      & \uMaps_S(X, X \fibprod_S Y) \ar{r}\ar{d}
      & \uMaps_S(X, Y) \ar{d}
      \\
      S \ar{r}
      & \uMaps_S(X, X) \ar{r}{h}
      & \uMaps_S(X, S),
    \end{tikzcd}
  \end{equation*}
  where the left-hand square is as in \propref{prop:wame} and the right-hand square is cartesian because $\uMaps_S(X, -)$ preserves limits.
  Since the lower horizontal composite is an isomorphism, it follows that the upper horizontal composite is also an isomorphism.
\end{proof}

\subsection{Base change}

We record some base change properties of the Weil restriction along $h : S \to T$.
The first is immediate from the definition:

\begin{lem}\label{lem:mouthiness}
  Denote by $h' : S' \to T'$ the base change of $h$ along a morphism $T' \to T$.
  For every derived stack $X$ over $S$, there is a canonical isomorphism
  \[
    h_*(X) \fibprod_T T' \simeq h'_*(X \fibprod_S S')
  \]
  of derived stacks over $T'$.
\end{lem}

From this we deduce:

\begin{cor}\label{cor:thinkably}
  Denote by $h' : S' \to T'$ the base change of $h$ along a morphism $T' \to T$.
  Given a morphism $X_1 \to X_2$ over $S$ and a morphism $f : T' \to h_*(X_2)$ over $T$, corresponding to a morphism $f^\flat : S' \to X_2$ over $S$, the square
  \[\begin{tikzcd}
    h'_*(X_1 \fibprod_{X_2,f^\flat} S')\ar{r}\ar{d}
    & h'_*(S') \simeq T' \ar{d}{f}
    \\
    h_*(X_1) \ar{r}
    & h_*(X_2)
  \end{tikzcd}\]
  is cartesian.
\end{cor}
\begin{proof}
  Since $h'_*$ preserves limits, it preserves fibre products:
  \[
  h'_*(X_1 \times_{X_2} S') \simeq h'_*(X_1\times_S S') \fibprod_{h'_*(X_2\times_S S')}
  h'_*(S')
  \]
  which by \lemref{lem:mouthiness} equals:
  \[
  h_*(X_1)\times_T T' \fibprod_{h_*(X_2)\times_T T'} T' \simeq
  h_*(X_1) \fibprod_{h_*(X_2)} T'.\qedhere
  \]
\end{proof}

\subsection{Morphisms of finite presentation}

\begin{prop}\label{prop:whalefin}
  Let $h : S \to T$ be a morphism of derived stacks and $f : X_1 \to X_2$ a morphism over $S$.
  Then we have:
  \begin{thmlist}
    \item
    If $h$ is quasi-compact quasi-separated and representable and $f$ is locally hfp, then $h_*(f) : h_*(X_1) \to h_*(X_2)$ is locally hfp.
    
    \item
    If $h$ is quasi-compact quasi-separated, representable, and of finite Tor-amplitude, and $f$ is locally afp, then $h_*(f) : h_*(X_1) \to h_*(X_2)$ is locally afp.
  \end{thmlist}
\end{prop}
\begin{proof}
  It will suffice to show that for every affine $T'$ and every morphism $T' \to h_*(X_2)$, the base change $h_*(X_1) \fibprod_{h_*(X_2)} T' \to T'$ has the claimed property.
  By \corref{cor:thinkably} we may replace $T$ by $T'$, $h$ by $h' : S' := S \fibprod_T T' \to T'$, and $f$ by $f' : X_1\fibprod_{X_2} S' \to S'$, and thereby assume that $T$ is affine and $X_2=S$.

  Let $\{Y_\alpha\}_\alpha$ be a cofiltered system of affines over $h_*(S)\simeq T$ with limit $Y$.
  Assume either that
  \begin{inlinelist}
    \item $X$ is locally hfp over $S$, or
    \item $X$ is locally afp over $S$, $h : S \to T$ is of Tor-amplitude $\le d$, and there exists an integer $n$ such that each $\sO_{Y_\alpha}$ is $n$-truncated.
  \end{inlinelist}
  The map
  \[
    \colim_\alpha \Maps_T(Y_\alpha, h_*(X))
    \to \Maps_T(Y, h_*(X))
  \]
  is identified by adjunction with the map
  \begin{equation}\label{eq:ybo1bl}
    \colim_\alpha \Maps_S(Y_\alpha\fibprod_T S, X)
    \to \Maps_S(Y\fibprod_T S, X).
  \end{equation}
  We will say that $S$ is \emph{good} (for $\{Y_\alpha\}_\alpha$) if \eqref{eq:ybo1bl} is invertible.
  More generally, $S'$ over $S$ is \emph{good} if
  \begin{equation}\label{eq:ybo1bl2}
    \colim_\alpha \Maps_{S}(Y_\alpha\fibprod_T S', X)
    \to \Maps_{S}(Y\fibprod_T S', X)
  \end{equation}
  is invertible.
  The claim will follow if we show that $S$ is good.
  
  The proof is a standard Mayer--Vietoris argument using the following observation:
  \begin{enumerate}
    \item[($\ast$)] Suppose given an étale morphism $V \to S$ which is an isomorphism away from an open $U \sub S$.
    If $U$, $V$, and $U\fibprod_S V$ are good, then so is $S$.
  \end{enumerate}
  By Nisnevich descent\footnote{
    To be precise, we are appealing to Nisnevich excision in the sense of \cite[2.2.5]{localization}; by \cite[Thm.~2.2.6]{localization}, Nisnevich excision is equivalent to Nisnevich descent, and $S \mapsto \dStk_{/S}$ satisfies even fpqc descent.
  } for the assignment $S \mapsto \dStk_{/S}$, we have
  \[
      \Maps_S( Y \fibprod_T S, X ) \simeq
      \Maps_S(Y \fibprod_T U, X) \fibprod_{\Maps_{S}(Y \fibprod_T U\times_S V, X)} \Maps_S(Y \fibprod_T V, X)
  \]
  and similarly for each $Y_\alpha$ in place of $Y$.
  Since filtered colimits commute with finite limits, ($\ast$) follows.

  We begin with the affine case.
  If $S$ is affine, $\{Y_\alpha\fibprod_T S\}_\alpha$ is a cofiltered system of affines over $S$ with limit $Y\fibprod_T S$.
  Thus if $X$ is locally hfp over $S$, then \eqref{eq:ybo1bl} is invertible so $S$ is good.
  In the locally afp case, the assumptions imply that each $Y_\alpha \fibprod_T S$ has $(n+d)$-truncated structure sheaf.
  Thus the assumption that $X$ is locally afp over $S$ implies that \eqref{eq:ybo1bl} is invertible, hence $S$ is good.

  Next suppose that $S$ is a quasi-compact and \emph{separated} algebraic space.
  By \cite[Thm.~3.2.3.1]{LurieSAG} there exists a filtration by opens $\initial = U_0 \sub U_1 \sub \cdots \sub U_n = S$, together with affine schemes $V_i$ and étale morphisms $v_i : V_i \to U_i$ invertible away from $U_{i-1}$ (for $1\le i\le n$).
  We refer to the data $(U_i, V_i, v_i)_{0\le i\le n}$ as a scallop decomposition by affines, and argue by induction on the length $n$.
  For $n\leq 1$ this is the affine case proven above.
  Assume $n>1$ and that any $S$ admitting a length $n-1$ scallop decomposition by affines is good.
  The scallop decomposition $(U_i, V_i, v_i)_{0\le i\le n-1}$ for $U_{n-1}$ base changes along $V_n \to U_n = S$ to a scallop decomposition
  \[
    (U_i \fibprod_S V_n,
    V_i \fibprod_S V_n,
    v_i \fibprod_S V_n)_{0\le i\le n-1}
  \]
  for $U_{n-1} \fibprod_S V_n \simeq U_{n-1} \fibprod_{U_n} V_n$.
  Since $S$ is separated, the canonical morphism $V_i \fibprod_S V_n \hook V_i \times V_n$ is a closed immersion, hence $V_i \fibprod_S V_n$ is affine.
  Thus by the induction hypothesis, $V_n$, $U_{n-1}$, and $U_{n-1} \fibprod_{U_n} V_n$ are each good, so by claim ($\ast$) we conclude that $U_n = S$ is good.
  
  Finally, consider the case of a general qcqs derived algebraic space $S$.
  We argue by induction as in the separated case, replacing the role of affines by separated schemes.
  We thus consider scallop decompositions by separated schemes, i.e., $(U_i, V_i, v_i)_{0\le i\le n}$ where $V_i$ are only assumed to be separated schemes.
  Such exist by \cite[Thm.~3.4.2.1]{LurieSAG}, and we choose one.\footnote{In fact, $V_i$ can still be taken affine, but we need to allow scallop decompositions with $V_i$ separated in the induction step.}
  Now $V_n$ is separated, $U_{n-1}$ admits a length $n-1$ scallop decomposition by separated schemes, and (arguing as above) $U_{n-1} \fibprod_{U_n} V_n$ also admits a length $n-1$ scallop decomposition by separated schemes.
  Indeed, since $S$ has separated diagonal, $V_i \fibprod_S V_n$ is a separated scheme.
  We thus conclude by claim ($\ast$) again.
\end{proof}

\subsection{Morphisms of finite Tor-amplitude}

\begin{prop}\label{prop:weiltrunc}
  Let $h : S \to T$ be an eventually representable morphism of derived stacks and $f : X_1 \to X_2$ an almost $n$-truncated\footnote{%
    See \ssecref{ssec:convent/trunc}; for example, any $n$-representable morphism is almost $n$-truncated.
  } morphism over $S$.
  If $h$ is of Tor-amplitude $\le d$, then $h_*(X_1) \to h_*(X_2)$ is almost $(n+d)$-truncated.
\end{prop}

Note that \propref{prop:weiltrunc} is obvious when $h$ is affine.
For the general case we will use the following lemma.

\begin{lem}\label{lem:coconnfta}
  Let $h : S \to T$ be an eventually representable morphism of derived stacks.
  If $h$ is of Tor-amplitude $\le d$, then for any $k$-coconnective derived stack $V$ over $T$, the fibre $V \fibprod_T S$ is $(k+d)$-coconnective.
\end{lem}
\begin{proof}
  Since pulling back along $h$ preserves colimits, we may assume that $V=\Spec R$ where $R$ is
  $k$-truncated.  By replacing $h$ with the base change along $V\to T$, we may also assume that
  $V=T$ is affine so that $S$ is $m$-Artin.

  Let $S_0 \twoheadrightarrow S$ be a smooth surjection where $S_0$ is a disjoint union of affine schemes. Then $S_0$ is $(k+d)$-coconnective by definition.
  Let $S_n := S_0\fibprod_S S_0 \fibprod_S \cdots \fibprod_S S_0$.
  If all the $S_n$, for $n\ge 0$, are $(k+d)$-coconnective, then so is their colimit $S$.  

  If $S$ is $1$-Artin with affine diagonal, then each $S_n$ will be a disjoint union of affine schemes and hence $(k+d)$-coconnective, so the argument above shows that $S$ is $(k+d)$-coconnective.
  If $S$ is $1$-Artin with separated diagonal, then each $S_n$ will have affine diagonal (since separated algebraic spaces have affine diagonal) and hence be $(k+d)$-coconnective by the previous case.
  If $S$ is $1$-Artin with arbitrary diagonal, then each $S_n$ will have separated diagonal (since algebraic spaces have separated diagonal) and hence be $(k+d)$-coconnective by the previous case.

  We now argue by induction for the case of $m$-Artin $S$.
  Since $S$ has $(m-1)$-representable diagonal, each $S_n$ will be $(m-1)$-Artin and hence
  $(k+d)$-connective by the inductive hypothesis. We conclude that $S$ is $(k+d)$-connective.
\end{proof}

\begin{proof}[Proof of \propref{prop:weiltrunc}]
  By base change it is enough to show that if $T$ is affine, $S$ is a derived Artin stack over $T$ such that $h : S \to T$ is of Tor-amplitude $\le d$, and $X$ is an almost $n$-truncated derived stack over $S$, then $h_*(X)$ is almost $(n+d)$-truncated.
  Note that $h_*(X)$ is given by the functor of points
  \begin{equation}\label{eq:Hivuwi}
    V \mapsto \Maps_T(V, h_*(X)) \simeq \Maps_S(V \fibprod_T S, X).
  \end{equation}
  If $V$ is $k$-coconnective, then $V \fibprod_T S$ is $(k+d)$-coconnective by \lemref{lem:coconnfta}.
  Hence the assumption that $X$ is almost $n$-truncated yields that \eqref{eq:Hivuwi} is $(k+d+n)$-truncated.
  As $V$ varies among $k$-coconnectives, this shows that $h_*(X)$ is almost $(n+d)$-truncated.
\end{proof}

\subsection{Representable étale morphisms}

In this paragraph we will study Weil restrictions in the special case of representable étale morphisms.

\begin{thm}\label{thm:flocculency}
	Let $h : S \to T$ be a proper representable morphism of derived stacks.
	For any representable étale morphism $f : X_1 \to X_2$ over $S$, the induced morphism $h_*(f) : h_*(X_1) \to h_*(X_2)$ is representable and étale.
  Moreover, if $f$ is an open immersion, then $h_*(f)$ is an open immersion.
\end{thm}

The proof of \thmref{thm:flocculency} will involve a brief detour on the theory of nonabelian étale sheaves on derived algebraic spaces.
Given a derived algebraic space $S$, let $\dSpc_{/S}$ denote the \inftyCat of derived algebraic spaces over $S$ and $S_\et$ the full subcategory spanned by $U\in\dSpc_{/S}$ with $U \to S$ étale. 
The inclusion $u : S_\et \hook \dSpc_{/S}$ is continuous with respect to the étale topology and, under the identification $\Shv(\dSpc_{/S}) \simeq \dStk_{/S}$, gives rise to a pair of adjoint functors
\[
u^* : \Shv(S_\et) \to \dStk_{/S},
\qquad u_* : \dStk_{/S} \to \Shv(S_\et)
\]
where $u_*$ is restriction along $u$ and $u^*$ is the unique colimit-preserving functor sending the sheaf $\h_S(U)$ represented by $U \in S_\et$ to the sheaf on $\dSpc_{/S}$ represented by $U$.
Moreover, $u^*$ is fully faithful, so in particular
\begin{equation}\label{eq:Wadpawnuw}
  \Gamma(U, u^*\sF)
  \simeq \Maps(u^*\h_S(U), u^*\sF)
  \simeq \Gamma(U, \sF)
\end{equation}
for every $\sF \in \Shv(S_\et)$ and $U\in S_\et$.

\begin{thm}\label{thm:kappland}
	Let $S$ be a derived algebraic space.
	Then a sheaf on $S_\et$ is representable if and only if it is discrete (i.e., takes values in sets).
\end{thm}
\begin{proof}
  By derived invariance of the étale site, we may assume that $S$ is classical.
  In this case, $S_\et$ is the classical étale site (consisting of underived algebraic spaces étale over $S$), and the claim is proven in \cite[Ch.~VII, \S 1]{ArtinAlgSpc}\footnote{Under a noetherian assumption which only is used to ensure that the sheaf is represented by a \emph{quasi-separated} algebraic space. Our algebraic spaces are not required to be quasi-separated.}.
\end{proof}

\begin{defn}\label{defn:keratoplastic}
	Given a sheaf of \emph{sets} $\sF \in \Shv(S_\et)$, it follows from \thmref{thm:kappland} that there exists a unique derived algebraic space $X \in S_\et$ representing $\sF$ on $S_\et$.
	By \eqref{eq:Wadpawnuw}, the derived stack $u^*(\sF) \in \dStk_{/S}$ is then also represented by $X$.
	We call $X$ the \emph{espace \'etal\'e} of $\sF$ (compare \cite[Ch.~VII, Def.~(1.3)]{ArtinAlgSpc}).
\end{defn}

\begin{constr}
	Given a morphism of derived algebraic spaces $h : S \to T$, there is an adjoint pair of functors
	\[
	h_\et^* : \Shv(T_\et) \to \Shv(S_\et),
	\qquad h_{\et,*} : \Shv(S_\et) \to \Shv(T_\et).
	\]
	The right adjoint $h_{\et,*}$ is restriction along $h^{-1} : T_\et \to S_\et$ and the left adjoint $h_\et^*$ is the unique colimit-preserving functor sending the sheaf $\h_T(V)$ represented by $V \in T_\et$ to the sheaf represented by $V\fibprod_T S\in S_\et$.
	There is a canonical isomorphism $h^* u^* \simeq u^* h_\et^*$, since both are the unique colimit-preserving functor $\Shv(T_\et) \to \dStk_{/S}$ sending $\h_T(V)$ to the derived stack $V\fibprod_T S$.
    By adjunction, we also have $u_* h_* \simeq h_{\et,*} u_*$.
\end{constr}

\begin{thm}[Proper base change]\label{thm:trifasciated}
	Let $h : S \to T$ be a proper morphism of derived algebraic spaces.
	Then the canonical natural transformation
	\[
	u^* h_{\et,*}
	\xrightarrow{\mrm{unit}} h_* h^* u^* h_{\et,*}
	\simeq h_* u^* h_\et^* h_{\et,*}
	\xrightarrow{\mrm{counit}} h_* u^*
	\]
	is invertible on discrete sheaves.
\end{thm}
\begin{proof}
	It will suffice to show that for every discrete $\sF \in \Shv(S_\et)$ and every $a : T' \to T \in \dSpc_{/T}$ the induced map
	\[
	\Gamma(T', u^* h_{\et,*}(\sF)) \to \Gamma(T', h_* u^*(\sF))
	\]
	is invertible. Let $b : S' \to S$ and $h': S' \to T'$ denote the base
    changes of $a$ and $h$. Using \eqref{eq:Wadpawnuw}, the left hand side is
    identified with
    \[
    \Gamma(T',a^*u^*h_{\et,*}(\sF)) \simeq \Gamma(T',a_{\et}^*h_{\et,*}(\sF))
    \]
    and the right hand side with
    \[
    \Gamma(S', b^*u^*(\sF)) \simeq \Gamma(S',b_{\et}^*(\sF))=\Gamma(T', h'_{\et,*}b_\et^*(\sF)).
    \]
    It is thus enough to prove that the canonical map
    $a_{\et}^*h_{\et,*}(\sF)\to h'_{\et,*}b_\et^*(\sF)$ is invertible.
	By derived invariance of the étale site, we may replace all derived algebraic spaces in sight by their classical truncations.
	Now the statement follows from the classical nonabelian proper base change theorem of \cite[Exp.~XII, Thm.~5.1(i)]{SGA4} (which generalizes to algebraic spaces by Chow's lemma for algebraic spaces, see \cite[Tag~\spref{0DG0}]{Stacks}).
\end{proof}

\begin{proof}[Proof of \thmref{thm:flocculency}]
  Let $f : X_1 \to X_2$ be a representable étale morphism over $S$.
  It will suffice to show that for every affine $T'$ and every morphism $T' \to h_*(X_2)$, the base change $h_*(X_1) \fibprod_{h_*(X_2)} T' \to T'$ has the claimed property.
  By \corref{cor:thinkably} we may replace $T$ by $T'$, $h$ by $h' : S' := S \fibprod_T T' \to T'$, and $f$ by $f' : X_1\fibprod_{X_2} S' \to S'$, and thereby assume that $T$ is affine and $X_2=S$.
  Then $S$ is a derived algebraic space and $X := X_1$ is a derived algebraic space, étale over $S$, so by \thmref{thm:trifasciated} we have
	\[
    h_* (X)
    \simeq h_* (u^* \h_S(X))
    \simeq u^* h_{\et,*} (\h_S(X)).
	\]
	It follows from \thmref{thm:kappland} that $h_*(X)$ is represented by the espace \'etal\'e of $h_{\et,*} (\h_S(X)) \in \Shv(T_\et)$.

  Since $h_*$ preserves monomorphisms (as it preserves fibred products), the case of open immersions follows from that of representable étale morphisms.
\end{proof}

\subsection{Effective epimorphisms}

The following results are the only ingredients in the proof of \thmref{thm:intro/weil} which require the assumption that $h$ is \emph{finite} and not just sharp (see App.~\ref{sec:sharp}).
We recall that finite morphisms are proper representable by assumption.

\begin{prop}\label{prop:unconversableness}
  Let $h : S \to T$ be a finite morphism.
  Let $p : U \to X$ be a locally hfp morphism of derived stacks over $S$.
  If $p$ is an effective epimorphism in the étale topology, then so is the induced morphism $h_*(p): h_*(U) \to h_*(X)$.
  In particular, it implies that $h_*(p)$ is surjective.
\end{prop}
\begin{proof}
  Since $p$ is locally hfp, so is $h_*(p) : h_*(U)\to h_*(X)$ by \propref{prop:whalefin}.
  To show that $h_*(p)$ is an effective epimorphism, it will thus suffice to show that for every local ring $R$ which is strictly henselian and every morphism $a : \Spec(R) \to h_*(X)$, the base change
  \begin{equation*}
    h_*(U)\fibprod_{h_*(X)} \Spec(R) \to \Spec(R)
  \end{equation*}
  admits a section.
  By \corref{cor:thinkably}, this is identified with the morphism
  \[
    h'_*(U\fibprod_X S')
    \to h'_*(S') \simeq \Spec(R)
    \]
  where $S' := S\fibprod_T \Spec(R)$, $h' : S' \to \Spec(R)$ denotes the base change of $h$, and $a^\flat : S' \to X$ is the $S$-morphism corresponding to $a$.
  
  Since $h'$ is finite, each connected component $S'_i$ of $S'$ is strictly henselian.
  Since $p$ is an effective epimorphism, each $U\fibprod_X S'_i \to S'_i$ admits a section.
  These give rise to a section of $U\fibprod_X S' \to S'$, whence by functoriality a section of $h'_*(U\fibprod_X S') \to h'_*(S')$ as required.
\end{proof}

\begin{lem}\label{lem:metapsychology}
  Let $h : S \to T$ be a finite morphism.
  Let $(U_\alpha)_{\alpha\in I}$ be a small collection of derived stacks over $S$.
  Let $U_J=\coprod_{\alpha\in J} U_\alpha$ for every $J\subseteq I$. Then
  \[
  \{ h_*(U_J)\to h_*(U_I) \}_{J\subseteq I, \;J\;\text{finite}}
  \]
  is a Zariski covering family.
\end{lem}
\begin{proof}
  For every $J$ the morphism $h_*(U_J)\to h_*(U_I)$ is an open immersion (\thmref{thm:flocculency}), so it remains to show that
  \[ \coprod_{J} h_*(U_J) \to h_*(U_I) \]
  is surjective.
  It will suffice to show that for every field $\kappa$ and morphism
  $a : T'=\Spec(\kappa)\to h_*(U_I)$, the base change (\corref{cor:thinkably})
  \[
    \coprod_J h'_*(U_J\fibprod_{U_I} S')
    \to h'_*(S') \simeq \Spec(\kappa)
  \]
  admits a section, where $h' : S' \to T'$ denotes the base change of $h$.
  Since $S'$ is finite over $T'=\Spec(\kappa)$, the set $\abs{S'}$ is finite.
  Thus $a^\flat(\abs{S'}) \sub \abs{U_I}$ lies in $\abs{U_J}$ for some finite subset $J$, i.e., $U_J\fibprod_{U_I} S' \simeq S'$ and in particular $h'_*(U_J \fibprod_{U_I} S') \simeq h'_*(S')$.
\end{proof}

\begin{cor}\label{cor:diffusedly}
  Let $h : S \to T$ be a finite morphism.
  Let $X$ be a derived stack over $S$ and $\{U_\alpha \to X\}_{\alpha\in I}$ a collection of locally hfp morphisms such that $\coprod_{\alpha\in I} U_\alpha \to X$ is an effective epimorphism in the étale topology.
  Then the locally hfp morphism
  \[
    p:\coprod_{J\;\text{finite}} h_*(U_J)
    \to h_*(X)
  \]
  is an effective epimorphism in the étale topology, where $U_J := \coprod_{\alpha\in J} U_J$.
\end{cor}
\begin{proof}
  Combine \propref{prop:unconversableness} and \lemref{lem:metapsychology}.
\end{proof}

\subsection{Cotangent complexes; étale and smooth morphisms}
\label{ssec:cot}

In this paragraph we discuss cotangent complexes of derived Weil restrictions.
In the context of spectral algebraic geometry, a special case of this computation is obtained using a slightly different argument in \cite[\S 19.1.3]{LurieSAG}.\footnote{
  This is not exactly true (except in characteristic zero), since \emph{loc.\ cit.}\ works with the ``topological'' cotangent complex rather than the ``algebraic'' one; see \cite[\S 25.3.5]{LurieSAG} for a discussion of the difference.
  Thus \propref{prop:cotan} combined with \examref{exam:stayship} yields an analogue of \cite[Prop.~19.1.4.3]{LurieSAG} for the \emph{algebraic} cotangent complex.
}

The computation of the cotangent complex of a Weil restriction will make use of the notion of \emph{sharpness} introduced in Appendix~\ref{sec:sharp}.
We recall in particular that for a sharp morphism $f$, the pull-back of quasi-coherent sheaves $f^*$
admits a left adjoint $f_\sharp$ satisfying base change (\remref{rem:hsharp}).

\begin{prop}\label{prop:cotan}
  Let $h : S \to T$ be a sharp morphism of derived stacks and $f : X_1 \to X_2$ a morphism of derived stacks over $S$.
  Consider the following commutative diagram:
  \begin{equation}\label{eq:Egizaes}
    \begin{tikzcd}
      & h_*(X_1) \fibprod_T S \ar{r}{h'}\ar{d}{p}\ar[swap]{ld}{\varepsilon}
      & h_*(X_1) \ar{d}
      \\
      X_1 \ar{r}
      & S \ar{r}{h}
      & T,
    \end{tikzcd}
  \end{equation}
  where $\varepsilon$ is the counit of the adjunction.
  If $f$ admits a cotangent complex $\sL_{X_1/X_2}$, and
  $h'_\sharp \varepsilon^*(\sL_{X_1/X_2})$ is locally eventually connective, then
  \begin{equation*}
    \sL_{h_*(X_1)/h_*(X_2)} = h'_\sharp \varepsilon^*(\sL_{X_1/X_2})
  \end{equation*}
  is a cotangent complex for $h_*(f) : h_*(X_1) \to h_*(X_2)$.
\end{prop}
\begin{proof}
  For every morphism $V \to h_*(X_1)$ with $V$ affine and every $\sF \in \QCoh(V)_{\ge 0}$, consider the commutative diagram
  \begin{equation*}
    \begin{tikzcd}
      & U \ar{r}{h_V}\ar{d}\ar[swap,bend right]{ldd}{u}
      & V \ar{d}{v}
      \\
      & h_*(X_1) \fibprod_T S \ar{r}{h'}\ar{d}\ar{ld}{\varepsilon}
      & h_*(X_1) \ar{d}
      \\
      X_1 \ar{r}
      & S \ar{r}{h}
      & T
    \end{tikzcd}
  \end{equation*}
  where both squares are cartesian.
  It will suffice to demonstrate functorial isomorphisms
  \begin{equation*}
    \Maps_{\QCoh(V)}(v^* h'_\sharp \varepsilon^* \sL_{X_1/X_2}, \sF)
    \simeq \on{Der}_V(h_*(X_1)/h_*(X_2); \sF).
  \end{equation*}
  By adjunction, the right-hand side is identified with
  \begin{equation*}
    \on{Der}_U(X_1/X_2; h_V^*(\sF)).
  \end{equation*}
  Similarly, the left-hand side is identified with
  \begin{equation*}
    \Maps_{\QCoh(V)}(h_{V,\sharp} u^* \sL_{X_1/X_2}, \sF)
    \simeq \Maps_{\QCoh(U)}(u^* \sL_{X_1/X_2}, h_V^*\sF)
  \end{equation*}
  by the base change formula \eqref{eq:sharpbc} and adjunction.
  Now the claim follows by the universal property of $\sL_{X_1/X_2}$.
\end{proof}

\begin{cor}\label{cor:ontogenic}
  Let $h : S \to T$ be an afp proper representable morphism which is of Tor-amplitude $\le d$ and universally of cohomological dimension $\le e$.
  If $f : X_1 \to X_2$ is a morphism of derived stacks over $S$ admitting a perfect cotangent complex of Tor-amplitude $[a,b]$, then $h_*(X_1) \to h_*(X_2)$ admits the cotangent complex
  \begin{equation*}
    \sL_{h_*(X_1)/h_*(X_2)}
    \simeq h'_* (\varepsilon^*(\sL_{X_1/X_2}) \otimes p^*\omega_{S/T}),
  \end{equation*}
  where the notation is as in \eqref{eq:Egizaes} and $\omega_{S/T} = h^!(\sO_T)$ is the dualizing complex.
  Moreover, $\sL_{h_*(X_1)/h_*(X_2)}$ is perfect of Tor-amplitude $[a-d,b+e]$.
  In particular, if $f$ is étale, then so is $h_*(f) : h_*(X_1) \to h_*(X_2)$.
  Similarly, if $e=0$ and $f$ is smooth, resp.~quasi-smooth, then so is $h_*(f)$.
\end{cor}
\begin{proof}
  By \examref{exam:stayship}, the morphism $h$ is sharp and $h'_\sharp(-)\simeq
  h'_*(- \otimes p^*\omega_{S/T})$. By \lemref{lem:fortuitist},
  the complex $h'_\sharp(\varepsilon^*(\sL_{X_1/X_2}))$ is perfect of Tor-amplitude $[a-d,b+e]$ and by \propref{prop:cotan}
  this is a cotangent complex for $h_*(X_1) \to h_*(X_2)$.
  For the last claims it suffices to note that $h_*$ preserves locally hfp morphisms by \propref{prop:whalefin}.
\end{proof}

\begin{cor}\label{cor:middlemost}
	Let $h : S \to T$ be a finite morphism of finite Tor-amplitude.
	For any étale morphism $X_1 \to X_2$ over $S$ which is representable by $n$-Deligne--Mumford stacks, the induced morphism $h_*(X_1) \to h_*(X_2)$ is representable by $n$-Deligne--Mumford stacks.
\end{cor}
\begin{proof}
  By \corref{cor:thinkably} we may assume that $T$ is affine and $X_2=S$.
	If $n=0$, then $X := X_1$ is an étale algebraic space over $S$ and $h_*(X)$ is an étale algebraic space over $T$ by \thmref{thm:flocculency}.
	Take $n>0$ and assume the statement known for $n-1$.
	Given an $n$-Deligne--Mumford stack $X$ étale over $S$, choose an étale surjection $p : U \twoheadrightarrow X$ where $U$ is an algebraic space.
	By the $n=0$ case, $h_*(U)$ is an algebraic space.
	By \corref{cor:ontogenic} and \propref{prop:unconversableness}, $h_*(U) \to h_*(X)$ is an étale surjection.
	Since $X$ has $(n-1)$-representable étale diagonal, the diagonal $h_*(X) \to h_*(X \fibprod_S X) \simeq h_*(X) \fibprod_T h_*(X)$ is $(n-1)$-representable by the induction hypothesis.
	Thus $h_*(X)$ is $n$-Deligne--Mumford.
\end{proof}

\subsection{Derived vector bundles and zero loci}

Let $\sE \in \QCoh(S)$ be a quasi-coherent complex on a derived stack $S$.
We denote by $\V_S(\sE) \to S$ the derived stack of cosections $\sE \to \sO$, see Appendix~\ref{sec:vb}.
See Appendix~\ref{sec:sharp} for the notion of sharp morphisms and the notation $h_\sharp$.

\begin{prop}\label{prop:unsaddled}
  Let $h : S \to T$ be a sharp morphism.
  Then for every quasi-coherent complex $\sE \in \QCoh(S)$, there is a canonical isomorphism $h_*(\V_S(\sE)) \simeq \V_T(h_\sharp(\sE))$ over $T$.
\end{prop}
\begin{proof}
  For every $T'\in\dStk_{/T}$ we have a canonical isomorphism
  \begin{align*}
    \Maps_T\bigl(T',\V_T(h_\sharp(\sE))\bigr)
    &\simeq \Maps_{\QCoh(T')}(h_\sharp(\sE)|_{T'}, \sO_{T'}) \\
    &\simeq \Maps_{\QCoh(T')}(h'_\sharp(\sE|_{S'}), \sO_{T'}) \qquad \text{(by \eqref{eq:sharpbc})}\\
    &\simeq \Maps_{\QCoh(S')}(\sE|_{S'}, \sO_{S'}) \\
    &\simeq \Maps_S\bigl(S',\V_S(\sE)\bigr)
       \simeq \Maps_T\bigl(T',h_*(\V_S(\sE))\bigr)
  \end{align*}
  functorial in $T'$ where $h' : S'\to T'$ is the base change of $h$.
\end{proof}

\begin{cor}\label{cor:coxcombical}
  Let $h : S \to T$ be a qcqs 1-representable morphism which is universally of finite cohomological dimension, and sharp of Tor-amplitude $\le d$.
  Let $X$ be a derived stack over $S$ and $\sE \in \Perf(X)$ a perfect complex of Tor-amplitude $[a,b]$.
  Then the morphism $h_*(\V_X(\sE)) \to h_*(X)$ is:
  \begin{thmlist}
    \item affine if $a \ge d$.
    \item $(d-a)$-representable if $a < d$.
  \end{thmlist}
\end{cor}
\begin{proof}
  It will suffice to show that for every affine $T'$ and every morphism $T' \to h_*(X)$, the base change $h_*(\V_X(\sE)) \fibprod_{h_*(X)} T' \to T'$ has the claimed property.
  By \corref{cor:thinkably} we may thus assume that $X=S$.
  By \lemref{lem:fortuitist}, the complex $h_\sharp(\sE)$ is perfect and $(a-d)$-connective. The
  result follows by \propref{prop:unsaddled} and \thmref{thm:V}.  
\end{proof}

\begin{cor}\label{cor:floriated}
  Let $h : S \to T$ be an afp proper representable morphism of finite Tor-amplitude.
  Then for every perfect complex $\sE \in \Perf(S)$, there is a canonical isomorphism $h_*(\V_S(\sE)) \simeq \V_T(h_*(\sE^\vee)^\vee)$ over $T$.
\end{cor}
\begin{proof}
  Combine \propref{prop:unsaddled} with \examref{exam:stayship}.
\end{proof}

Let $\sE \in \QCoh(X)$ be a quasi-coherent complex on a derived stack $X$.
Given a cosection of $\sE$, i.e., a morphism $s : \sE \to \sO_X$ in $\QCoh(X)$, the (derived) \emph{zero locus} of $s$ is defined by the cartesian square
\begin{equation*}
  \begin{tikzcd}
    Z \ar{r}\ar{d}
    & X \ar{d}{s}
    \\
    X \ar{r}{0}
    & \V_X(\sE),
  \end{tikzcd}
\end{equation*}
where $s$ is the section of $\V_X(\sE) \to X$ corresponding to the cosection under \eqref{eq:besnivel}.

\begin{prop}\label{prop:zeroalg}
  Let $h : S \to T$ be a sharp morphism of Tor-amplitude $\le d$ which is qcqs and representable.
  Let $X$ be a derived stack over $S$, $\sE \in \Perf(X)$ a perfect complex of Tor-amplitude $[a,b]$, and $s : \sE \to \sO_X$ a cosection with zero locus $Z \to X$.
  If $a \ge d$ (resp. $a < d$), then $h_*(Z) \to h_*(X)$ is a closed immersion (resp. $(d-a-1)$-representable).
\end{prop}

\begin{proof}
  Since $h_*$ preserves limits, we have a cartesian square
  \[\begin{tikzcd}
    h_*(Z)\ar{r}\ar{d}
    & h_*(X)\ar{d}{h_*(s)}
    \\
    h_*(X)\ar{r}{h_*(0)}
    & h_*(\V_X(\sE)).
  \end{tikzcd}\]
  By functoriality, the lower horizontal arrow is a section of the morphism $h_*(\V_X(\sE)) \to h_*(X)$.
  By \corref{cor:coxcombical}, the latter is affine when $a\ge d$ and $(d-a)$-representable otherwise; any section of it is therefore a closed immersion, resp.~$(d-a-1)$-representable.
  By base change, the upper horizontal arrow has the same property.
\end{proof}

%!TEX root = ../weilres.tex

\section{Normal deformations}

\subsection{Virtual Cartier divisors}
\label{ssec:VCD}

Recall the following definition from \cite{blowups}:

\begin{defn}
  Let $S$ be a derived stack.
  A \emph{virtual Cartier divisor} on $S$ is a closed immersion $D \hook S$ which is quasi-smooth of relative virtual dimension $-1$.
  That is, it is hfp with conormal complex $\sL_{D/S}[-1] \in \QCoh(D)$ perfect of Tor-amplitude $[0,0]$ and of rank $1$.
\end{defn}

For any virtual Cartier divisor $i : D \hook S$, there is a canonical exact triangle
\begin{equation*}
  \sO_S(-D) \to \sO_S \to i_*(\sO_D)
\end{equation*}
in $\QCoh(S)$, where $\sO_S(-D)$ is locally free of rank one.
Let $\sO_D(-D) := i^*(\sO_S(-D))$ denote the restriction and $\sO_D(D) := \sO_D(-D)^{\otimes -1}$ the inverse.
The conormal complex of $D \hook S$ is given by
\begin{equation*}
  \sL_{D/S}[-1] \simeq \sO_D(-D)
\end{equation*}
and its relative dualizing complex is
\begin{equation*}
  \omega_{D/S} := i^!(\sO_S) \simeq \sO_D(D)[-1].
\end{equation*}
See \cite[\S\S 4.1--4.2]{kblow}.

When $S$ is a derived scheme, a closed immersion $D \hook S$ is a virtual Cartier divisor if and only if there exists, Zariski-locally on $S$, a cartesian square
\begin{equation*}
  \begin{tikzcd}
    D\ar[hookrightarrow]{r}\ar{d}
    & S\ar{d}{f}
    \\
    \pt \ar[hookrightarrow]{r}{0}
    & \A^1.
  \end{tikzcd}
\end{equation*}
In other words, it is locally of the form
\begin{equation*}
  \Spec(A\modmod f) \hook \Spec(A)
\end{equation*}
for some animated ring $A$ and some $f \in A$, where $A\modmod f$ is the ``derived quotient'' as in \cite[2.3.1]{blowups}.
For a general $S$, a closed immersion $D \hook S$ is a virtual Cartier divisor if and only if for every (affine) derived scheme $S$ and every morphism $T \to S$, the base change $D \fibprod_S T \hook T$ is a virtual Cartier divisor.
See \cite[Prop.~2.3.8]{blowups}.

For any virtual Cartier divisor $D \hook S$, there exists (globally) a cartesian square
\begin{equation*}
  \begin{tikzcd}
    D\ar[hookrightarrow]{r}\ar{d}
    & S\ar{d}{f}
    \\
    \BGm \ar[hookrightarrow]{r}{\quo{0}}
    & \AGm,
  \end{tikzcd}
\end{equation*}
where $\AGm$ is as in \eqref{eq:Theta}, $\quo{0}$ denotes the zero section, and $D \to \BGm$ classifies the invertible sheaf $\sO_D(-D)$.

More precisely, write $\VCart(S)$ for the \inftyGrpd of virtual Cartier divisors on $S$; as $S$ varies, this defines a derived stack $\VCart$.
The morphism
\begin{equation*}
  \AGm \to \VCart,
\end{equation*}
classifying the virtual Cartier divisor $\quo{0} : \BGm \to \AGm$, is invertible (see \cite[Prop.~3.2.6]{blowups}).

\subsection{Normal deformations}
\label{ssec:D}

\begin{defn}
  Let $f : X \to Y$ be a morphism of derived stacks.
  Given a derived stack $S$ over $Y$, a \emph{virtual Cartier divisor on $S$ over $f$} is the data of a virtual Cartier divisor $D$ on $S$ together with a commutative square
  \begin{equation}\label{eq:vcdover}
    \begin{tikzcd}
      D \ar{r}{i_D}\ar{d}
      & S \ar{d}
      \\
      X \ar{r}{f}
      & Y.
    \end{tikzcd}
  \end{equation}
  We write $\cD_{X/Y}(S \to Y)$ for the \inftyGrpd of virtual Cartier divisors on $S$ over $f$.
  As $S$ varies among derived schemes over $Y$, this defines a derived stack $\cD_{X/Y}$ over $Y$.
\end{defn}

\begin{exam}
  When $f$ is the identity of $X$, we have $\cD_{X/X} \simeq X \times \AGm$.
\end{exam}

By construction, $\cD_{X/Y}$ is equipped with a universal virtual Cartier divisor $\cN_{X/Y}$ over $f$:
\begin{equation}\label{eq:intertrude}
  \begin{tikzcd}
    \cN_{X/Y} \ar{r}{i_\cN}\ar{d}
    & \cD_{X/Y} \ar{d}
    \\
    X \ar{r}{f}
    & Y.
  \end{tikzcd}
\end{equation}
This is classified by a morphism $\cD_{X/Y} \to \AGm$ together with a cartesian square
\begin{equation}
  \begin{tikzcd}
    \cN_{X/Y} \ar{r}{i_\cN}\ar{d}
    & \cD_{X/Y} \ar{d}
    \\
    \BGm \ar{r}{\quo{0}}
    & {\AGm},
  \end{tikzcd}
\end{equation}
or equivalently
\begin{equation}\label{eq:defGm}
  \begin{tikzcd}
    \cN_{X/Y} \ar{r}{i_\cN}\ar{d}
    & \cD_{X/Y} \ar{d}
    \\
    Y \times \BGm \ar{r}{\quo{0}}
    & Y \times {\AGm}.
  \end{tikzcd}
\end{equation}
Note that over $Y \times \AGm$, the functor of points of $\cD_{X/Y}$ is as follows: for $S \in \dStk_{Y \times \AGm}$, classifying a virtual Cartier divisor $D \hook S$, a morphism $S \to \cD_{X/Y}$ over $Y\times\AGm$ is given by a morphism $D \to X$ over $Y$ together with the datum of a commutative square of the form \eqref{eq:vcdover}.

The datum of the cartesian square \eqref{eq:defGm} is equivalently that of a cartesian square
\begin{equation}\label{eq:def}
  \begin{tikzcd}
    \Nl_{X/Y} \ar{r}{i_N}\ar{d}
    & \Dl_{X/Y} \ar{d}
    \\
    Y \ar{r}{0}
    & Y \times \A^1
  \end{tikzcd}
\end{equation}
of derived stacks with $\Gm$-actions, where $\Gm$ acts trivially on $Y$ and by scaling with weight $-1$ on $\A^1$.
Here $\Nl_{X/Y} = \cN_{X/Y} \fibprod_{\BGm} \{0\}$ by definition, so that $\cN_{X/Y} \simeq [\Nl_{X/Y}/\Gm]$.
Likewise for $\Dl_{X/Y} = \cD_{X/Y} \fibprod_{\BGm} \{0\}$. 

\begin{defn}
	\label{def:normdef}
  The \emph{normal deformation} of a morphism of derived stacks $f : X \to Y$ is the derived stack $\Dl_{X/Y}$ over $Y\times\A^1$ together with its canonical $\Gm$-action.
\end{defn}

\subsection{Description via Weil restrictions}

\begin{prop}\label{prop:defweil}
  Let $f : X \to Y$ be a morphism of derived stacks.
  There is a canonical $\Gm$-equivariant\footnote{%
    Note that by functoriality of Weil restrictions, $0_*(X \to Y)$ inherits a canonical $\Gm$-action such that the projection to $Y\times\A^1$ is $\Gm$-equivariant.
  } isomorphism
  \begin{equation*}
    \Dl_{X/Y}
    \simeq 0_*(X \xrightarrow{f} Y)
  \end{equation*}
  of derived stacks over $Y \times \A^1$, where $0 : Y \to Y \times \A^1$ is the zero section.
  In particular, there is a canonical isomorphism
  \begin{equation*}
    \cD_{X/Y}
    \simeq \quo{0}_*(X \times \BGm \xrightarrow{f\times\id} Y \times \BGm)
  \end{equation*}
  of derived stacks over $Y \times \AGm$, where $\quo{0} : Y \times \BGm \to Y \times \AGm$.
\end{prop}
\begin{proof}
  It will suffice to show the second statement.
  The functor of points of the right-hand side is given as follows: given $S \in \dStk_{Y \times \AGm}$, classified by a virtual Cartier divisor $D \hook S$, we have
  \begin{align*}
    \Maps_{Y\times\AGm}(S, \quo{0}_*(X \times \BGm \to Y \times \BGm))
    &\simeq \Maps_{Y\times\BGm}(D, X \times \BGm)\\
    &\simeq \Maps_{Y}(D, X)
  \end{align*}
  where $D$ is regarded over $Y$ via the composite $D \hook S \to Y$.
  But this is also the functor of points of $\cD_{X/Y}$ over $Y\times\AGm$.
\end{proof}

From \propref{prop:map=weil} we deduce:

\begin{cor}\label{cor:d=map}
  Let $f : X \to Y$ be a morphism of derived stacks.
  There is a canonical $\Gm$-equivariant isomorphism
  \begin{equation*}
    \Dl_{X/Y}
    \simeq \uMaps_{Y\times\A^1}(Y, X \times \A^1)
  \end{equation*}
  over $Y \times \A^1$, or equivalently
  \begin{equation*}
    \cD_{X/Y}
    \simeq \uMaps_{Y\times\AGm}(Y \times \BGm, X \times \AGm)
  \end{equation*}
  over $Y \times \AGm$.
\end{cor}

\subsection{Description via ring stacks}
\label{ssec:ringstk}

One can describe the normal deformation $\Dl_{X/Y}$ through an intermediate \emph{ring stack}.

\begin{constr}\label{constr:ringstk}
  Let $\cR \to \AGm$ denote the vector bundle stack
  \[
    \cR := \V_{\AGm}(\sR),
    \quad \text{where}~\sR = \Fib(\sO_{\AGm} \xrightarrow{\sigma^\vee} \sO_{\AGm}(-1))
  \]
  and $\sigma^\vee$ is the dual of the tautological cosection \eqref{eq:sigma} on $\AGm$.
  In other words, $\cR$ is the quotient of the linear morphism $\V_{\AGm}(\sO(-1)) \to \V_{\AGm}(\sO)$ induced by $\sigma^\vee$.
  For a derived stack $Y$, we also set $\cR_Y := Y\times\cR$.
\end{constr}

Unravelling definitions, we find that as a functor $(\Aff_{/Y\times\AGm})^\op \to \Grpdoo$, $\cR_Y$ is given by the assignment
\[
  (\Spec(A) \xrightarrow{a} Y \times \AGm)
  \mapsto B^\circ
\]
where $\Spec(B) \hook \Spec(A)$ is the virtual Cartier divisor classified by $a$, and $B^\circ$ denotes the underlying \inftyGrpd of $B$.\footnote{Note that $\V_{\AGm}(\sO(-1))(\Spec A) = \sO_A(1)$, and $B$ is the cofibre of the map $\sO_A(1) \to A$ classifying $a$.}
In particular we may regard $\cR_Y$ as a functor valued in animated rings, i.e., a ``derived ring stack'', via \cite[Lem.~3.2.8]{blowups}.

\begin{prop}
  For every morphism of derived stacks $f : X \to Y$, there is a canonical isomorphism of derived stacks
  \[
    \cD_{X/Y}(-)
    \simeq X(\cR_Y(-)).
  \]
\end{prop}
\begin{proof}
  Given $\Spec(A) \to Y \times \AGm$, recall that the corresponding virtual Cartier divisor is defined by
  \[
    \Spec(B) := \Spec(A) \fibprod_{Y \times \AGm} Y \times \BGm \simeq \quo{0}^*(\Spec(A)),
  \]
  where $\quo{0} : Y \times \BGm \to Y \times \AGm$ is the zero section.
  Thus we have canonical functorial isomorphisms
  \[\begin{multlined}
    \Maps_{Y\times\AGm}(\Spec(A), \quo{0}_*(X \times \BGm))\\
    \simeq \Maps_{Y\times\BGm}(0^*\Spec(A), X \times \BGm) \simeq X(B)
  \end{multlined}\]
  and the claim follows.
\end{proof}

\begin{rem}\label{rem:GadR}
  Consider the variant of $\cR$ where $\V(\sO(-1))$ is replaced by its formal completion $\V(\sO(-1))^\wedge$ along the zero section, i.e., the quotient stack of the linear morphism
  \[
    \V_{\AGm}(\sO(-1))^\wedge \to \V_{\AGm}(\sO).
  \]
  This recovers the ring stack $\bG_a^\mrm{dR,+}$ which controls Hodge-filtered algebraic de Rham cohomology in characteristic zero.
  In particular, taking $Y=\pt$ for simplicity, the formal completion of the normal deformation
  \[
    (\cD_{X/\pt})^\wedge_{X\times\AGm},
  \]
  along $\cD f : X\times\AGm \to \cD_{X/\pt}$ recovers the filtered de Rham stack of Simpson, which degenerates the de Rham stack $X^\mrm{dR}$ to the Hodge stack $X^{\mrm{Hodge}} = [BT_X/\Gm]^\wedge$; see the discussion in \sssecref{sssec:GR}.\footnote{%
    We are following the notation and terminology of \cite{BhattLect}.
    In Simpson's paper \cite{Simpson}, the stacks $X^{\mrm{dR},+}$ and $X^{\mrm{Hodge}}$ were denoted $X_{\mrm{Hod}}$ and $X_{\mrm{Dol}}$, respectively.
  }
  We refer to \cite[\S 2.3]{BhattLect} for an introduction to this circle of ideas, which go back to Simpson \cite[\S 5]{Simpson}.
  See also Drinfeld \cite{DrinfeldPrism}, Bhatt--Lurie \cite{BhattLuriePrism}, and \cite[\S 5.3]{BhattLect} for an analogous construction in prismatic cohomology, called the Cartier--Witt stack.
\end{rem}

\subsection{Description as a right adjoint}
\label{ssec:defadj}

The constructions $\cD_{X/Y}$ and $\Dl_{X/Y}$ are clearly functorial in morphisms $X \to Y$.
In particular, given a commutative diagram
\begin{equation}\label{eq:pinken}
  \begin{tikzcd}
    X' \ar{r}{f'}\ar{d}{p}
    & Y' \ar{d}{q}
    \\
    X \ar{r}{f}
    & Y,
  \end{tikzcd}
\end{equation}
there is a canonical $\Gm$-equivariant morphism
\begin{equation*}
  Dq : \Dl_{X'/Y'} \to \Dl_{X/Y}
\end{equation*}
over $Y$.
It factors as the composite
\begin{equation}\label{eq:subdean}
  Dq : \Dl_{X'/Y'} \xrightarrow{dq} \Dl_{X/Y}\fibprod_Y Y' \xrightarrow{q_D} \Dl_{X/Y}
\end{equation}
where $q_D$ is the base change of $q : Y' \to Y$, and the morphism
\begin{equation*}
  dq : \Dl_{X'/Y'} \to \Dl_{X\fibprod_Y Y'/Y'} \simeq \Dl_{X/Y}\fibprod_Y Y'
\end{equation*}
over $Y'$ is induced by functoriality.

\propref{prop:defweil} admits the following reformulation, where $\Arr(\sC)$ denotes the \inftyCat of morphisms in an \inftyCat $\sC$.

\begin{cor}\label{cor:Drightadj}\leavevmode
  \begin{thmlist}
    \item
    The assignment $(X\to Y) \mapsto \cD_{X/Y}$ determines a right adjoint to the functor
    \begin{equation*}
      \dStk_{/\AGm} \to \Arr(\dStk)
    \end{equation*}
    that sends a derived stack $S$ over $\AGm$ to the induced morphism $D := S \fibprod_{\AGm} \BGm \to S$.

    \item
    The assignment $(X\to Y) \mapsto \Dl_{X/Y}$ determines a right adjoint to the functor
    \begin{equation*}
      \dStk_{/\A^1} \to \Arr(\dStk)
    \end{equation*}
    that sends a derived stack $S$ over $\A^1$ to the induced morphism $D := S \fibprod_{\A^1} \{0\} \to S$.

    \item
    The assignment $(X \to Y) \mapsto (\Dl_{X/Y} \to Y\times\A^1)$ determines a right adjoint to the derived zero-fibre functor
    \begin{equation*}
      0^* : \Arr(\dStk_{/\A^1}) \to \Arr(\dStk),
    \end{equation*}
    that sends a morphism $X \to Y$ over $\A^1$ to the induced morphism $X\fibprod_{\A^1} \{0\} \to Y\fibprod_{\A^1} \{0\}$.
  \end{thmlist}
\end{cor}

\begin{cor}
	\label{cor:Dbc}
  The functor $(X \to Y) \mapsto \Dl_{X/Y}$ preserves limits.
\end{cor}

\subsection{Base change properties}

We record the basic properties of the normal deformation.

\begin{prop}\label{prop:dbc}
  For every cartesian square \eqref{eq:pinken}, the $\Gm$-equivariant morphism over $Y \times \A^1$
  \begin{equation*}
    dq : \Dl_{X'/Y'} \to \Dl_{X/Y} \fibprod_Y Y'
  \end{equation*}
  is invertible.
\end{prop}

\begin{cor}\label{cor:bacach}
  Given morphisms $f : X \to Y$ and $g : Y \to Z$, there is a canonical $\Gm$-equivariant isomorphism
  \begin{equation*}
    \Dl_{X/Y} \simeq \Dl_{X/Z} \fibprod_{\Dl_{Y/Z}} Y \times \A^1
  \end{equation*}
  over $\A^1$.
\end{cor}
\begin{proof}
  The square
  \begin{equation*}
    \begin{tikzcd}
      (X \xrightarrow{f} Y) \ar{r}\ar{d}
      & (X \xrightarrow{g\circ f} Z) \ar{d}
      \\
      (Y \xrightarrow{\id} Y) \ar{r}
      & (Y \xrightarrow{g} Z)
    \end{tikzcd}
  \end{equation*}
  is cartesian in $\Arr(\dStk)$.
\end{proof}

\begin{cor}
  For any commutative square \eqref{eq:pinken}, there is a canonical $\Gm$-equivariant cartesian square
  \begin{equation*}
    \begin{tikzcd}
      \Dl_{X'/X\fibprod_Y Y'} \ar{r}\ar{d}
      & \Dl_{X'/Y'} \ar{d}{dq}
      \\
      X\fibprod_Y Y'\times\A^1 \ar{r}
      & \Dl_{X/Y} \fibprod_Y Y'
    \end{tikzcd}
  \end{equation*}
  over $Y' \times \A^1$.
\end{cor}

\begin{proof}
  The bottom arrow is equivalent to $\Dl_{X\times_YY'/X\times_YY'} \to \Dl_{X\times_YY'/Y'}$.
\end{proof}

\subsection{The normal bundle and the conormal complex}

Let $f : X \to Y$ be a morphism of derived stacks. We revisit the universal virtual Cartier divisor over $f$ \eqref{eq:intertrude}.
The \emph{normal bundle} of $f$ is the derived stack
\begin{equation}
  \Nl_{X/Y} := \Dl_{X/Y} \fibprod_{\A^1} \{0\}
  \simeq {0}^*{0}_*(X \xrightarrow{f} Y)
\end{equation}
with its canonical $\Gm$-action, where $0 : Y \to Y \times \A^1$ is the zero section and the isomorphism is \propref{prop:defweil}.

We denote the quotient stack by $\cN_{X/Y} := [\Nl_{X/Y}/\Gm]$, so that
\begin{equation}\label{eq:cN}
  \cN_{X/Y}
  \simeq \cD_{X/Y} \fibprod_{\AGm} \BGm
  \simeq \quo{0}^*\quo{0}_*(X \times \BGm \xrightarrow{f\times\id} Y \times \BGm),
\end{equation}
where $\quo{0} : Y \times \BGm \to Y \times \AGm$ is the zero section.

By construction, the normal bundle is equipped with canonical morphisms
\begin{equation}
  v : \Nl_{X/Y} \to Y,
  \quad \quo{v} : \cN_{X/Y} \to Y \times \BGm.
\end{equation}
Via the adjunction counit $0^*0_* \to \id$, these in fact lift to canonical morphisms
\begin{equation}\label{eq:pi}
  \pi : \Nl_{X/Y} \to X,
  \quad \quo{\pi} : \cN_{X/Y} \to X \times \BGm.
\end{equation}
We also have the inclusions
\begin{equation}
  i_N : \Nl_{X/Y} \hook \Dl_{X/Y},
  \quad i_\cN : \cN_{X/Y} \hook \cD_{X/Y}
\end{equation}
exhibiting $\cN_{X/Y}$ as the tautological virtual Cartier divisor on $\cD_{X/Y}$, i.e., the universal one over $f$.

\begin{defn}
	If $f$ admits a cotangent complex, then the \emph{conormal complex} of $f$ is the shifted cotangent complex $\sNv_{X/Y} := \sL_{X/Y}[-1]$.
\end{defn}

When $f$ admits a cotangent complex, the normal bundle can be described in terms of the conormal complex.
We recall the notation $\V_X(-)$ from Appendix~\ref{sec:vb}.

\begin{prop}\label{prop:normal}
  If $f : X \to Y$ admits a cotangent complex $\sL_{X/Y}$, then there is a canonical $\Gm$-equivariant isomorphism
  \begin{equation*}
    \Nl_{X/Y} \simeq \V_X(\sNv_{X/Y})
  \end{equation*}
  of derived stacks over $X$.
  In particular, there is a canonical isomorphism
  \begin{equation*}
    \cN_{X/Y} \simeq [\V_X(\sNv_{X/Y})/\Gm]
    \simeq \V_{X\times\BGm}(\sNv_{X/Y}(-1))
  \end{equation*}
  over $X \times \BGm$.
\end{prop}
\begin{proof}
  Denote by $0_\sharp$ the left adjoint to $0^*$ \eqref{eq:hsharp}.
  The adjunction formula for the pair $(0^*0_\sharp, 0^*0_*)$ reads, for every derived scheme $S$ and morphism $S \to Y$,
  \begin{equation}\label{eq:Nokbuzak}
    \Maps_Y(S, \Nl_{X/Y})
    \simeq \Maps_Y(S[\sO_S[1]], X),
  \end{equation}
  up to identifying $\Nl_{X/Y}$ with $0^*0_*(X \to Y)$, and the trivial square zero-extension $S[\sO_S[1]]$ (see \ssecref{ssec:convent/cot}) with
  \[
    0^*0_\sharp(S \to Y) \simeq S \times (0 \fibprod_{\A^1} 0).
  \]
  Moreover, the isomorphism \eqref{eq:Nokbuzak} commutes with the respective forgetful maps to $\Maps(S, X)$, where on the left-hand side we post-compose with the counit $0^*0_* \to \id$ and on the right-hand side we pre-compose with the unit $\id \to 0^*0_\sharp$.

  Now given a derived scheme $S$ with a morphism $s : S \to X$, we compute:
  \begin{align*}
    \Maps_X(S, \V_X(\sNv_{X/Y}))
    &\simeq \Maps_{\QCoh(S)}(s^*\sL_{X/Y}, \sO_S[1])\\
    &\simeq \Fib(\Maps_Y(S[\sO_S[1]], X) \to \Maps(S, X))\\
    &\simeq \Fib(\Maps_Y(S, \Nl_{X/Y}) \to \Maps(S, X)),\\
    &\simeq \Maps_X(S, \Nl_{X/Y}).
  \end{align*}
  where both fibres are taken at the point determined by $s : S \to X$.
  The first isomorphism is the universal property of $\V_X(-)$ (see App.~\ref{sec:vb}), the second is the universal property of $\sL_{X/Y}$ \eqref{eq:derivations}, the third is \eqref{eq:Nokbuzak}, and the last is tautological.
  These are functorial as $(S,s)$ varies and we obtain the canonical isomorphism $\Nl_{X/Y} \simeq \V_X(\sNv_{X/Y})$.

  We can take $\Gm$-equivariance into account as follows.
  The $\Gm$-equivariant refinement of \eqref{eq:Nokbuzak} reads
  \begin{equation*}
    \Maps_{Y\times\BGm}(S, \cN_{X/Y})
    \simeq \Maps_{Y\times\BGm}(S[\sO_S(1)[1]], X\times\BGm).
  \end{equation*}
  This implies as above that for every derived scheme $S$ with a morphism $S \to X \times \BGm$, we have
  \begin{equation*}
    \Maps_{X\times\BGm}(S, \V_{X\times\BGm}(\sL_{X/Y} \boxtimes \sO(-1)[-1]))
    \simeq \Maps_{X\times\BGm}(S, \cN_{X/Y}),
  \end{equation*}
  whence the claim.
\end{proof}

Recall the functoriality of the normal deformation from \ssecref{ssec:defadj}.
This gives rise to the following functoriality for the normal bundle.
Given a commutative diagram
\begin{equation*}\label{eq:pinken1}
  \begin{tikzcd}
    X' \ar{r}{f'}\ar{d}{p}
    & Y' \ar{d}{q}
    \\
    X \ar{r}{f}
    & Y,
  \end{tikzcd}
\end{equation*}
there is an induced morphism
\begin{equation}\label{eq:Nq}
  Nq : \Nl_{X'/Y'} \xrightarrow{dq} \Nl_{X/Y}\fibprod_Y Y' \xrightarrow{q_N} \Nl_{X/Y}
\end{equation}
where $q_N$ is the base change of $q : Y' \to Y$ and the morphism
\begin{equation*}
  dq : \Nl_{X'/Y'} \to \Nl_{X\fibprod_Y Y'/Y'} \simeq \Nl_{X/Y}\fibprod_Y Y'
\end{equation*}
over $Y'$ is induced by functoriality.
There is by construction a cartesian square
\begin{equation*}
  \begin{tikzcd}
    \Nl_{X'/Y'} \ar{r}\ar{d}{Nq}
    & \Dl_{X'/Y'} \ar{d}{Dq}
    \\
    \Nl_{X/Y} \ar{r}
    & \Dl_{X/Y}.
  \end{tikzcd}
\end{equation*}

\begin{exam}
  Given morphisms $X \to Y \to Z$, \corref{cor:bacach} yields
  \[
    \Nl_{X/Y}
    \simeq \Nl_{X/Z} \fibprod_{\Nl_{Y/Z}} Y
    \simeq \Nl_{X/Z} \fibprod_{\Nl_{Y/Z} \fibprod_Y X} X.
  \]
  When $X \to Y$ and $Y \to Z$ admit cotangent complexes, this recovers the transitivity triangle for cotangent complexes via \propref{prop:normal}.
\end{exam}

\subsection{The deformation diagram}
\label{ssec:defdiag}

For any morphism $f : X \to Y$, the morphisms
\[(X \xrightarrow{\id} X) \to (X \xrightarrow{f} Y) \to (Y \xrightarrow{\id} Y)\]
induce by functoriality, under the $\Gm$-equivariant isomorphisms $\Dl_{X/X} \simeq X \times \A^1$, $\Dl_{Y/Y} \simeq Y \times \A^1$, canonical $\Gm$-equivariant morphisms
\begin{equation}
  X \times \A^1
  \xrightarrow{Df} \Dl_{X/Y}
  \to Y \times \A^1
\end{equation}
over $\A^1$.

Taking fibres over $0$ and $\Gm$, respectively, we get a commutative diagram of cartesian squares of derived stacks with $\Gm$-action
\begin{equation}\label{eq:defdiag}
  \begin{tikzcd}
    X \ar{r}{0}\ar{d}{0}
    & X \times \A^1 \ar[leftarrow]{r}\ar{d}{Df}
    & X \times \Gm \ar{d}{f\times \id}
    \\
    \Nl_{X/Y} \ar{r}{i_N}\ar{d}{v}
    & \Dl_{X/Y} \ar[leftarrow]{r}\ar{d}{u}
    & Y \times \Gm \ar[equals]{d}
    \\
    Y \ar{r}{0}
    & Y \times \A^1 \ar[leftarrow]{r}
    & Y \times \Gm.
  \end{tikzcd}
\end{equation}
In the left-hand column we have used the canonical isomorphisms $\Nl_{X/X} \simeq X$ and $\Nl_{Y/Y} \simeq Y$; recall also that $v$ factors through the projection $\pi : \Nl_{X/Y} \to X$ \eqref{eq:pi}.
In the right-hand column we use \propref{prop:defweil} and base change for the Weil restriction (\lemref{lem:mouthiness}).

Taking quotient stacks, we get the commutative diagram of cartesian squares

\begin{equation}\label{eq:defdiagq}
  \begin{tikzcd}
    X\times\BGm \ar{r}{\quo{0}}\ar{d}{\quo{0}}
    & X \times \AGm \ar[leftarrow]{r}\ar{d}{\cD f}
    & X \ar{d}{f}
    \\
    \cN_{X/Y} \ar{r}{i_\cN}\ar{d}{\quo{v}}
    & \cD_{X/Y} \ar[leftarrow]{r}\ar{d}{\quo{u}}
    & Y \ar[equals]{d}
    \\
    Y\times\BGm \ar{r}{\quo{0}}
    & Y\times\AGm \ar[leftarrow]{r}
    & Y,
  \end{tikzcd}
\end{equation}
where we have used the decoration $\quo{}$ to indicate the induced morphism on $\Gm$-quotients, and $\cD f := \quo{Df}$.

\subsection{Cotangent complex}
\label{ssec:def/cot}

Let $f : X \to Y$ be a morphism of derived stacks.
Specializing \corref{cor:ontogenic}, we obtain the following description of the cotangent complex of the normal deformation:

\begin{prop}\label{prop:defcot}
  If $f$ admits a perfect cotangent complex $\sL_{X/Y}$, then the canonical morphisms
  \begin{align}
    \sL_{\cD_{X/Y}/Y\times\AGm}
    &\to i_{\cN,*}(\sL_{\cN_{X/Y}/X\times\BGm}),\label{eq:defcot1}\\
    \sL_{\cD_{X/Y}/Y}
    &\to i_{\cN,*}(\sL_{\cN_{X/Y}/X}).\label{eq:defcot2}
  \end{align}
  are invertible.
\end{prop}
\begin{proof}
  The projection $\quo{\pi} : \cN_{X/Y} \to X\times\BGm$  \eqref{eq:pi} is the derived vector bundle associated with the perfect complex $\sL_{X\times\BGm/Y\times\BGm}(-1)[-1]$ (\propref{prop:normal}), so we have
  \begin{equation*}
    \sL_{\cN_{X/Y}/X\times\BGm} \simeq \quo{\pi}^*\sL_{X\times\BGm/Y\times\BGm}(-1)[-1].
  \end{equation*}
  Under this identification, \eqref{eq:defcot1} is the isomorphism
  \begin{equation*}
    \sL_{\cD_{X/Y}/Y\times\AGm}
    \simeq i_{\cN,*} (\quo{\pi}^*\sL_{X\times\BGm/Y\times\BGm}(-1)[-1]),
  \end{equation*}
  obtained by applying \corref{cor:ontogenic} for the Weil restriction of $\quo{f} : X \times \BGm \to Y \times \BGm$ along the inclusion $\quo{0} : Y \times \BGm \hook Y \times \AGm$.
  Note that $\quo{0}_\sharp(-) \simeq \quo{0}_*(- \otimes \sO(-1)[-1])$ by \examref{exam:sharpvcd}, since the dualizing complex $\omega_{\BGm/\AGm}$ is given by $\sO(-1)[-1]$  since we regard $\AGm$ as the dual of the tautological line bundle over $\BGm$.
  
  For the second isomorphism \eqref{eq:defcot2}, consider the commutative diagram
  \[\begin{tikzcd}
    & \cN_{X/Y} \ar{r}{i_\cN}\ar{d}{\quo{v}}\ar[swap]{ld}{\quo{\pi}}
    & \cD_{X/Y} \ar{d}{\quo{u}}
    \\
    X\times\BGm \ar{r}{\quo{f}}\ar{d}
    & Y\times\BGm \ar{r}{\quo{0}}\ar{d}{q}
    & Y \times \AGm \ar{ld}{p}
    \\
    X \ar{r}{f}
    & Y
    &
  \end{tikzcd}\]
  where the notation is as in \eqref{eq:defdiagq}, with $p$ and $q$ the projections, and both squares are cartesian.
  We have a morphism of exact triangles
  \[\begin{tikzcd}
    \quo{u}^* \sL_{Y\times\AGm/Y} \ar{r}\ar{d}
    & \sL_{\cD_{X/Y}/Y} \ar{r}\ar{d}{\eqref{eq:defcot2}}
    & \sL_{\cD_{X/Y}/Y\times\AGm}\ar{d}{\eqref{eq:defcot1}}
    \\
    i_{\cN,*}\quo{\pi}^*(\sL_{X\times\BGm/X}) \ar{r}
    & i_{\cN,*}(\sL_{\cN_{X/Y}/X}) \ar{r}
    & i_{\cN,*}(\sL_{\cN_{X/Y}/X\times\BGm}).
  \end{tikzcd}\]
  The left-hand vertical arrow is the isomorphism
  \begin{align*}
    i_{\cN,*}\quo{\pi}^*(\sL_{X\times\BGm/X})
    &\simeq i_{\cN,*}\quo{\pi}^*\quo{f}^*(\sL_{Y\times\BGm/Y})\\
    &\simeq \quo{u}^*\quo{0}_*(\sL_{Y\times\BGm/Y})\\
    &\simeq \quo{u}^*\sL_{Y\times\AGm/Y},
  \end{align*}
  where the last identification follows from \eqref{eq:LAGm}.
  It follows that \eqref{eq:defcot2} is also invertible, as claimed.
\end{proof}

\subsection{Closed immersions}
\label{ssec:Rees}

Recall that for a closed immersion of classical algebraic stacks $i : X \to Y$ with ideal sheaf $\sI$, the classical Rees algebra and its extended version are the $\N$-graded and $\Z$-graded $\sO_Y$-algebras
\begin{equation}\label{eq:Rcl}
  {}^\heartsuit\R_{X/Y} := \bigoplus_{n\ge 0} \sI^n \cdot t^n, \qquad
  {}^\heartsuit\Rext_{X/Y} := \bigoplus_{n\in\Z} \sI^n \cdot t^n,
\end{equation}
respectively, where $\sI^n := \sO_Y$ for $n\leq 0$.
The relative spectrum of ${}^\heartsuit\Rext_{X/Y}$ defines the classical deformation to the normal cone.

For a closed immersion of derived stacks $f : X \to Y$, we use the normal deformation to define analogues of \eqref{eq:Rcl}.
We begin with the following affineness result:\footnote{%
  If $f : X \to Y$ is just an affine morphism, one can show that $u : \Dl_{X/Y} \to Y \times \A^1$ is still affine in the sense of \emph{nonconnective} derived algebraic geometry; see \cite[Thm.~4.6]{BenBassatHekking}.
}

\begin{thm}[Hekking]\label{thm:Daff}
  For any closed immersion of derived stacks $f : X \hook Y$, the morphism $u : \Dl_{X/Y} \to Y\times\A^1$ is affine.
  In particular, the canonical $\Gm$-equivariant morphism
  \[
    \Dl_{X/Y} \to \uSpec_{Y\times\A^1}(u_*\sO_{\Dl_{X/Y}})
  \]
  is invertible.
\end{thm}
\begin{proof}
  Since $\Dl_{X/Y}$ is stable under base change in $Y$, we may reduce to the case where $Y$ is a derived scheme, treated in \cite[Thm.~6.3.3]{Hekking}.
\end{proof}

In the notation of \eqref{eq:defdiag}, the structure sheaf of $\Dl_{X/Y}$ determines a canonical $\Gm$-equivariant quasi-coherent $\sO_{Y\times\A^1}$-algebra\footnote{\label{fn:qcohalg}%
  A quasi-coherent $\sO_S$-algebra, where $S=\Spec(R)$ is affine, is an animated $R$-algebra.
  One extends the notion to derived stacks by limits, as in \ssecref{ssec:convent/qcoh}.
  The assignment sending a relatively affine $(f : X \to S)$ to the quasi-coherent $\sO_S$-algebra $f_*(\sO_X)$ determines an equivalence of \inftyCats, inverse to the relative spectrum functor $\sA \mapsto \uSpec_S(\sA)$.
  There exists a nonconnective version of this notion (see \cite[\S 4]{Raksit}), but we will not use it here.
}
$u_*(\sO_{\Dl_{X/Y}})$
by push-forward along $u : \Dl_{X/Y} \to Y \times \A^1$.
This may be regarded equivalently as a $\Z$-graded quasi-coherent $\sO_Y[t^{-1}]$-algebra, where $t^{-1}$ denotes the coordinate of $\A^1$.\footnote{See \cite[\S 3]{Hekking} or \cite[Prop.~3.39]{BenBassatHekking} for the equivalence in the derived setting between $\Z$-graded quasi-coherent algebras and $\Gm$-equivariant quasi-coherent algebras.}

\begin{defn}\label{defn:Rees}
  The \emph{extended Rees algebra} $\Rext_{X/Y}$ is the $\Z$-graded quasi-coherent $\sO_Y[t^{-1}]$-algebra determined by the $\Gm$-equivariant quasi-coherent $\sO_{Y\times\A^1}$-algebra $u_*(\sO_{\Dl_{X/Y}})$.
\end{defn}

We denote by $\R_{X/Y} := (\Rext_{X/Y})_{\ge 0}$ the $\N$-graded algebra consisting of only the components of nonnegative degrees.
When $X = \Spec(B)$ and $Y = \Spec(A)$ are affine, we also write $\Rext_{B/A}$ and $\sR_{B/A}$.

\begin{prop}
  Let $f : X \hook Y$ be a closed immersion of derived stacks.
  Then we have:
  \begin{thmlist}
    \item The canonical homomorphism $\sO_Y[t^{-1}] \to \Rext_{X/Y}$ is an isomorphism in degrees $\le 0$.
    \item The $\N$-graded quasi-coherent $\pi_0(\sO_Y)$-algebra $\pi_0(\R_{X/Y})$ is generated in degree $1$.
    \item Let $\sI\in\QCoh(Y)$ denote the fibre of $\sO_Y \to f_*(\sO_X)$.
    Then the canonical morphism $(\R_{X/Y})_1 \simeq (\Rext_{X/Y})_1  \to \sO_Y$ induces an isomorphism $(\R_{X/Y})_1 \simeq \sI$.
  \end{thmlist}
\end{prop}
\begin{proof}
  When $Y$ (and hence $X$) is a derived scheme, these properties were proven in \cite[\S 6]{Hekking}.
  Since the claims are local on $Y$, the general case follows.
\end{proof}

\begin{cor}
  For any closed immersion of derived stacks $f : X \hook Y$, there are canonical isomorphisms
  \begin{align}
    \Rext_{X/Y} \otimes_{\sO_Y[t^{-1}]} \sO_Y[t,t^{-1}] &\simeq \sO_Y[t,t^{-1}],\label{eq:underlying}\\
    \Rext_{X/Y} \otimes_{\sO_Y[t^{-1}]} \sO_Y[t^{-1}]/(t^{-1}) &\simeq f_*\Sym^*_{\sO_X}(\sNv_{X/Y}).\label{eq:assgr}
  \end{align}
  of $\Z$-graded algebras over $\sO_Y[t,t^{-1}]$ and $\sO_Y[t^{-1}]/(t^{-1}) \simeq \sO_Y$, respectively.
\end{cor}
\begin{proof}
  Since $u$ is affine (\thmref{thm:Daff}), $u_*$ in particular satisfies the base change formula (see \cite[Cor.~3.4.2.2]{LurieSAG}).
  The claimed isomorphisms are the base change isomorphisms for the cartesian squares in \ssecref{ssec:defdiag}.
\end{proof}
  
\subsection{Algebraicity}
\label{ssec:def/alg}

Our main result about the normal deformation reads as follows:

\begin{thm}\label{thm:Dart}
  Let $f : X \to Y$ be a morphism, locally of finite type, between derived stacks.
  If $f$ is $n$-representable, then the normal deformation $\Dl_{X/Y}$ is $(n+1)$-representable over $Y$.
  In particular, if $X$ and $Y$ are Artin, then so is $\Dl_{X/Y}$.
\end{thm}

We will prove \thmref{thm:Dart} in \ssecref{ssec:proofDart} as an application of \thmref{thm:algvcd}.

\subsection{Functoriality: finite presentation, smoothness and surjectivity}
\label{ssec:functsm}

Given a commutative diagram
\begin{equation}\label{eq:pinken2}
  \begin{tikzcd}
    X' \ar{r}{f'}\ar{d}{p}
    & Y' \ar{d}{q}
    \\
    X \ar{r}{f}
    & Y,
  \end{tikzcd}
\end{equation}
the deformation diagram \eqref{eq:defdiag} for $f'$ factors through a $\bG_m$-equivariant cartesian diagram
\begin{equation}\label{eq:statocyst}
  \begin{tikzcd}
    \Nl_{X'/Y'}\ar{d}{Nq}\ar{r}{i'_D}
    & \Dl_{X'/Y'}\ar{d}{Dq}\ar[leftarrow]{r}{j'_{D}}
    & Y' \times \bG_m \ar{d}{q\times \id}
    \\
    \Nl_{X/Y} \ar{r}{i_D}
    & \Dl_{X/Y} \ar[leftarrow]{r}{j_D}
    & Y \times \bG_m.
  \end{tikzcd}
\end{equation}

In this subsection we will study some properties of the morphism $\Dl_{X'/Y'} \to \Dl_{X/Y}$ that are inherited from those of $p$, $q$ and/or $Nq$.
We recall the factorization
\begin{equation*}
  Dq : \Dl_{X'/Y'} \xrightarrow{dq} \Dl_{X/Y}\fibprod_Y Y' \xrightarrow{q_D} \Dl_{X/Y}
\end{equation*}
from \eqref{eq:subdean}.

\begin{prop}[Finite presentation]\label{prop:Dafp}
  If $p : X' \to X$ and $q : Y' \to Y$ are locally afp (resp. locally hfp), then so is $Dq : \Dl_{X'/Y'} \to \Dl_{X/Y}$.
\end{prop}
\begin{proof}
  If $q$ is locally afp (resp. hfp), then so is its base change $q_D$.
  Hence it is enough to consider $dq : \Dl_{X'/Y'} \to \Dl_{X/Y}\fibprod_Y Y' \simeq \Dl_{X\fibprod_Y Y'/Y'}$.
  Replacing \eqref{eq:pinken2} by the square
  \[\begin{tikzcd}
    X' \ar{r}{f'}\ar[swap]{d}{(p,f')}
    & Y' \ar[equals]{d}
    \\
    X \fibprod_Y Y' \ar{r}{f \fibprod_Y Y'}
    & Y',
  \end{tikzcd}\]
  where the left-hand vertical arrow is locally afp (resp. hfp) because $X'$ and $X \fibprod_Y Y'$ are both locally afp (resp. hfp) over $X$, we may assume that $Y'=Y$ and show that $\Dl_{X'/Y} \to \Dl_{X/Y}$ is locally afp (resp. hfp).
  By definition, this is the morphism of Weil restrictions $0_*(p) : 0_*(X' \to Y) \to 0_*(X \to Y)$, where $0 : Y \to Y \times \A^1$.
  Since $p : X' \to X$ is locally afp (resp. hfp), the claim follows from \propref{prop:whalefin}.
\end{proof}

\begin{prop}\label{prop:wire}
  Suppose $p : X' \to X$ is locally hfp and $q : Y' \to Y$ is invertible.
  If $p$ is an effective epimorphism in the étale topology, then so is $Dq : \Dl_{X'/Y'} \to \Dl_{X/Y}$.
  In particular, it is surjective.
\end{prop}
\begin{proof}
  Without loss of generality we may assume $q = \id_Y$.
  By definition, $Dq$ is the morphism of Weil restrictions $0_*(p) : 0_*(X' \to Y) \to 0_*(X \to Y)$, where $0 : Y \to Y \times \A^1$.
  Hence it is an effective epimorphism by \propref{prop:unconversableness}.
\end{proof}

\begin{prop}\label{prop:Damp}
  If $p : X' \to X$ and $q : Y' \to Y$ are locally afp, and $q$ and $Nq : \Nl_{X'/Y'} \to \Nl_{X/Y}$ have relative cotangent complexes perfect of Tor-amplitude $[a,b]$, then $Dq : \Dl_{X'/Y'} \to \Dl_{X/Y}$ is locally afp and has relative cotangent complex perfect of Tor-amplitude $[a,b]$.
  In particular, if $q$ and $Nq$ are smooth (resp. quasi-smooth), then so is $Dq$.
\end{prop}
\begin{proof}
  By \propref{prop:Dafp}, $Dq$ is locally afp.
  In particular, its relative cotangent complex $\sL_{Dq}$ is pseudo-coherent (see \cite[Cor.~17.4.2.2]{LurieSAG}).
  We may thus apply \cite[Cor.~6.1.4.7]{LurieSAG} to conclude that $\sL_{Dq}$ is of Tor-amplitude $\le b$, since it is fibrewise so by \eqref{eq:statocyst} and the assumptions.
  By \cite[Prop.~7.2.4.23(4)]{LurieHA} we conclude that $\sL_{Dq}$ is perfect.
  Applying \cite[Cor.~6.1.4.7]{LurieSAG} again, we similarly conclude that the dual of $\sL_{Dq}$ is of Tor-amplitude $\le -a$ since it is fibrewise so.
  Thus, $\sL_{Dq}$ is of Tor-amplitude $[a,b]$.
\end{proof}

\begin{rem}[Smoothness]\label{rem:Nqsm}
  If $p$ and $q$ are both smooth, then so is $Nq : \Nl_{X'/Y'} \to \Nl_{X/Y}$.
  In particular, \propref{prop:Damp} implies that $Dq$ is also smooth in this case.
  To see this, recall that $Nq$ is the composite
  \[
    \Nl_{X'/Y'} \to \Nl_{X/Y} \fibprod_X X' \to \Nl_{X/Y}.
  \]
  As a base change of $p$, the second arrow is smooth.
  The first arrow is a morphism of derived vector bundles over $X'$ with fibre $\Nl_{X'/X\fibprod_Y Y'}$; since $X' \to X \fibprod_Y Y'$ is quasi-smooth (as source and target are both smooth over $X$), the latter is a vector bundle stack, hence smooth over $X'$.
\end{rem}

\begin{prop}[Surjectivity]\label{prop:DsurjY}
  Suppose that $f$ and $f'$ are locally of finite type and eventually representable.
  If $p$ and $q$ are smooth, surjective, and eventually representable, then so is $Dq : \Dl_{X'/Y'} \to \Dl_{X/Y}$.
\end{prop}
\begin{proof}
  Since $Dq$ is smooth by \remref{rem:Nqsm} and eventually representable by \thmref{thm:Dart}, it will suffice to show that it is surjective, or equivalently, that $Dq$ is surjective on $\kappa$-valued points where $\kappa$ is separably closed.
  Let $S=\Spec(\kappa)$ be the spectrum of a separably closed field and let
  $y : S \to \Dl_{X/Y}$ be a morphism. This corresponds to a virtual Cartier divisor 
  \[\begin{tikzcd}
    D \ar{r}{i_D}\ar{d}{x}
    & S \ar{d}{y}
    \\
    X \ar{r}{f}
    & Y
  \end{tikzcd}\]
  over $f : X \to Y$. It suffices to show that there exists:
  \begin{enumerate}
    \item a morphism $x' : D \to X'$ lifting $x$,
    \item a morphism $y' : S \to Y'$ lifting $y$,
    \item a virtual Cartier divisor over $f' : X' \to Y'$ lifting $i_D : D \hook S$.
  \end{enumerate}
  Note that $D$ is either empty or the trivial square-zero extension $\Spec(\kappa\oplus\kappa[1])$, depending on whether it is cut out by a unit or $0 \in \sO_S$.
  In the former case, the claim follows from the surjectivity of $q : Y' \twoheadrightarrow Y$ on $\kappa$-valued points.
  Otherwise, we argue as follows.
  The point $x \in X(\kappa\oplus\kappa[1])$ restricts to a $\kappa$-valued point $x_0 \in X(\kappa)$, which by surjectivity of $p : X'(\kappa) \to X(\kappa)$ lifts to $x'_0 \in X'(\kappa)$. Put $y' \coloneqq f' \circ x'_0 \in Y'(\kappa)$. Write $z : S \to D$ for the canonical section of $i_D$.
  Since $p$ is smooth and eventually representable we may then find a unique diagonal lift as in the following diagram:
  \[\begin{tikzcd}
    S = \Spec(\kappa) \ar{r}{x'_0}\ar{d}{z}
    & X' \ar[twoheadrightarrow]{d}{p}
    \\
    D = \Spec(\kappa\oplus\kappa[1]) \ar{r}{x}\ar[dashed]{ru}{x'}
    & X
  \end{tikzcd}\]
  by the universal property of the cotangent complex of $p$.
  Now in the diagram
  \[\begin{tikzcd}
  S = \Spec(\kappa) \ar{r}{y'}\ar{d}{z}
  & Y' \ar[twoheadrightarrow]{d}{q}
  \\
  D = \Spec(\kappa\oplus\kappa[1]) \ar{r}{f \circ x}\ar[dashed]{ru}{\delta}
  & Y,
  \end{tikzcd}\]  
  both $\delta=y' \circ i_D$ and $\delta = f' \circ x'$ are possible diagonal lifts.
  By uniqueness of such lifts, there exists a homotopy $y' \circ i_D \simeq f' \circ x'$.
  We thus have a commutative square
  \[\begin{tikzcd}
    D \ar{r}{i_D}\ar{d}{x'}
    & S \ar{d}{y'}
    \\
    X' \ar{r}{f'}
    & Y'
  \end{tikzcd}\]
  which defines the desired virtual Cartier divisor over $f'$.
\end{proof}

\begin{rem}\label{rem:DsurjY}
  \propref{prop:DsurjY} gives a useful description of the normal deformation.
  Namely, with notation as in the statement, $\Dl_{X/Y}$ is the geometric realization of the simplicial diagram
  \[
    \cdots \rightrightrightarrows \Dl_{X'\fibprod_X X' / Y'\fibprod_Y Y'} \rightrightarrows \Dl_{X'/Y'}.
  \]
  Indeed, we have $\Dl_{X'\fibprod_X X' / Y'\fibprod_Y Y'} \simeq \Dl_{X'/Y'} \fibprod_{\Dl_{X/Y}} \Dl_{X'/Y'}$ (and similarly for $n$-fold fibred products) and $Dq : \Dl_{X'/Y'} \to \Dl_{X/Y}$ is smooth and surjective, hence an effective epimorphism.
  Moreover, if $X$ and $Y$ are derived $n$-Artin stacks, we may choose $X' \to Y'$ to be a closed immersion of schemes (\lemref{lem:acturience} below), in which case $\Dl_{X'/Y'}$ can be described as in \ssecref{ssec:Rees} and the other terms in the simplicial presentation are normal deformations of $(n-1)$-Artin stacks.
\end{rem}

\subsection{Functoriality: properness and excessive squares}
\label{ssec:functprop}

We again suppose given a commutative square
\begin{equation}\label{eq:pinkenprop}
  \begin{tikzcd}
    X' \ar{r}{f'}\ar{d}{p}
    & Y' \ar{d}{q}
    \\
    X \ar{r}{f}
    & Y,
  \end{tikzcd}
\end{equation}
and preserve the notation of \ssecref{ssec:functsm}.

To study the functoriality of $\Dl_{X/Y}$ under proper morphisms, we introduce the notion of \emph{excessive} squares:

\begin{defn}\label{defn:excessive}
  A commutative square \eqref{eq:pinkenprop} is \emph{excessive} if the following conditions hold:
  \begin{defnlist}
    \item\label{item:excessive/cart}
    The square is cartesian on classical truncations; that is, the morphism $X \to X \fibprod_Y Y'$ induces an isomorphism on classical truncations.

    \item\label{item:excessive/norm}
    The morphism of normal bundles $dq : \Nl_{X'/Y'} \to \Nl_{X/Y} \fibprod_X X'$ is a closed immersion.\footnote{\label{fn:Nsurj}%
      When $f$ and $f'$ admit cotangent complexes, this is equivalent to requiring that the morphism of conormal complexes $p^*\sNv_{X/Y} \twoheadrightarrow \sNv_{X'/Y'}$ has connective fibre.
      When $f$ and $f'$ are formally unramified in the sense that $\sL_f$ and $\sL_{f'}$ are $1$-connective (e.g., when $f$ and $f'$ are closed immersions), then this is equivalent to surjectivity of $p^*\sNv_{X/Y} \twoheadrightarrow \sNv_{X'/Y'}$ on $\pi_0$.
    }
  \end{defnlist}
\end{defn}

The terminology is inspired by intersection theory: these conditions are precisely those appearing in the virtual excess intersection formula of \cite{excess}.
This notion will make another appearance in \secref{sec:bl} below.

The condition that \eqref{eq:pinkenprop} is excessive may be reformulated in terms of the induced morphism $e : X' \to X \times_Y Y'$.
Note that when $e$ is a closed immersion (as is necessarily the case when \eqref{eq:pinkenprop} is excessive), the fibre $\sI_e$ of $e^\sharp : \sO_{X\times_Y Y'} \to e_*(\sO_{X'})$ is connective (as it is surjective on $\pi_0$), as is the conormal complex $\sNv_{e} = \sL_{e}[-1]$.

\begin{lem}[{\cite[Rmk.~4.1.3]{blowups}}]\label{lem:excvcd}
  Suppose given a square \eqref{eq:pinkenprop} where $f$ and $f'$ admit cotangent complexes.
	The following conditions are equivalent:
	\begin{thmlist}
    \item\label{item:excvcd/1} The square \eqref{eq:pinkenprop} is excessive;
		\item\label{item:excvcd/2} $e$ is a closed immersion, $e^\sharp$ is bijective on $\pi_0$, and $\sNv_e$ is $1$-connective;
		\item\label{item:excvcd/3} $e$ is a closed immersion and $\sI_e$ is $1$-connective; equivalently, $e$ is a closed immersion, and $e^\sharp$ is bijective on $\pi_0$ and surjective on $\pi_1$.
	\end{thmlist}
\end{lem}
\begin{proof}
  Note first that part \itemref{item:excessive/cart} of \defnref{defn:excessive} is equivalent to the condition that $e$ is a closed immersion with $e^\sharp$ bijective on $\pi_0$.
  By Footnote~\ref{fn:Nsurj} and the exact triangle $e^*\sNv_{X\times_Y Y'/Y'} \simeq p^*\sNv_{X/Y} \to \sNv_{X'/Y'} \to \sNv_e$, part \itemref{item:excessive/norm} of \defnref{defn:excessive} is equivalent to the condition that $\sNv_e$ is $1$-connective.
  This shows \itemref{item:excvcd/1} $\Leftrightarrow$ \itemref{item:excvcd/2}.

	When $e$ is a closed immersion, we have the Hurewicz isomorphism $\pi_0(\sNv_e) \simeq \pi_1(\sL_e) \simeq \pi_0(e^*\sI_e)$ (see \cite[Prop.~25.3.6.1]{LurieSAG}).
  When $e^\sharp$ is bijective on $\pi_0$, the last term is further isomorphic to $\pi_0(\sI_e)$.
	This shows \itemref{item:excvcd/2} $\Leftrightarrow$ \itemref{item:excvcd/3}.
\end{proof}

\begin{exam}
  If \eqref{eq:pinkenprop} is cartesian, then it is excessive (as in that case $dq$ is invertible).
\end{exam}

\begin{exam}
  If \eqref{eq:pinkenprop} is a commutative square of \emph{classical} Artin stacks, then it is excessive if and only if it is cartesian in the ordinary sense.
  Indeed, in this case $e_*(\sO_{X'})$ is discrete and condition~\itemref{item:excvcd/3} of \lemref{lem:excvcd} collapses to the requirement that $e^\sharp$ is bijective on $\pi_0$.
\end{exam}

\begin{exam}
  When every morphism in \eqref{eq:pinkenprop} admits a cotangent complex, \lemref{lem:excvcd} shows that the condition of excessivity is invariant under the reflection swapping $f$ with $q$ and $f'$ with $p$. 
\end{exam}

\begin{exam}\label{exam:selfint}
  For closed immersions $Z \hook X \hook Y$, the square
  \[\begin{tikzcd}
    Z \ar{r}\ar[equals]{d}
    & X \ar{d}
    \\
    Z \ar{r}
    & Y
  \end{tikzcd}\]
  is excessive.
  In particular, for any closed immersion $Z \hook X$, the square
  \[\begin{tikzcd}
    Z \ar[equals]{r}\ar[equals]{d}
    & Z \ar{d}
    \\
    Z \ar{r}
    & X
  \end{tikzcd}\]
  is excessive.
\end{exam}

\begin{rem}\label{rem:twoexc}
  Suppose given a commutative diagram
  \begin{equation*}
    \begin{tikzcd}
    	X'' \ar{r} \ar{d}{p'} & Y'' \ar{d}{q'} \\
    	X' \ar{r} \ar{d}{p} & Y' \ar{d}{q} \\ 
    	X \ar{r} &  Y  	
    \end{tikzcd}
  \end{equation*}
  where the bottom square is excessive.
  Then the outer square is excessive if and only if the top square is excessive.
  Indeed, the analogous statement for condition~\itemref{item:excessive/cart} of \defnref{defn:excessive} is clear, and for condition~\itemref{item:excessive/norm} it follows from the factorization
  \[
    d(q\circ q') : N_{X''/Y''}
    \xrightarrow{dq'} N_{X'/Y'} \fibprod_{X'} X''
    \xrightarrow{dq} N_{X/Y} \fibprod_X X' \fibprod_{X'} X''
    \simeq N_{X/Y} \fibprod_X X''.
  \]
\end{rem}

\begin{prop}\label{prop:Dqprop}
  If $f : X \to Y$ and $f' : X'\to Y'$ are locally of finite type and eventually representable, $q : Y' \to Y$ is proper (resp.\ finite, resp.\ a closed immersion) and the square \eqref{eq:pinkenprop} is excessive, then $Dq : \Dl_{X'/Y'} \to \Dl_{X/Y}$ is proper (resp.\ finite, resp.\ a closed immersion).
\end{prop}
\begin{proof}
  Since $q$ is proper (resp.\ finite, resp.\ a closed immersion), it will suffice using the factorization \eqref{eq:subdean} to show that $dq : \Dl_{X'/Y'} \to \Dl_{X/Y}\fibprod_Y Y' \simeq \Dl_{X\fibprod_Y Y'/Y'}$ is a closed immersion.
  The square \eqref{eq:pinkenprop} factors through
  \begin{equation}\label{eq:pinken'}
    \begin{tikzcd}
      X' \ar{r}{f'}\ar[swap]{d}{(p,f')}
      & Y' \ar[equals]{d}
      \\
      X \fibprod_Y Y' \ar{r}{f \fibprod_Y Y'}
      & Y'
    \end{tikzcd}
  \end{equation}
  and the base change square, so \eqref{eq:pinkenprop} is excessive if and only if \eqref{eq:pinken'} is.
  Thus we may replace \eqref{eq:pinkenprop} by \eqref{eq:pinken'} and thereby assume that $Y'=Y$ and $q=\id_Y$.

  The question is local on $Y$, so we may assume that $Y$ is affine and $X$ is $n$-Artin.
  The question is also local on $X$; let $u : U \twoheadrightarrow X$ be a smooth surjection from a scheme $U$, and write $u' : U' \twoheadrightarrow X'$ for its base change along $p : X' \to X$.
  Consider the induced commutative square
  \[\begin{tikzcd}
    \Dl_{U'/Y} \ar[twoheadrightarrow]{r}\ar{d}
    & \Dl_{X'/Y} \ar{d}
    \\
    \Dl_{U/Y} \ar[twoheadrightarrow]{r}
    & \Dl_{X/Y}.
  \end{tikzcd}\]
  This square is cartesian by \corref{cor:Dbc} and the horizontal arrows are smooth and surjective
  by \propref{prop:DsurjY}. It is thus enough to prove that $\Dl_{U'/Y}\to \Dl_{U/Y}$ is a closed immersion. We may therefore work smooth-locally on $X$.

  In particular, we may assume that $f : X \to Y$ factors through a closed immersion $i : X \hook M$ and a smooth surjective morphism $g : M \to Y$.
  The square \eqref{eq:pinkenprop} then factors as
  \begin{equation*}
    \begin{tikzcd}
      X' \ar{r}{ip}\ar{d}{p}
      & M \ar{r}{g}\ar[equals]{d}
      & Y \ar[equals]{d}
      \\
      X \ar{r}{i}
      & M \ar{r}{g}
      & Y.
    \end{tikzcd}
  \end{equation*}
  Note that the left-hand square is excessive by \remref{rem:twoexc}.
  As before, the commutative square
  \[\begin{tikzcd}
  	\Dl_{X'/M} \ar[twoheadrightarrow]{r}\ar{d}
	  & \Dl_{X'/Y} \ar{d}
    \\
    \Dl_{X/M} \ar[twoheadrightarrow]{r}
    & \Dl_{X/Y}
	\end{tikzcd}\]
  is cartesian by \corref{cor:Dbc}, and the horizontal arrows are smooth and surjective by \propref{prop:DsurjY}, so we may assume that $f : X \to Y$ is a closed immersion.
  
  Since \eqref{eq:pinkenprop} is cartesian on classical truncations, it follows that $f' : X' \to Y$ also is a closed immersion.
  Write $A \twoheadrightarrow B$ and $A \twoheadrightarrow B'$ for the surjections of animated rings corresponding to $f$ and $f'$, respectively.
  We have $\Dl_{X/Y} = \Spec(\Rext_{B/A})$, where $\Rext_{B/A}$ is the extended Rees algebra (\ssecref{ssec:Rees}), and similarly for $\Dl_{X'/Y}$, so the claim is that $\Rext_{B/A} \to \Rext_{B'/A}$ is surjective on $\pi_0$.
  Let $t^{-1}$ denote the coordinate of $\A^1$.
  We recall from \cite[\S 6.10]{Hekking} that $\Rext_{B/A}$ is a $\Z$-graded $A[t^{-1}]$-algebra with
  \begin{align*}
    (\Rext_{B/A})_{\le 0} &\simeq A[t^{-1}],\\
    (\Rext_{B/A})_{\ge 0} &\simeq \sR_{B/A},
  \end{align*}
  where $\sR_{B/A}$ is the ``unextended'' Rees algebra, whose degree $1$ homogeneous component $(\sR_{B/A})_1 = \sI_{B/A}$ is the fibre of $A \to B$.
  Since $\pi_0(\sR_{B'/A})$ is generated in degree $1$, it suffices to show that the $A$-module map $\sI_{B/A} \to \sI_{B'/A}$ is surjective on $\pi_0$. But $\Fib(\sI_{B/A} \to \sI_{B'/A})=\Fib(B\to B')[-1]$ which is $0$-connective by Lemma~\ref{lem:excvcd}\itemref{item:excvcd/3}. The result follows.
\end{proof}

\subsection{Deformation to the (intrinsic) normal cone}
\label{ssec:defcl}

Let $i: X \hook Y$ be a closed immersion of derived schemes, inducing the closed immersion $i_\cl : X_\cl \hook Y_\cl$ on classical truncations.
In this paragraph we compare the classical truncation of the normal deformation $\Dl_{X/Y}$ with the classical deformation to the normal cone of $i_\cl$.
Recall from \cite[\S 5.1]{Fulton} that the latter is a family
\[ X_\cl \times \A^1 \to \sfD_{X_\cl/Y_\cl} \to \A^1 \]
degenerating $i_\cl : X_\cl \hook Y_\cl$ to the zero section $0 : X_\cl \to \sfC_{X_\cl/Y_\cl}$ of the normal cone.
Note that it need not agree with the classical truncation of $\Dl_{X/Y}$ even when $X=X_\cl$ and $Y=Y_\cl$ are classical\footnote{%
  Unless $i_\cl : X_\cl \to Y_\cl$ is quasi-smooth---which, since $X_\cl$ and $Y_\cl$ are classical, means precisely that $i_\cl$ is a local complete intersection morphism in the sense of \cite[Exp.~VIII, \S 1, D\'ef.~1.1]{SGA6}.
}.

More generally, for any morphism of classical (higher) Artin stacks $f : X \to Y$ that is locally of finite type, we may consider the deformation to the \emph{intrinsic normal cone} defined in \cite{AranhaPstragowski}.
This is a family
\[
  X \times \A^1 \to \sfD_{X/Y} \to \A^1
\]
degenerating $f$ to the zero section $0 : X \to \sfC_{X/Y}$ to the (relative) \emph{intrinsic normal cone}.
In the case of Deligne--Mumford-type morphisms of $1$-Artin stacks, this recovers a construction of Kresch (see \cite[Thm.~2.31]{Manolache} or the proof of \cite[Prop.~1]{KimKreschPantev}).
In that case, $\sfC_{X/Y}$ is the relative intrinsic normal cone of Behrend--Fantechi (see \cite[\S 7]{BehrendFantechi}).

We proceed to the comparison.
Let $f : X \to Y$ be a locally of finite type morphism of derived Artin stacks.
We consider two different degenerations of the induced morphism $f_\cl : X_\cl \to Y_\cl$ on classical truncations.
On one hand we have $(\Dl_{X/Y})_\cl$, the classical truncation of the normal deformation of $f$, degenerating to the ``virtual'' normal bundle $(N_{X/Y})_\cl$.
On the other hand we may consider the deformation $\sfD_{X_\cl/Y_\cl}$ to the intrinsic normal cone $\sfC_{X_\cl/Y_\cl}$ for $f_\cl : X_\cl \to Y_\cl$.
The square
\[\begin{tikzcd}
  \sfC_{X_\cl/Y_\cl} \ar{r}\ar{d}
  & \sfD_{X_\cl/Y_\cl} \ar{d}
  \\
  X_\cl \ar{r}{f_\cl}
  & Y_\cl
\end{tikzcd}\]
defines a virtual Cartier divisor over $f_\cl$, and hence over $f$.
This is classified by a canonical morphism $\sfD_{X_\cl/Y_\cl} \to \Dl_{X/Y}$, which factors through the classical truncation $(\Dl_{X/Y})_\cl$ and fits into the diagram of cartesian squares
\begin{equation}\label{eq:Bifrutam}
  \begin{tikzcd}
    \msf{C}_{X_\cl/Y_\cl} \ar{r}\ar{d}
    & \sfD_{X_\cl/Y_\cl} \ar[leftarrow]{r}\ar{d}
    & Y_\cl \times \Gm \ar[equals]{d}
    \\
    (\Nl_{X/Y})_\cl \ar{r}\ar{d}{v}
    & (\Dl_{X/Y})_\cl \ar[leftarrow]{r}\ar{d}{u}
    & Y_\cl \times \Gm \ar{d}
    \\
    0 \ar{r}{0}
    & \A^1 \ar[leftarrow]{r}
    & \Gm.
  \end{tikzcd}
\end{equation}

\begin{prop}\label{prop:compMD}
  Let $f : X \to Y$ be a morphism, locally of finite type, between derived Artin stacks.
  The morphisms $\sfD_{X_\cl/Y_\cl} \to (\Dl_{X/Y})_\cl$ and $\msf{C}_{X_\cl/Y_\cl} \to (\Nl_{X/Y})_\cl$ are closed immersions.
  Moreover, the inclusions
  \begin{equation*}
    Y_\cl \times \Gm \hook \sfD_{X_\cl/Y_\cl} \hook (\Dl_{X/Y})_\cl
  \end{equation*}
  exhibit $\sfD_{X_\cl/Y_\cl}$ as the schematic closure of $Y_\cl \times \Gm$ in $(\Dl_{X/Y})_\cl$.
\end{prop}

\begin{proof}
  In view of \propref{prop:DsurjY} and the analogous descent statement for $\sfD_{X_\cl/Y_\cl}$ in \cite{AranhaPstragowski}, we may work smooth-locally on $X$ and $Y$. We can thus assume that
  $X\to Y$ is a closed immersion (cf.\ \lemref{lem:acturience}).
  Using \thmref{thm:Daff} we may then translate the claim to the following statement about the extended Rees algebra $\Rext_{X/Y}$ and its classical version $\msf{R}^{\mrm{ext}}_{X_\cl/Y_\cl}$\footnote{
    See the recollection in \ssecref{ssec:Rees}.
  }: the homomorphism
  \begin{equation*}
    \pi_0 \Rext_{X/Y} \to \msf{R}^{\mrm{ext}}_{X_\cl/Y_\cl}
  \end{equation*}
  is surjective, and moreover exhibits its target as the image of the localization homomorphism $\pi_0 \Rext_{X/Y} \to (\pi_0 \Rext_{X/Y})_{t^{-1}} \simeq (\pi_0 \Rext_{X/Y}) \otimes_{\Z[t^{-1}]} \Z[t,t^{-1}]$, where $t^{-1}$ denotes the coordinate of $\A^1$.
  This statement is proven in \cite[Prop.~6.12.1]{Hekking}.
\end{proof}

\begin{rem}
  \propref{prop:compMD} admits the following geometric interpretation when $X$ and $Y$ are defined over a base field: informally speaking, the classical deformation $\sfD_{X_\cl/Y_\cl} \to \A^1$ may be regarded as the ``flattening'' of the classical truncation of its derived version $(\Dl_{X_\cl/Y_\cl})_\cl \to \A^1$ (cf. \cite[III, Prop.~9.8]{Hartshorne}).
\end{rem}

\begin{rem}\label{rem:Reituba}
  There is a canonical $\Gm$-action on $\sfD_{X_\cl/Y_\cl}$ which makes the projection to $\A^1$ equivariant (with respect to scaling by weight $-1$ action on $\A^1$) and which restricts to the canonical action on the cone stack $\sfC_{X/Y}$ (see \cite[Rem.~6.5]{AranhaPstragowski}).
  By construction, the morphism $\sfD_{X_\cl/Y_\cl} \to (\Dl_{X/Y})_\cl$ is $\Gm$-equivariant.
  In fact, the entire diagram \eqref{eq:Bifrutam} is $\Gm$-equivariant, and the statement of \propref{prop:compMD} passes to the quotient stacks.
\end{rem}

\subsection{Infinitesimal neighbourhoods}
\label{ssec:inf}

Recall that if $X \hook Y$ is a closed immersion of classical schemes defined by an ideal sheaf $\sI \sub \sO_Y$, the $n$th infinitesimal neighbourhood is the thickening defined by the ideal sheaf $\sI^n$.
In this paragraph we explain how to define an analogue of this construction in derived algebraic geometry, by reinterpreting the extended Rees algebra (\ssecref{ssec:Rees}) as a derived analogue of the adic filtration.

Let $f : X \hook Y$ be a closed immersion of derived stacks.
Recall that there is a canonical equivalence between quasi-coherent complexes on $Y\times\AGm$ and filtered quasi-coherent complexes on $Y$ (see e.g. \cite[Chap.~9, \S 1.3]{GaitsgoryRozenblyumII} or \cite[\S 2.2]{BhattLuriePrism}), as well as a compatible equivalence between quasi-coherent $\sO_{Y\times\AGm}$-algebras and filtered quasi-coherent $\sO_Y$-algebras\footnote{%
  I.e., quasi-coherent $\sO_Y$-algebras equipped with a multiplicative filtration.
  See Footnote~\ref{fn:qcohalg} for what we mean by ``quasi-coherent $\sO_Y$-algebra''.
} (see \cite[\S 3]{Hekking} or \cite[Prop.~3.39]{BenBassatHekking}).
Under this equivalence, the extended Rees algebra $\Rext_{X/Y}$ (\defnref{defn:Rees}) corresponds to a canonical multiplicative filtration
\begin{equation}\label{eq:Fil}
  \cdots
  \to \Fil_n (\sO_Y)
  \to \cdots
  \to \Fil_1 (\sO_Y)
  \to \Fil_0 (\sO_Y) \simeq \sO_Y
\end{equation}
where $\Fil_n (\sO_Y) := (\Rext_{X/Y})_n$.
The formulas \eqref{eq:underlying} and \eqref{eq:assgr} say that the underlying object of this filtration is $\sO_Y$, and the associated graded object is $f_*\Sym^*_{\sO_X}(\sNv_{X/Y})$, i.e., there are exact triangles
\begin{equation}
  \Fil_{n+1}(\sO_Y)
  \to \Fil_{n}(\sO_Y)
  \to f_*\Sym^n_{\sO_X}(\sNv_{X/Y}).
\end{equation}

We refer to \eqref{eq:Fil} as the \emph{adic filtration} on $\sO_Y$ determined by the closed immersion $f : X \hook Y$.
Note that the same filtration has also been constructed by Bhatt, by animating the classical adic filtration (see the proof of \cite[Cor.~4.14]{BhattComp}, or \cite[\S 3.4]{Mao}).
The comparison between these constructions will appear in forthcoming work of the first-named author.

\begin{thm}\label{thm:inf}
  Let $f : X \hook Y$ be a closed immersion of derived stacks.
  Then there exists a canonical tower of nilpotent closed immersions
  \begin{equation}\label{eq:inf}
    X = X^{(0)}
    \hook X^{(1)}
    \hook X^{(2)}
    \hook \cdots
  \end{equation}
  between closed immersions $f_n : X^{(n)} \hook Y$, such that for every $n>0$:
  \begin{thmlist}
    \item\label{item:inf/sym}
    There is a canonical exact triangle
    \begin{equation*}
      f_* \Sym^n_{\sO_X}(\sNv_{X/Y})
      \to f_{n,*}\sO_{X^{(n)}}
      \to f_{n-1,*}\sO_{X^{(n-1)}}
    \end{equation*}
    in $\QCoh(Y)$.
    \item\label{item:inf/Rn}
     There is a canonical exact triangle
    \begin{equation*}
      (\R_{X/Y})_{n+1}\simeq(\Rext_{X/Y})_{n+1}
      \to \sO_Y
      \to f_{n,*}\sO_{X^{(n)}}
    \end{equation*}
    in $\QCoh(Y)$.
    \item\label{item:inf/ideal}
    There is a short exact sequence
    \begin{equation*}
      0
      \to {}^\heartsuit\sI^n
      \hook \pi_0 \sO_Y
      \twoheadrightarrow \pi_0 f_{n,*} \sO_{X^{(n)}}
      \to 0,
    \end{equation*}
    of (discrete) quasi-coherent sheaves on $Y$,
    where ${}^\heartsuit\sI$ is the ideal of definition of $f_\cl : X_\cl \hook Y_\cl$.
  \end{thmlist}
\end{thm}

\begin{defn}
  We call $f_n : X^{(n)} \hook Y$ in \thmref{thm:inf} the \emph{$n$th infinitesimal neighbourhood} of $f : X \hook Y$.
\end{defn}

\begin{proof}[Proof of \thmref{thm:inf}]
  For each $n\ge 0$, consider the $\Z$-graded quasi-coherent $\sO_{Y}$-algebra
  \begin{equation*}
    \sR\{n\} := \Rext_{X/Y} \otimes_{\sO_Y[t^{-1}]} \sO_Y[t^{-1}]/(t^{-n-1}).
  \end{equation*}
  We define $\sR^{(n)}:=(\sR\{n\})_0$ as the degree zero component of $\sR\{n\}$ and $X^{(n)}$ as the relative spectrum,
  \begin{equation*}
    f_n : X^{(n)} := \uSpec_Y(\sR^{(n)}) \to Y.
  \end{equation*}

  Tensoring up from $(t^{-n-1}) \hook \sO_Y[t^{-1}] \twoheadrightarrow \sO_Y[t^{-1}]/(t^{-n-1})$, we obtain the canonical exact triangle of underlying $\Z$-graded quasi-coherent sheaves on $Y\times\A^1$
  \begin{equation}\label{eq:R(n)}
    \Rext_{X/Y} (n+1) \to \Rext_{X/Y} \to \sR \{n\},
  \end{equation}
  where $\Rext_{X/Y}(n+1)$ denotes the grading shift such that $\Rext_{X/Y}(n+1)_i \simeq (\Rext_{X/Y})_{n+1+i}$.
  Taking degree zero components of \eqref{eq:R(n)}, we obtain the exact triangle of \itemref{item:inf/Rn}.
  Taking $\pi_0$ and using the discussion in \cite[\S 6.12]{Hekking} (or the proof of \propref{prop:compMD}), this immediately yields \itemref{item:inf/ideal}.
  For the remaining point \itemref{item:inf/sym}, we use the exact triangle
  \begin{equation*}
    \begin{multlined}
      \Fib(\sO_Y \to f_{n,*}\sO_{X^{(n)}})
      \to \Fib(\sO_Y \to f_{n-1,*}\sO_{X^{(n-1)}})\\
      \to \Fib(f_{n,*}\sO_{X^{(n)}} \to f_{n-1,*}\sO_{X^{(n-1)}})
    \end{multlined}
  \end{equation*}
  to compute the last term as the cofibre of $(\Rext_{X/Y})_{n+1} \to (\Rext_{X/Y})_{n}$, i.e., the $n$th associated graded of the filtration determined by $\Rext_{X/Y}$, which is $\Sym^n_{\sO_X}(\sNv_{X/Y})$ by \eqref{eq:assgr}.
\end{proof}

\begin{cor}\label{cor:filf_*}
  Let $f : X \hook Y$ be a closed immersion of derived stacks.
  Then for every $\sF \in \QCoh(Y)$, there exists a canonical tower of quasi-coherent complexes on $Y$
  \begin{equation}
    \cdots \to \sF_n \to \cdots \to \sF_1 \to \sF_0 = f_*f^*(\sF),
  \end{equation}
  and for each $n>0$ a canonical exact triangle
  \begin{equation}
    f_* \Sym^n_{\sO_X}(\sNv_{X/Y}) \otimes \sF
    \to \sF_n
    \to \sF_{n-1}
  \end{equation}
  in $\QCoh(Y)$.
\end{cor}
\begin{proof}
  Under the equivalence between relatively affine derived stacks over $Y$ and quasi-coherent $\sO_Y$-algebras, the tower of \thmref{thm:inf} corresponds to a tower
  \begin{equation}\label{eq:filf_*}
    \cdots \to \sR^{(n)} \to \cdots \to \sR^{(1)} \to \sR^{(0)} \simeq f_*(\sO_X)
  \end{equation}
  of quasi-coherent complexes on $Y$, with exact triangles
  \begin{equation}
    f_* \Sym^n_{\sO_X}(\sNv_{X/Y})
    \to \sR^{(n)}
    \to \sR^{(n-1)}
  \end{equation}
  for each $n\ge 0$.
  This is the desired filtration for $\sF = \sO_Y$.
  The filtration for a general $\sF \in \QCoh(Y)$ is $\eqref{eq:filf_*}\otimes\sF$, up to the projection formula $f_*(\sO_X)\otimes\sF \simeq f_*f^*(\sF)$.
\end{proof}

In characteristic zero, a similar construction of infinitesimal neighbourhoods was given by Gaitsgory and Rozenblyum, using their formal version of the normal deformation; see \cite[Chap.~9, Sec.~5]{GaitsgoryRozenblyumII} and the discussion in \sssecref{sssec:GR}.
For almost finitely presented closed immersions between qcqs derived Artin stacks, another construction of infinitesimal neighbourhoods was given in \cite[Thm.~A.0.1]{HalpernLeistnerDEquiv}.
The construction there involves globalizing from the affine case, where the tower \eqref{eq:inf} is constructed by animation (compare e.g. \cite[Thm.~3.23]{Mao}).
However, as pointed out to us by Lucas Mann, the proof of assertions (2)--(4) in the affine case contains an error.\footnote{
  The proof of \cite[Thm.~A.0.1(2)--(4)]{HalpernLeistnerDEquiv} is done by checking that Construction~A.0.6 satisfies the conditions of Lemma~A.0.2, i.e., that for any $d\ge 0$ the morphism $\pi_0 (A/I) \to \tau_{\le d} ((A/I^n) \otimes_A^L \pi_0(A/I))$ is a weak equivalence for all $n\gg 0$.
  But this typically fails for $d>0$, take e.g. $A=\bZ[x]$, $I=(x)$.
}

%!TEX root = ../weilres.tex

\section{Blow-ups}
\label{sec:bl}

\subsection{Excessive virtual Cartier divisors}

Let $i : X \hook Y$ be a closed immersion of derived stacks.
We say that a virtual Cartier divisor $D$ on $S$ over $i$ is \emph{excessive} if the associated commutative square
\begin{equation}\label{eq:excess}
  \begin{tikzcd}
    D \ar{r}{i_D}\ar{d}{g}
    & S \ar{d}{f}
    \\
    X \ar{r}{i}
    & Y
  \end{tikzcd}
\end{equation}
is excessive in the sense of \defnref{defn:excessive}.
That is:
\begin{defnlist}
  \item
  It is cartesian on classical truncations.

  \item
  The morphism of normal bundles $df : N_{D/S} \to N_{X/Y} \fibprod_X D$ is a closed immersion.
  Equivalently, the morphism of conormal complexes $g^*\sNv_{X/Y} \to \sNv_{D/S}$ is surjective on $\pi_0$.
\end{defnlist}
Recall here that the conormal complex of $i : X \hook Y$ is defined as $\sNv_{X/Y} := \sL_{X/Y}[-1]$, and that \lemref{lem:excvcd} gives a reformulation of the excessiveness condition in terms of the induced closed immersion $e : D \to X_S := X \times_Y S$. 

\begin{rem}\label{rem:Nft}
  Note that $\pi_0(\sNv_e)$ is of finite type, regardless of whether \eqref{eq:excess} is excessive.
  Indeed, in the exact triangle $g^*\sNv_{X/Y} \to \sNv_{D/S} \to \sNv_e$, the fibre of the first map is connective so the second map is surjective on $\pi_0$.
  Since $\sNv_{D/S}$ is locally free, the claim follows.
\end{rem}

\subsection{The excessive locus}

Recall from \ssecref{ssec:D} that the normal deformation $\cD_{X/Y}$ is the derived stack classifying virtual Cartier divisors over $i : X \hook Y$.
We identify the locus of \emph{excessive} virtual Cartier divisors as the open complement of the canonical morphism
\begin{equation}\label{eq:Di}
  \cD i : X \times \AGm \simeq \cD_{X/X} \to \cD_{X/Y},
\end{equation}
which is a closed immersion by \propref{prop:Dqprop} and \examref{exam:selfint}.

\begin{lem}\label{lem:orange}
  Let $i : X \hook Y$ be a closed immersion.
  Denote by $\sU \sub \cD_{X/Y}$ the substack consisting of points classifying virtual Cartier divisors over $i$ which are excessive.
  Then the inclusion $\sU \hook \cD_{X/Y}$ is the complement of the closed immersion \eqref{eq:Di}; in particular, it is an open immersion.
\end{lem}
\begin{proof}
  Consider the universal virtual Cartier divisor
  \[\begin{tikzcd}
    D \ar{r}{i_{D}}\ar{d}
    & S \ar{d}
    \\
    X \ar{r}{i}
    & Y
  \end{tikzcd}\]
  where $D := \cN_{X/Y}$ and $S := \cD_{X/Y}$, and let $e : D\hook X_S = X \fibprod_Y S$ denote the induced morphism with corresponding map $e^\sharp : \sO_{X_{S}} \to e_*(\sO_{D})$.
  By \lemref{lem:excvcd}, $\sU$ is the locus of points in $\cD_{X/Y}$ where
  \begin{inlinelist}
    \item $i_{D,*}(\sNv_e) \in \QCoh(S)$ has vanishing $\pi_0$, and
    \item $i'_*(e^\sharp)$ is bijective on $\pi_0$, where $i': X_S \hook S$ is the projection.
  \end{inlinelist}
  Since $\pi_0 (i_{D,*}(\sNv_e))$ is of finite type (\remref{rem:Nft}), its support is closed in $S_\cl \sub S$ (see \cite[Chap.~0\textsubscript{I}, 5.2.2]{EGA}), hence its vanishing locus is open.
  Similarly, consider the surjection of (discrete) quasi-coherent sheaves
  \begin{equation*}
    \pi_0 (i'_* e^\sharp) : \pi_0 (i'_* \sO_{X_S}) \to \pi_0 (i_{D,*} \sO_D).
  \end{equation*}
  The target is of finite presentation (since $i_D : D \hook S$ is quasi-smooth) and the source is of finite type, so the kernel is of finite type, hence also has open vanishing locus.
  This shows that the inclusion $\sU \hook \cD_{X/Y}$ is an open immersion.

  We next show that it is complementary to the closed immersion \eqref{eq:Di}.
  Let $S = \Spec(\kappa) \to \cD_{X/Y}$ be a morphism where $\kappa$ is a field, classifying a virtual Cartier divisor $D \hook S$ over $i : X \to Y$.
  Unravelling definitions, $S \to \cD_{X/Y}$ lifts along $\cD i : X \times \AGm \to \cD_{X/Y}$ if and only if in the solid arrow diagram below, there exists a dashed arrow together with homotopies witnessing the commutativity of the two triangles:
  \begin{equation}\label{eq:wood}
    \begin{tikzcd}
      D \ar{r}{i_D}\ar{d}{g}
      & S \ar{d}{s}\ar[dashed,swap]{ld}{h}\\
      X \ar{r}{i}
      & Y.
    \end{tikzcd}
  \end{equation}
  The claim is that such a lift exists if and only if $D \hook S$ is not excessive over $i$.

  Suppose that such a lift exists and assume for the sake of contradiction that $D \hook S$ is excessive.
  We find that $\sNv_{X/Y}|_{D} \to \sNv_{X/X}|_{D} \to \sNv_{D/S}$ is surjective on $\pi_0$.
  Since the middle term vanishes, the rank one free sheaf $\sNv_{D/S}$ on $D$ also vanishes, hence $D$ must be empty.
  But then the classical cartesianness condition implies that $S$ is also empty, whence the contradiction.

  Conversely, suppose $D \hook S$ is not excessive.
  As a virtual Cartier divisor in $S = \Spec(\kappa)$, $D$ is either empty or the trivial square zero extension $\Spec(\kappa\oplus\kappa[1])$.
  Let us construct a lift in both cases.

  If $D$ is empty, then it is excessive if and only if $D \simeq X\fibprod_Y S$ on classical truncations (as the surjectivity condition is vacuous), hence if and only if $s : S \to Y$ does not lie in the image of $i : X \hook Y$.
  Since we assumed it is not excessive, we conclude $s : S \to Y$ does factor through $i : X \hook Y$.
  This provides the lift $h$ in this case.

  If $D = \Spec(\kappa\oplus\kappa[1])$, we want to construct a lift $h$ in the following diagram:
  \begin{equation}\label{eq:ancient}
    \begin{tikzcd}
      \Spec(\kappa\oplus\kappa[1]) \ar{r}{i_D}\ar[swap]{d}{g}
      & \Spec(\kappa) \ar{d}{s}\ar[dashed,swap]{ld}{h}
      \\
      X \ar{r}{i}
      & Y.
    \end{tikzcd}
  \end{equation}
  Note that commutativity of the solid arrow square implies at the level of classical truncations that $s$ lies in $X_\cl$ and hence in $X$; we let $h$ denote the lift thus obtained.
  It remains to show commutativity of the upper triangle.
  Adopting the notation of \ssecref{ssec:convent/cot}, the following commutative diagram exhibits $g$ as a derivation of $i : X \hook Y$ at $s$ with values in $\sO[1]$:
  \begin{equation*}
    \begin{tikzcd}
      & \Spec(\kappa) \ar{ld}{z}\ar{rd}{h} & 
      \\
      \Spec(\kappa)[\sO[1]] \ar{rr}{g}\ar{rd} & & X\ar{ld}{i}
      \\
      & Y, &
    \end{tikzcd}
  \end{equation*}
  where $z$ is the canonical section. The assumption that $D \hook S$ is not excessive implies that the canonical map $g^*\sNv_{X/Y} \to \sO_D$ is zero on $\pi_0$, and hence that $h^*\sNv_{X/Y} \to \sO_S$ is null-homotopic.
  This translates to the fact that the derivation $g$ above is homotopic to the trivial derivation
  \[
    \Spec(\kappa\oplus\kappa[1]) \xrightarrow{i_D} \Spec(\kappa) \xrightarrow{h} X.
  \]
  In particular, this yields a homotopy making the upper triangle in \eqref{eq:ancient} commute.
\end{proof}

\subsection{Blow-ups}
\label{ssec:blups}

Let $i : X \hook Y$ be a closed immersion of derived stacks.

\begin{defn}
  The \emph{blow-up} $\Bl_{X/Y}$ is the derived stack classifying excessive virtual Cartier divisors over $i$.
  That is, it is the derived stack over $Y$ whose $S$-points, for a derived scheme $S$ over $Y$, are excessive virtual Cartier divisors on $S$ over $i$.
\end{defn}

Note that $\Bl_{Y/Y} \simeq \initial$ and $\Bl_{\initial/Y} \simeq Y$, while if $i$ is a virtual Cartier divisor then the projection $\Bl_{X/Y} \to Y$ is invertible.
It is also clear that the formation of $\Bl_{X/Y}$ is stable under arbitrary base change in $Y$.

\begin{prop}
  For any morphism $Y' \to Y$, there is a canonical isomorphism $\Bl_{X/Y} \fibprod_Y Y' \simeq \Bl_{X\fibprod_YY'/Y'}$ over $Y'$.
\end{prop}

We have the universal excessive Cartier divisor
\begin{equation}\label{eq:blowupsquare}
  \begin{tikzcd}
    \El_{X/Y} \ar{r}{i_E}\ar{d}
    & \Bl_{X/Y} \ar{d}
    \\
    X \ar{r}{i}
    & Y
  \end{tikzcd}
\end{equation}
over $i$.
Away from $X$, the projection $\Bl_{X/Y} \to Y$ is invertible, since
\[
  \Bl_{X/Y} \fibprod_Y (Y\setminus X)
  \simeq \Bl_{\initial/Y\setminus X}
  \simeq Y\setminus X.
\]

We claim that the projection $\El_{X/Y} \to X$ exhibits the exceptional divisor as the projectivized normal bundle.
Consider the diagram of cartesian squares over $X$,
\begin{equation}\label{eq:Oznufci}
  \begin{tikzcd}
    \El_{X/Y} \ar{r}{i_E}\ar{d}
    & \Bl_{X/Y} \ar{d}\ar[leftarrow]{r}
    & Y \setminus X \ar{d}
    \\
    \cN_{X/Y} \ar{r}{i_\cD}
    & \cD_{X/Y} \ar[leftarrow]{r}{j_\cD}
    & Y
    \\
    X \times \BGm \ar{r}\ar{u}{0}
    & X \times \AGm \ar[leftarrow]{r}\ar{u}{\cD i}
    & X \ar{u}{i}
  \end{tikzcd}
\end{equation}
which is built as follows.
We begin with the middle row, where $i_\cD$ and $j_\cD$ are the inclusions of the fibres over and away from $0 : \BGm \hook \AGm$, respectively.
The lower row is the same for $\id : X \to X$ in place of $i : X \hook Y$.
Finally, the upper row is obtained by taking open complements in the vertical direction, using \lemref{lem:orange} to compute the central term.
We conclude in particular that the exceptional divisor $\El_{X/Y}$ can be described as
\begin{equation}
  \El_{X/Y} \simeq \P_X(N_{X/Y}) := [(N_{X/Y} \setminus X)/\Gm] \simeq
  \cN_{X/Y} \setminus (X \times \BGm)
\end{equation}
over $X$.

\subsection{Blow-ups as projective cones}

Let $Y$ be a derived stack.
An \emph{affine cone} $C$ over $Y$ is the affine spectrum $\uSpec_Y(\sA)$ of an $\N$-graded connective quasi-coherent $\sO_Y$-algebra $\sA$, such that the canonical homomorphism $\sO_Y \to \sA_0$ is invertible and $\sA$ is generated in degree $1$.
We identify $Y$ with the vertex of the cone via the closed immersion $Y \hook C$, determined by the projection $\sA \twoheadrightarrow \sA_0 \simeq \sO_Y$.
The scaling action of $\bG_m$ on $C$, determined by the grading on $\sA$, fixes the vertex and thus restricts to an action on the complement.
The \emph{projective cone} $\bP_Y(C) = \uProj_Y(\sA)$ is the quotient stack
\begin{equation*}
  \bP_Y(C) := [C \setminus Y/\bG_m]
\end{equation*}
which is schematic over $Y$ (see e.g. \cite[\S 5]{Hekking}).

If $\sA$ is $\Z$-graded, we write $\sA_{\ge 0}$ for the $\N$-graded algebra formed by the components of nonnegative degrees.
We assume that $\sA_{\ge 0}$ satisfies the conditions above, and write $C$ for the corresponding affine cone.
In terms of $\sA$, the projective cone $\bP_Y(C)$ can also be realized by
\begin{equation}\label{eq:0noqejwp6kV}
  \bP_Y(C) = [(\uSpec_Y(\sA) \setminus \uSpec_Y(\sA \otimes_{\sA_{\ge 0}} \sA_0))/\bG_m].
\end{equation}
Indeed, we have:

\begin{lem}\label{lem:town}
  The $\Gm$-equivariant morphism $\uSpec_Y(\sA) \to \uSpec_Y(\sA_{\ge 0}) = C$ induces an isomorphism
  \begin{equation*}
    \uSpec_Y(\sA) \setminus \uSpec_Y(\sA \otimes_{\sA_{\ge 0}} \sA_0) \simeq C \setminus Y
  \end{equation*}
  away from the vertex.
\end{lem}

\begin{proof}
  Since the claim is stable under base change in $Y$, we may assume that $Y$ is affine.
  It will suffice to show that, for an animated commutative ring $R$ and a $\Z$-graded $R$-algebra $A$, the canonical morphism $\Spec(A) \to \Spec(A_{\ge 0})$ induces an isomorphism over every principal open $D(f) \sub \Spec(A_{\ge 0})$ where $f \in A$ is homogeneous of degree $1$; in other words, that the homomorphism $(A_{\ge 0})[f^{-1}] \to A[f^{-1}]$ is invertible for every $f \in A_1$.
  Since formation of homotopy groups commutes with localization, it will suffice to note that the $\pi_0(A)_{\ge 0}[f^{-1}] \simeq \pi_0(A_{\ge 0})[f^{-1}]$-module homomorphism $\pi_n(A)_{\ge 0}[f^{-1}] \simeq \pi_n(A_{\ge 0})[f^{-1}] \to \pi_n(A)[f^{-1}]$ is invertible for every $f \in A_1$ and $n\ge 0$ (see e.g. \cite[II, (8.2.1.5)]{EGA}).
\end{proof}

Let $i : X \hook Y$ be a closed immersion of derived stacks.
By \thmref{thm:Daff}, $\Dl_{X/Y}$ is the affine spectrum over $Y$ of the extended Rees algebra $\Rext_{X/Y}$.
Let $C$ be the affine cone $\uSpec_Y(\R_{X/Y})$.
In this case, the left-hand side of \eqref{eq:0noqejwp6kV} is the projective cone $\P_Y(C) = \uProj_Y(\R_{X/Y})$.
The right-hand side is the $\Gm$-scaling quotient of
\begin{equation}\label{eq:gray}
  \uSpec_Y(\Rext_{X/Y}) \setminus \uSpec_Y(\Rext_{X/Y} \otimes_{\R_{X/Y}} \sO_Y).
\end{equation}

The following lemma will show that \eqref{eq:gray} is identified with
\begin{equation}\label{eq:gray2}
  \uSpec_Y(\Rext_{X/Y}) \setminus X \times \A^1 = \Dl_{X/Y} \setminus X \times \A^1,
\end{equation}
whose $\Gm$-scaling quotient is nothing else than the derived blow-up $\Bl_{X/Y}$ (\lemref{lem:orange}).

\begin{lem}\label{lem:ice}
  Let $A \to B$ be a homomorphism of animated commutative rings which is surjective on $\pi_0$.
  Then the canonical $\Z$-graded $A$-algebra homomorphism $\Rext_{B/A} \otimes_{\R_{B/A}} A \to B[t^{-1}]$ is bijective on $\pi_0$.\footnote{%
    In fact, one can show that it is invertible; see the proof of \cite[Lem.~4.28]{BenBassatHekking}.
  }
\end{lem}
\begin{proof}
  For simplicity we set $R := \R_{B/A}$ and $R^+ := \Rext_{B/A}$.
  We first note that $\pi_0(R^+\otimes_R A) \simeq \pi_0(R^+) \otimes_{\pi_0(R)} \pi_0(A)$ vanishes in degrees $d>0$.
  Indeed, its homogeneous component of degree $d$ is a quotient of
  \[
    \pi_0(R^+_{d}) \otimes_{\pi_0(R_{0})} \pi_0 (A)
    \simeq \pi_0(R_{d}) \otimes_{\pi_0(A)} \pi_0(A)
  \]
  by relations including
  \[
    x \cdot y \otimes a = y \otimes x \cdot a
  \]
  for $x \in \pi_0(R_1)$, $y \in \pi_0(R_{d-1})$, $a \in \pi_0(A)$.
  But we have $x \cdot a = 0$ since the action of $R$ on $R_0 = A$ is via the projection sending all positive degree components to zero.
  Since $\pi_0(R)$ is generated in degree $1$, this implies that $z \otimes a = 0$ for all $z \in \pi_0(R_d)$.

  In degree $-d$ (for any $d\geq 0$), $\pi_0(R^+ \otimes_R A)$ is a quotient of
  \[
    \pi_0(R^+_{-d}) \otimes_{\pi_0(R_0)} \pi_0(A)
    \simeq \pi_0(A) \cdot t^{-d} \otimes_{\pi_0(A)} \pi_0(A)
  \]
  by relations including
  \[
    x \cdot y \otimes a = y \otimes x \cdot a
  \]
  for $x \in \pi_0(R_1) \simeq \pi_0(I) \cdot t$, $y \in \pi_0(R^+_{-d-1}) \simeq \pi_0(A) \cdot t^{-d-1}$, $a \in \pi_0(A)$.
  Again, we have $x \cdot a = 0$ as above, hence in particular we have the relation
  \[
    f \cdot t^{-d} \otimes a
    = (f \cdot t) \cdot (t^{-d-1}) \otimes a = 0
  \]
  for all $f \in \pi_0(I)$ and $a \in \pi_0(A)$.
  The other relations, coming from $x \in \pi_0(R_j)$, $y \in \pi_0(R^+_k)$, $a \in \pi_0(A)$, with $j+k=-d$ and $j>1$, are redundant.
  Thus the map induced by $\pi_0(R^+\otimes_R A) \to \pi_0(B)[t^{-1}]$ in degree $-d$ is identified with the isomorphism $\pi_0(A) \cdot t^{-d}/\pi_0(I)\cdot t^{-d} \to \pi_0(B) \cdot t^{-d}$.
\end{proof}

In view of \eqref{eq:gray2}, \lemref{lem:town} yields in this situation:

\begin{thm}\label{thm:blproj}
  For any closed immersion of derived stacks $i : X \hook Y$, there is a canonical isomorphism
  \[
    \Bl_{X/Y} \simeq \uProj_Y(\R_{X/Y})
  \]
  over $Y$.
\end{thm}

\begin{cor}
  For any closed immersion of derived stacks $i : X \hook Y$, the structural morphism $\pi : \Bl_{X/Y} \to Y$ is schematic.
  Moreover, if the fibre $\sI$ of $i^\sharp : \sO_Y \twoheadrightarrow i_*(\sO_X)$ has $\pi_0(\sI)$ is of finite type (e.g. if $i$ is almost of finite presentation), then $\pi$ is projective (\emph{a fortiori} proper).
\end{cor}

\subsection{Functoriality}

Suppose given a commutative square
\begin{equation}\label{eq:pinken4}
  \begin{tikzcd}
    X' \ar{r}{i'}\ar{d}{p}
    & Y' \ar{d}{q}
    \\
    X \ar{r}{i}
    & Y,
  \end{tikzcd}
\end{equation}
where $i$ and $i'$ are closed immersions.
If the square is excessive in the sense of \defnref{defn:excessive}, then this gives rise to a canonical morphism
\[
  Bq : \Bl_{X'/Y'} \to \Bl_{X/Y}
\]
which factors as
\[
  \Bl_{X'/Y'} \xrightarrow{dq}
  \Bl_{X\fibprod_Y Y'/Y'}
  \simeq \Bl_{X/Y}\fibprod_Y Y' \xrightarrow{q_B}
  \Bl_{X/Y}.
\]
In particular, for any sequence of closed immersions $X' \to Y' \to Y$ we obtain a canonical map $\Bl_{X'/Y'} \to \Bl_{X'/Y}$ via \examref{exam:selfint}.

\begin{prop}\label{prop:Blprop}
  If $q : Y' \to Y$ is proper (resp. finite, a closed immersion), then the morphism $Bq : \Bl_{X'/Y'} \to \Bl_{X/Y}$ is proper (resp. finite, a closed immersion).
\end{prop}
\begin{proof}
  If $q$ is proper (resp. finite, a closed immersion), then so is $q_B$.
  Hence it suffices to show that $dq: \Bl_{X'/Y'} \to \Bl_{X\fibprod_Y Y'/Y'}$ is a closed immersion.
  We may assume that $Y'$ is affine.
  By \thmref{thm:blproj} it will then suffice to show that the induced homomorphism of Rees algebras is surjective on $\pi_0$, which was already demonstrated in the proof of \propref{prop:Dqprop}.
\end{proof}

\subsection{Cotangent complex}
\label{ssec:bl/cot}

Let $i : X \hook Y$ be a closed immersion of derived stacks.
Recall the blow-up square \eqref{eq:blowupsquare}:
\begin{equation*}
  \begin{tikzcd}
    \El_{X/Y} \ar{r}{i_E}\ar{d}{q}
    & \Bl_{X/Y} \ar{d}{p}
    \\
    X \ar{r}{i}
    & Y.
  \end{tikzcd}
\end{equation*}

\begin{prop}
  The canonical morphism
  \begin{equation*}
    \sL_{\Bl_{X/Y}/Y} \to i_{E,*} (\sL_{\El_{X/Y}/X})
  \end{equation*}
  is invertible.
\end{prop}
\begin{proof}
  Since $\El_{X/Y} \simeq \cN_{X/Y} \fibprod_{\cD_{X/Y}} \Bl_{X/Y}$ (see \eqref{eq:Oznufci}), the morphism in question is the restriction of the isomorphism
  \[
    \sL_{\cD_{X/Y}/Y}
    \to i_{\cN,*}(\sL_{\cN_{X/Y}/X})
  \]
  of \propref{prop:defcot}.
\end{proof}

\subsection{From blow-ups to normal deformations}
\label{ssec:bl/undeck}

In the case of closed immersions, the normal deformation may be realized in terms of blow-ups, just as in classical algebraic geometry.

\begin{prop}\label{prop:farmyardy}
  Let $i : X\hook Y$ be a closed immersion of derived stacks.
  Then there is a canonical $\Gm$-equivariant isomorphism
  \[
    \Dl_{X/Y} \simeq \Bl_{X\times\{0\}/Y\times\A^1} \setminus \Bl_{X\times\{0\}/Y\times\{0\}}
  \]
  over $Y \times \A^1$.
\end{prop}

In particular, we see that the normal deformation $\Dl_{X/Y}$ recovers the construction of \cite{blowups} in the quasi-smooth case.

Slightly more generally, we have:

\begin{prop}\label{prop:conversation}
  Let $i : Z \hook S$ be a closed immersion and $h : S \hook T$ a virtual Cartier divisor.
  Then the morphism $Bh : \Bl_{Z/S} \to \Bl_{Z/T}$ is a closed immersion, and its open complement is the Weil restriction $h_*(Z)$.
\end{prop}

\begin{proof}
  The first statement follows from \propref{prop:Blprop}.
  
  Let $D \hook h_*(Z)$ be the virtual Cartier divisor defined by the base change of $h : S \hook T$ along $h_*(Z) \to T$.
  We have a counit morphism $D \to Z$ over $S$ which fits in a commutative square
  \[\begin{tikzcd}
    D \ar{r}\ar{d}
    & h_*(Z) \ar{d}
    \\
    Z \ar{r}{h\circ i}
    & T
  \end{tikzcd}\]
  which one verifies is excessive.
  Regarded as an excessive virtual Cartier divisor over $h \circ i$, this gives rise to a canonical morphism $h_*(Z) \to \Bl_{Z/T}$.
  For any $u : U \to T$, the functor
  \[
    h_*(Z)(U\xrightarrow{u}T) \to \Bl_{Z/T}(U\xrightarrow{u}T)
  \]
  sends a morphism $\phi : U_S := U \fibprod_T S \to Z$ over $S$ to the excessive virtual Cartier divisor
  \[\begin{tikzcd}
    U_S \ar{rr}\ar{d}{\phi}\ar{rd}{u_S}
    &
    & U \ar{d}{u}
    \\
    Z \ar{r}{i}
    & S \ar{r}{h}
    & T.
  \end{tikzcd}\]
  This functor is clearly fully faithful, and an excessive virtual Cartier divisor $D \hook U$ over $h\circ i$ lies in the essential image if and only if the induced morphism $D \to U_S$ is invertible:
  \[\begin{tikzcd}
    D \ar{r}{\sim}\ar{d}
    & U_S \ar{r}\ar{d}
    & U \ar{d}{u}\ar{d}
    \\
    Z \ar{r}{i}
    & S \ar{r}{h}
    & T.
  \end{tikzcd}\]
  This is an open condition (e.g. it is equivalent to the condition that the fibre of the induced map on structure sheaves vanishes, which is open by \cite[Lem.~2.9.3.3]{LurieSAG}).
  This shows that $h_*(Z) \to \Bl_{Z/T}$ is an open immersion.
  
  It remains to show that when $U=\Spec(\kappa)$ for a field $\kappa$, a morphism $u : U \to \Bl_{Z/T}$ factors through $\Bl_{Z/S}$ if and only if $D \to U_S$ is not invertible.
  The former condition is equivalent to the existence of a lift in the following diagram:
  \[\begin{tikzcd}
    D \ar{r}\ar{d}
    & U_S \ar{r}\ar{d}
    & U \ar{d}{u}\ar{d}\ar[dashed]{ld}
    \\
    Z \ar{r}{i}
    & S \ar{r}{h}
    & T,
  \end{tikzcd}\]
  or equivalently that of a section of $U_S \to U$ compatible with the morphisms $D \hook U$ and $D \hook U_S$.
  Since $U = \Spec(\kappa)$, the virtual Cartier divisors $D\hook U$ and $U_S\hook U$ are either empty or of the form $\Spec(\kappa\oplus\kappa[1])$.
  The only nontrivial case is where neither is empty.
  In that case, a morphism $\Spec(\kappa\oplus\kappa[1]) \to \Spec(\kappa\oplus\kappa[1])$ over $\Spec(\kappa)$ is either invertible or factors through $\Spec(\kappa) = U$.
\end{proof}

\begin{proof}[Proof of \propref{prop:farmyardy}]
  Apply \propref{prop:conversation} with $h$ the virtual Cartier divisor $0 : Y \hook Y \times \A^1$ and $i$ the closed immersion $i : X \hook Y$.
\end{proof}

\subsection{Classical blow-ups}

Let $i : X \hook Y$ be a closed immersion of derived Artin stacks, and denote by $i_\cl : X_\cl \hook Y_\cl$ the closed immersion on classical truncations.
In this paragraph we compare the classical blow-up of $\sf{Y}$ along $\sf{X}$ with the classical truncation of the derived blow-up $\Bl_{X/Y}$.

Denote by $\msf{Bl}_{X_\cl/Y_\cl}$ the classical blow-up of $i_\cl : X_\cl \hook Y_\cl$.
We have the blow-up squares
\[\begin{tikzcd}
  (\El_{X/Y})_\cl \ar{r}\ar{d}
  & (\Bl_{X/Y})_\cl \ar{d}
  \\
  X_\cl \ar{r}{i_\cl}
  & Y_\cl,
\end{tikzcd}
\qquad
\begin{tikzcd}
  \sfE_{X_\cl/Y_\cl} \ar{r}\ar{d}
  & \msf{Bl}_{X_\cl/Y_\cl} \ar{d}
  \\
  X_\cl \ar{r}{i_\cl}
  & Y_\cl.
\end{tikzcd}\]
where $(\El_{X/Y})_\cl \simeq \P_{X_\cl}((\cN_{X/Y})_\cl)$ is the ``virtual'' exceptional divisor (i.e., the projectivization of the classical truncation of the normal bundle of $i$) and $\sfE_{X_\cl/Y_\cl} := \P_{X_\cl}(\sfC_{X_\cl/Y_\cl})$ is the classical one (i.e., the projectivized normal cone of $i_\cl$).

Since the right-hand square defines an excessive virtual Cartier divisor over $i_\cl$ (and hence $i$), it determines a canonical morphism $\msf{Bl}_{X_\cl/Y_\cl} \to (\Bl_{X/Y})_\cl$ which fits in the diagram of cartesian squares
\[\begin{tikzcd}
  \sfE_{X_\cl/Y_\cl} \ar{r}\ar{d}
  & \msf{Bl}_{X_\cl/Y_\cl} \ar{d}\ar[leftarrow]{r}
  & \msf{Bl}_{X_\cl/Y_\cl}\setminus\sfE_{X_\cl/Y_\cl}\ar[equals]{d}
  \\
  (\El_{X/Y})_\cl \ar{r}\ar{d}
  & (\Bl_{X/Y})_\cl \ar[leftarrow]{r}\ar{d}
  & (\Bl_{X/Y})_\cl\setminus (\El_{X/Y})_\cl\ar{d}
  \\
  X_\cl \ar{r}{i_\cl}
  & Y_\cl \ar[leftarrow]{r}
  & Y_\cl \setminus X_\cl.
\end{tikzcd}\]

\begin{prop}\label{prop:compbl}
  The morphisms $\msf{Bl}_{X_\cl/Y_\cl} \to (\Bl_{X/Y})_\cl$ and $\sfE_{X_\cl/Y_\cl} \to (\El_{X/Y})_\cl$ are closed immersions.
  Moreover, the inclusions
  \begin{equation*}
    (\Bl_{X/Y})_\cl \setminus (\El_{X/Y})_\cl
    \simeq \msf{Bl}_{X_\cl/Y_\cl} \setminus \sfE_{X_\cl/Y_\cl}
    \hook \msf{Bl}_{X_\cl/Y_\cl}
    \hook (\Bl_{X/Y})_\cl
  \end{equation*}
  exhibit the classical blow-up $\msf{Bl}_{X_\cl/Y_\cl}$ as the schematic closure of $(\Bl_{X/Y})_\cl \setminus (\El_{X/Y})_\cl$ in $(\Bl_{X/Y})_\cl$.
\end{prop}
\begin{proof}
  This follows from \propref{prop:compMD} (and \remref{rem:Reituba}) by base change along the open immersion $(\Bl_{X/Y})_\cl \hook (\cD_{X/Y})_\cl$ (\lemref{lem:orange}).
\end{proof}

%!TEX root = ../weilres.tex

\section{Algebraicity I}
\label{sec:alg1}

\subsection{Statement}

Our goal in this section is to prove the following algebraicity statement, which in particular implies \thmref{thm:intro/weil}\itemref{item:intro/weil/hfp}:

\begin{thm}\label{thm:alg}
  Let $h : S \to T$ be an afp finite morphism of Tor-amplitude $\le d$.
  Let $f : X_1 \to X_2$ be an $n$-representable locally hfp morphism of derived stacks over $S$.
  Then $h_*(f) : h_*(X_1) \to h_*(X_2)$ is $(n+d)$-representable.

  If $d=0$ (i.e., $h$ is finite flat), then $h_*(f)$ is moreover separated representable if $f$ is separated representable, and $n$-representable with separated $(n+1)$-fold diagonal if $f$ is $n$-representable.
\end{thm}

Using \propref{prop:map=weil} we also obtain the following reformulation in terms of mapping stacks, in particular proving \corref{cor:intro/map}\itemref{item:intro/weil/hfp}:

\begin{cor}\label{cor:algmap}
  Let $X \to S$ and $Y \to S$ be eventually representable morphisms of derived stacks.
  If $X \to S$ is afp, finite, and of Tor-amplitude $\le d$, and $Y \to S$ is locally hfp and $n$-representable, then $\uMaps_S(X, Y) \to S$ is $(n+d)$-representable.
\end{cor}

\subsection{Proof for closed immersions}

We begin by proving the following version of \thmref{thm:alg} in the case of closed immersions:

\begin{thm}\label{thm:clalg}
  Let $h : S \to T$ be an afp finite morphism of Tor-amplitude $\le d$.
  Let $f : X_1 \to X_2$ be an hfp closed immersion of derived stacks over $S$.
  Then the induced morphism $h_*(f) : h_*(X_1) \to h_*(X_2)$ is:
  \begin{thmlist}
    \item $(d-1)$-representable, if $d\ge 1$;
    \item a closed immersion, if $d=0$ (i.e., $h$ is finite flat);
    \item affine, if $f$ is an isomorphism on classical truncations and $d=1$.
  \end{thmlist}
\end{thm}

We will need the following local structural result for hfp closed immersions, due to J.~Lurie.

\begin{thm}[Lurie]\label{thm:hfp}
  Let $f : X \to Y$ be a closed immersion between derived schemes.
  Then $f$ is hfp if and only if Zariski-locally on $Y$, there exists a factorization
  \[
    f : X := X^{(n)} \xrightarrow{f_n}
    X^{(n-1)} \xrightarrow{f_{n-1}}
    \cdots \xrightarrow{f_2}
    X^{(1)} \xrightarrow{f_1}
    X^{(0)} = Y
  \]
  where $n$ is an integer for which $\sL_{X/Y}$ is of Tor-amplitude $\le n$, and the morphism $f_{i+1}$ exhibits $X^{(i+1)}$ as the zero locus of a cosection $\sigma_i : \sE_i \to \sO_{X^{(i)}}$ of a perfect complex $\sE_i \in \QCoh(X^{(i)})$ of Tor-amplitude $[i,i]$ for each $i\ge 0$.
\end{thm}
\begin{proof}
  The condition is clearly sufficient.
  For the other direction, let $f$ be an hfp morphism and assume without loss of generality that $Y$ is affine.
  It follows from the definitions that the induced morphism $f_\cl : X_\cl \to Y_\cl$ is of finite presentation, and the forward direction of \cite[Prop.~3.2.18]{LurieDAG} shows that the relative cotangent complex $\sL_{X/Y}$ is perfect.
  Now, a tower of the desired form is constructed in the proof of the reverse implication of \cite[Prop.~3.2.18]{LurieDAG}.
\end{proof}

\begin{rem}\label{rem:rumply}
  If $Y$ is affine, then the factorization of \thmref{thm:hfp} exists globally on $Y$.
  Moreover, one may take $\sE_i \simeq \sO^{\oplus r_i}[i]$ where $r_i \ge 0$, so that each $X^{(i+1)}$ is the zero locus of some $i$-shifted functions on $X^{(i)}$.
\end{rem}

\begin{lem}\label{lem:goblin}
  Let $h : S \to T$ be an afp proper representable morphism of Tor-amplitude $\le d$.
  Let $f : X_1 \to X_2$ be a closed immersion of derived stacks over $S$ with a factorization
  \[
    X_1 = X^{(m)} \xrightarrow{f_m}
    \cdots \xrightarrow{f_2}
    X^{(1)} \xrightarrow{f_1}
    X^{(0)} = X_2
  \]
  where for each $0\le k<m$ the morphism $f_{k+1}$ exhibits $X^{(k+1)}$ as the zero locus of a cosection of a perfect complex on $X^{(k)}$ of Tor-amplitude $[k,k]$.
  Then the induced morphism $h_*(f) : h_*(X_1)\to h_*(X_2)$ is:
  \begin{thmlist}
    \item $(d-1)$-representable, if $d>0$;
    \item a closed immersion, if $d=0$;
    \item affine, if $d=1$ and $f$ is an isomorphism on classical truncations.
  \end{thmlist}
\end{lem}
\begin{proof}
  Applying \examref{exam:stayship} and \propref{prop:zeroalg}, we see:
  \begin{itemize}
    \item
    The morphism $h_*(f_1) : h_*(X^{(1)}) \to h_*(X^{(0)})$ is a closed immersion if $d=0$ and $(d-1)$-representable if $d>0$.

    \item
    For $0< k< m$, the morphism $h_*(X^{(k+1)}) \to h_*(X^{(k)})$ is a closed immersion if $k\ge d$ and $(d-k-1)$-representable otherwise.
  \end{itemize}
  Now suppose $f$ is an isomorphism on classical truncations.
  Since $f_k$ is always an isomorphism on classical truncations for $k>1$, this implies that the same holds for $f_1 : X^{(1)} \to X^{(0)}$.
  In particular, the cosection defining $X^{(1)}$ is null-homotopic, and we may identify $f_1$ with a projection of the form $\V_{X^{(0)}}(\sE_0) \to X^{(0)}$ where $\sE_0$ is perfect of Tor-amplitude $[1,1]$.
  By \corref{cor:coxcombical} we see that $h_*(f_1) : h_*(X^{(1)}) \to h_*(X^{(0)})$ is affine if $d \le 1$ and $(d-1)$-representable otherwise.
\end{proof}

\begin{proof}[Proof of \thmref{thm:clalg}]
  It will suffice to show that for every affine $T'$ and every morphism $T' \to h_*(X_2)$, the base change $h_*(X_1) \fibprod_{h_*(X_2)} T' \to T'$ has the claimed property.
  By \corref{cor:thinkably} we may replace $T$ by $T'$, $h$ by $h' : S' := S \fibprod_T T' \to T'$, and $f$ by $f' : X_1\fibprod_{X_2} S' \to S'$, and thereby assume that $T$ is affine and $X_2=S$.
  Since $S$ is then also affine, there exists (see \remref{rem:rumply}) a factorization
  \[
    f : Z = Z^{(m_\alpha)} \to
    \cdots \to
    Z^{(1)} \to
    Z^{(0)} = S
  \]
  as in \thmref{thm:hfp}.
  Now the claim follows from \lemref{lem:goblin}.
\end{proof}

\subsection{Proof for smooth separated representable morphisms}

The other special case we consider as preparation for \thmref{thm:alg} is as follows:

\begin{lem}\label{lem:depository}
  Let $h : S \to T$ be an afp finite morphism of Tor-amplitude $\le d$.
  Let $f : X_1 \to X_2$ be a morphism of derived stacks over $S$.
  If $f$ is smooth, separated and representable, then $h_*(f) : h_*(X_1) \to h_*(X_2)$ is $d$-representable.
  Moreover, if $d=0$ (i.e., $h$ is finite flat), then $h_*(f)$ is separated.
\end{lem}

\begin{proof}
  By \corref{cor:thinkably} (as in the proof of \thmref{thm:clalg}) we may assume that $T$ is affine and $X_2=S$.
  It will suffice to show that if $X$ is a separated algebraic space which is smooth over $S$, then $h_*(X)$ is $d$-Artin.
  
  We may choose a scheme $U$ and an étale surjection $U \twoheadrightarrow X$ such that the composite $U \twoheadrightarrow X \to S$ factors via an étale (necessarily representable) morphism to a vector bundle over $S$ (see, e.g., \cite[Tag~\spref{039P}]{Stacks}).
  Then $h_*(U)$ is $d$-Artin (\thmref{thm:flocculency} and \corref{cor:coxcombical}) and $h_*(U) \to h_*(X)$ is an étale surjection (\corref{cor:ontogenic} and \propref{prop:unconversableness}).
  In particular, $h_*(X)$ admits an étale surjection from a derived scheme.
  
  Since $X$ is \emph{separated} and locally hfp, its diagonal is an hfp closed immersion.
  Thus \thmref{thm:clalg} implies that the diagonal $h_*(X) \to h_*(X\fibprod_S X) \simeq h_*(X) \fibprod_T h_*(X)$ is $(d-1)$-representable if $d\ge 1$ and a closed immersion if $d=0$.
  It follows that $h_*(X)$ is $d$-Artin, and moreover separated if $d=0$.
\end{proof}

\subsection{Proof for hfp morphisms}

We need the following standard lemma.

\begin{lem}\label{lem:acturience}
  Let $f : X \to Y$ be a locally of finite type morphism of derived Artin stacks.
  Then there exists a commutative square
  \begin{equation*}
    \begin{tikzcd}
      U \ar{r}{f_0}\ar[twoheadrightarrow]{d}{u}
      & V \ar[twoheadrightarrow]{d}{v}
      \\
      X \ar{r}{f}
      & Y
    \end{tikzcd}
  \end{equation*}
  with $u$ and $v$ smooth surjections and $f_0$ a closed immersion of separated derived schemes.
\end{lem}
\begin{proof}
  See \cite[Lemma~3.18]{dimredcoha}.
\end{proof}

\begin{proof}[Proof of \thmref{thm:alg}]
  By \corref{cor:thinkably} (as in the proof of \thmref{thm:clalg}) we may assume that $T$ is affine and $X_2=S$.
  Let $X := X_1$ be an $n$-Artin stack which is locally hfp over $S$.
  Choose a commutative square
  \begin{equation*}
    \begin{tikzcd}
      U \ar{r}{f_0}\ar[twoheadrightarrow]{d}{u}
      & V \ar[twoheadrightarrow]{d}{v}
      \\
      X \ar{r}{f}
      & S
    \end{tikzcd}
  \end{equation*}
  as in \lemref{lem:acturience}.
  Since $V$ is a separated scheme, $v$ is separated representable, hence $h_*(V)$ is $d$-Artin by \lemref{lem:depository}.
  Since the closed immersion $f_0$ is hfp (as $U$ and $V$ are both locally hfp over $X_2$), \thmref{thm:clalg} implies that $h_*(U) \to h_*(V)$ is $(d-1)$-representable if $d\ge 1$ and a closed immersion if $d=0$.
  In particular, $h_*(U)$ is also $d$-Artin.
  By \corref{cor:ontogenic} and \propref{prop:unconversableness}, $h_*(U) \twoheadrightarrow h_*(X)$ is a smooth surjection.
  In particular, $h_*(X)$ admits a smooth surjection from a derived scheme.
  To show it is $(n+d)$-Artin, it remains to show that it has $(n+d-1)$-representable diagonal.

  Now suppose $X$ is a separated algebraic space (so $n=0$).
  The diagonal $X \to X\fibprod_S X$ is then an hfp closed immersion, so by \thmref{thm:clalg} the diagonal $h_*(X) \to h_*(X\fibprod_S X) \simeq h_*(X) \fibprod_T h_*(X)$ is $(d-1)$-representable if $d\ge 1$ and a closed immersion if $d=0$.
  We conclude that $h_*(X)$ is $d$-Artin, and separated if $d=0$.
  More generally, we have shown at this point that $h_*$ sends separated representable locally hfp morphisms to $d$-representable morphisms (which are moreover separated if $d=0$).

  If $X$ is an arbitrary algebraic space, its diagonal is separated and representable.
  Hence $h_*(X)$ has $d$-representable diagonal by the previous case (which is moreover separated if $d=0$), so that $h_*(X)$ is $(d+1)$-Artin.
  Since it is almost $d$-truncated by \propref{prop:weiltrunc}, it is moreover $d$-Artin.
  More generally, we have shown that $h_*$ sends representable locally hfp morphisms to $d$-representable morphisms (with separated diagonal if $d=0$).

  Finally consider the case $n>0$.
  By induction we may assume that $h_*$ sends $(n-1)$-representable locally hfp morphisms to $(n+d-1)$-representable morphisms (with separated $n$-fold diagonal if $d=0$).
  If $X$ is $n$-Artin then its diagonal is $(n-1)$-representable, hence $h_*(X)$ has $(n+d-1)$-representable diagonal (with separated $n$-fold diagonal if $d=0$).
  We conclude that $h_*(X)$ is $(n+d)$-Artin (with separated $(n+1)$-fold diagonal if $d=0$), as desired.
\end{proof}

\section{Algebraicity II}
\label{sec:alg2}

\subsection{Statement}

Our goal in this section is to prove the following algebraicity statement, which in particular implies \thmref{thm:intro/weil}\itemref{item:intro/weil/vcd}:

\begin{thm}\label{thm:algvcd}
  Let $h : S \to T$ be a virtual Cartier divisor.
  Let $f : X_1 \to X_2$ be a locally of finite type morphism of derived stacks over $S$.
  If $f$ is $n$-representable, then $h_*(f) : h_*(X_1) \to h_*(X_2)$ is $(n+1)$-representable.
\end{thm}
\begin{rem}
  The proof will show moreover that the $(n+2)$-fold diagonal of $h_*(f)$ is affine.
\end{rem}

Using \propref{prop:map=weil} we also obtain the following reformulation in terms of mapping stacks, in particular proving \corref{cor:intro/map}\itemref{item:intro/weil/vcd}:

\begin{cor}\label{cor:algmapvcd}
  Let $X \to S$ and $Y \to S$ be eventually representable morphisms of derived stacks.
  If $X \to S$ is a virtual Cartier divisor, and $Y \to S$ is locally of finite type and $n$-representable, then $\uMaps_S(X, Y) \to S$ is $(n+1)$-representable.
\end{cor}

\subsection{Proof of \thmref{thm:Dart}}
\label{ssec:proofDart}

Before giving the proof of \thmref{thm:algvcd}, let us note that it implies algebraicity of the normal deformation.

Let $f : X \to Y$ be a morphism of derived stacks.
By definition, the morphism $\Dl_{X/Y} \to Y \times \A^1$ is induced by Weil restricting $f: X \to Y$ along the virtual Cartier divisor $0 : Y \to Y\times\A^1$.
Thus if $f$ is locally of finite type and $n$-representable, then \thmref{thm:algvcd} asserts that $\Dl_{X/Y} \to Y \times \A^1$ is $(n+1)$-representable.
In particular, $\Dl_{X/Y} \to Y$ is $(n+1)$-representable.

\subsection{Proof for closed immersions}

For closed immersions, we have the following refinement of \thmref{thm:algvcd}:

\begin{thm}\label{thm:vapidism}
  Let $h : S \to T$ be a virtual Cartier divisor.
  For any closed immersion $f : X_1 \to X_2$ of derived stacks over $S$, $h_*(f) : h_*(X_1) \to h_*(X_2)$ is affine.
\end{thm}

\begin{proof}
  By \corref{cor:thinkably} (as in the proof of \thmref{thm:clalg}) we may assume that $T$ is affine and $X_2=S$.
  Let $Z := X_1$ be a closed derived subscheme of $S$.
  Affineness of the morphism $h_*(Z) \to T$ may be checked Zariski-locally on $T$, so by \lemref{lem:mouthiness} we may further assume that the virtual Cartier divisor is cut out globally by a section $s : T \to \A^1_T$.
  Consider the commutative diagram where all squares are cartesian:
  \begin{equation*}
    \begin{tikzcd}
      Z \ar{r}{g}
      & Z \fibprod_T S \ar{r}\ar{d}
      & S \fibprod_T S \ar{r}\ar{d}
      & S \ar{r}{h}\ar{d}
      & T \ar{d}{s}
      \\
      & Z \ar{r}{f}
      & S \ar{r}{h}
      & T \ar{r}{0}
      & \A^1_T
    \end{tikzcd}
  \end{equation*}
  and $g$ is the graph of $Z \hook S$.
  Since $g$ is an isomorphism on classical truncations, $h_*(g) : h_*(Z) \to h_*(Z\fibprod_T S)$ is affine by \thmref{thm:clalg}.
  It will suffice to show that $h_*(Z \fibprod_T S) \to h_*(S) \simeq T$ is also affine.
  By base change (\lemref{lem:mouthiness}) it is enough to see that the normal deformation $\Dl_{Z/T} \simeq 0_*(Z) \to 0_*(T) \simeq \A^1_T$ is affine (see \propref{prop:defweil}).
  This is \thmref{thm:Daff}.
\end{proof}

\subsection{Proof in general}

The proof is similar to that of \thmref{thm:alg}.
By \corref{cor:thinkably} we may assume that $T$ is affine and $X_2=S$.
Let $X := X_1$ be an $n$-Artin stack which is locally of finite type over $S$.
Choose a commutative square
\begin{equation*}
  \begin{tikzcd}
    U \ar{r}{f_0}\ar[twoheadrightarrow]{d}{u}
    & V \ar[twoheadrightarrow]{d}{v}
    \\
    X \ar{r}{f}
    & S
  \end{tikzcd}
\end{equation*}
as in \lemref{lem:acturience}.
Since $V$ is a separated scheme, $v$ is separated representable, hence $h_*(V)$ is $1$-Artin by \lemref{lem:depository}.
By \thmref{thm:vapidism}, $h_*(U)$ is affine over $h_*(V)$, hence also $1$-Artin.
By \corref{cor:ontogenic} and \propref{prop:unconversableness}, $h_*(U) \twoheadrightarrow h_*(X)$ is a smooth surjection.
In particular, $h_*(X)$ admits a smooth surjection from a derived scheme.

Now suppose $X$ is a separated algebraic space.
The diagonal $X \to X\fibprod_S X$ is then a closed immersion, so by \thmref{thm:vapidism} the diagonal $h_*(X) \to h_*(X\fibprod_S X) \simeq h_*(X) \fibprod_T h_*(X)$ is affine.
We conclude that $h_*(X)$ is $1$-Artin with affine diagonal.
More generally, we have shown at this point that $h_*$ sends separated representable locally of finite type morphisms to $1$-representable morphisms with affine diagonal.

If $X$ is an arbitrary algebraic space, its diagonal is separated and representable.
Hence the diagonal of $h_*(X)$ is $1$-representable with affine diagonal by the previous case, so that $h_*(X)$ is $2$-Artin with affine second diagonal.
Since $h$ is of Tor-amplitude $\le 1$, $h_*(X)$ is almost $1$-truncated by \propref{prop:weiltrunc}.
Thus $h_*(X)$ is moreover $1$-Artin.
More generally, we have shown that $h_*$ sends representable locally of finite type morphisms to $1$-representable morphisms with affine second diagonal.

Finally consider the case $n>0$.
By induction we assume that $h_*$ sends $(n-1)$-representable locally of finite type morphisms to $n$-representable morphisms with affine $(n+1)$-fold diagonal.
If $X$ is $n$-Artin then its diagonal is $(n-1)$-representable, hence $h_*(X)$ has $n$-representable diagonal and affine $(n+2)$-fold diagonal.
We conclude that $h_*(X)$ is $(n+1)$-Artin with affine $(n+2)$-fold diagonal, as desired.

\appendix
%!TEX root = ../weilres.tex

\section{Sharp morphisms}
\label{sec:sharp}

\begin{defn}\label{defn:sharp}
  Let $h : S \to T$ be a morphism of derived stacks.
  We say that $h$ is \emph{sharp} if for every diagram of cartesian squares
  \begin{equation*}
    \begin{tikzcd}
      S'' \ar{r}{h''}\ar{d}{p'}
      & T'' \ar{d}{q'}
      \\
      S' \ar{r}{h'}\ar{d}{p}
      & T' \ar{d}{q}
      \\
      S \ar{r}{h}
      & T,
    \end{tikzcd}
  \end{equation*}
  where $T'$ and $T''$ are affine, the functors $h'^*$ and $h''^*$ admit left adjoints
  \[
    h'_\sharp : \QCoh(S') \to \QCoh(T'), \quad
    h''_\sharp : \QCoh(S'') \to \QCoh(T'')
  \]
  respectively, and the base change transformation
  \begin{equation}\label{eq:wartweed}
    h'^* q'_*
    \xrightarrow{\mrm{unit}} p'_* p'^* h'^* q'_*
    \simeq p'_* h''^* q'^* q'_*
    \xrightarrow{\mrm{counit}} p'_* h''^*.
  \end{equation}
  is invertible.
\end{defn}

\begin{rem}
  Given cartesian squares as in \defnref{defn:sharp}, suppose that the left adjoints $h'_\sharp$ and $h''_\sharp$ exist.
  Then the condition that \eqref{eq:wartweed} is invertible can be reformulated in terms of the left adjoints: it is equivalent to the invertibility of the natural transformation
  \begin{equation}\label{eq:sharpbc}
    h''_\sharp p'^*
    \xrightarrow{\mrm{unit}} h''_\sharp p'^* h'^*h'_\sharp
    \simeq h''_\sharp h''^*q'^*h'_\sharp
    \xrightarrow{\mrm{counit}} q'^* h'_\sharp.
  \end{equation}
\end{rem}

\begin{rem}\label{rem:hsharp}
  If $h$ is sharp, then $h^*$ itself admits a left adjoint $h_\sharp$.
  Indeed, for every affine scheme $T'$ and every morphism $q : T' \to T$, we have the functor $h'_\sharp : \QCoh(S \fibprod_T T') \to \QCoh(T')$ left adjoint to $h'^*$, where $h' : S\fibprod_T T' \to T'$ is the base change.
  These are compatible as $(T', q)$ varies, since they satisfy base change by assumption.
  Hence, passing to the limit over $(T', q)$, they give rise to a functor
  \[ h_\sharp : \QCoh(S) \simeq \lim \QCoh(S \fibprod_T T') \to \lim \QCoh(T') \simeq \QCoh(T). \]
  One immediately verifies that $h_\sharp$ is left adjoint to $h^*$.
  Moreover, the base change transformation
  \[ h'_\sharp p^* \to q^* h_\sharp \]
  is invertible by construction, where $h' : S\fibprod_T T' \to T'$ is the base change of $h$ along any morphism $q : T' \to T$ with $T'$ affine.
  In particular, passing to right adjoints, we also have the base change isomorphism $h^*q_* \to p_* h'^*$.
\end{rem}

\begin{rem}\label{rem:sharpdual}
  Let $h : S \to T$ be a morphism of derived stacks.
  If $h_* : \QCoh(S) \to \QCoh(T)$ preserves perfect complexes, then it is easy to check that the assignment $\sF \mapsto h_*(\sF^\vee)^\vee$ determines a left adjoint to $h^* : \Perf(T) \to \Perf(S)$.

  Conversely, suppose that $h_\sharp$ exists and preserves perfect complexes.
  It is then clear that on perfect complexes, the assignment $\sF \mapsto h_\sharp(\sF^\vee)^\vee$ determines a right adjoint to $h^* : \Perf(T) \to \Perf(S)$.
  Moreover, the following conditions are equivalent:
  \begin{defnlist}
    \item
    The functor $h_* : \QCoh(S) \to \QCoh(T)$ preserves perfect complexes.

    \item\label{item:sharpdual/ii}
    For every $\sF \in \Perf(S)$, the canonical morphism
    \begin{equation}\label{eq:nuprrq}
      h_\sharp(\sF^\vee)^\vee \to h_*(\sF)
    \end{equation}
    is invertible, where \eqref{eq:nuprrq} is defined by transposition from the morphism
    \[
      h^*(h_\sharp(\sF^\vee)^\vee)
      \simeq h^*(h_\sharp(\sF^\vee))^\vee
      \to \sF,
    \]
    dual to the unit $\sF^\vee \to h^*h_\sharp(\sF^\vee)$.
  \end{defnlist}
  Note also that if $h_*$ satisfies base change with respect to $*$-inverse image along any morphism $q : T' \to T$ with $T'$ affine, then the second condition holds.
  Indeed, the invertibility of \eqref{eq:nuprrq} then reduces to the affine case.
  When $T$ is affine, $\QCoh(T)$ is generated under colimits by $\sO_T$; thus, \eqref{eq:nuprrq} is invertible if and only if the induced map
  \[ \Maps_{\QCoh(T)}(\sG, h_\sharp(\sF^\vee)^\vee ) \to \Maps_{\QCoh(T)}(\sG,h_*(\sF)) \]
  is invertible for the \emph{perfect} complex $\sG = \sO_T$.
  Using the fact that $h^*$ is left adjoint to $h_*(-)$ on $\QCoh$, resp. to $h_\sharp(\sF^\vee)^\vee$ on $\Perf$, this map is identified with the identity map on $\Maps_{\QCoh(S)}(h^*(\sG), \sF)$.
\end{rem}

\begin{lem}\label{lem:fortuitist}
  Let $h : S \to T$ be a qcqs $1$-representable morphism which is universally of cohomological dimension $\le e$\footnote{%
     see e.g. \cite[\S 1.5]{kstack} or \cite[A.1.4]{HalpernLeistnerPreygel}; for example, qcqs representable morphisms belong to this class.
  }.
  If $h$ is sharp, then we have:
  \begin{thmlist}
    \item\label{item:fortuitist/zimvuja} The functor $h_\sharp$ preserves perfect complexes.
    \item If moreover $h$ is of Tor-amplitude $\le d$, then $h_\sharp$ sends perfect complexes of Tor-amplitude $[a,b]$ to perfect complexes of Tor-amplitude $[a-d,b+e]$.
  \end{thmlist}
\end{lem}
\begin{proof}
  By the base change formula \eqref{eq:sharpbc} we may assume that $T$ is affine and $S$ is a qcqs $1$-Artin stack of cohomological dimension $\le e$.
  In particular, the perfect complexes on $T$ are precisely the compact objects of $\QCoh(T)$, and every perfect complex $\sF \in \Perf(S)$ is compact (see \cite[Rem.~1.21, Lem.~1.43]{kstack}).
  To see that $h_\sharp(\sF)$ is perfect, it thus suffices to observe that the functor
  \begin{equation*}
    \Maps_{\QCoh(T)}(h_\sharp(\sF), -)
    \simeq \Maps_{\QCoh(S)}(\sF, h^*(-))
  \end{equation*}
  preserves filtered colimits since $\sF$ is compact.

  For the second assertion, the condition that $h$ is of Tor-amplitude $\le d$ means that $h^*$ sends $(a-d-1)$-coconnective complexes to $(a-1)$-coconnective complexes.
  By adjunction, it follows that $h_\sharp$ sends $a$-connective complexes to $(a-d)$-connective complexes.
  It remains to show that $h_\sharp$ sends perfect complexes of Tor-amplitude $\le b$ to perfect complexes of Tor-amplitude $\le b+e$.
  By \remref{rem:sharpdual} and since $h_*$ is stable under arbitrary base change by \cite[Prop.~A.1.5]{HalpernLeistnerPreygel}, assertion~\itemref{item:fortuitist/zimvuja} implies that $h_*$ preserves perfect complexes and that $h_\sharp(\sF) \simeq h_*(\sF^\vee)^\vee$ for $\sF \in \Perf(S)$.
  If $\sF \in \Perf(S)$ is of Tor-amplitude $\leq b$, so that $\sF^\vee$ is $(-b)$-connective, then $h_*(\sF^\vee)$ is $(-b-e)$-connective and hence $h_\sharp(\sF)$ is of Tor-amplitude $\leq b+e$.
\end{proof}

We give a sufficient criterion for sharpness.
Following \cite[Def.~1.38]{kstack}, we say that a derived stack $X$ is called \emph{perfect} if $\QCoh(X)$ is compactly generated and the compact objects are the perfect complexes.

\begin{lem}\label{lem:neurology}
  Let $h : S \to T$ be a morphism of derived stacks which is eventually representable.
  Suppose that for every affine scheme $T'$ and every morphism $q : T' \to T$, we have:
  \begin{thmlist}
    \item The derived stack $S':= S \fibprod_T T'$ is perfect.
    \item The functor $h'_*$ preserves perfect complexes, where $h' : S' \to T'$ is the base change of $h$.
    \item Given a cartesian square
    \begin{equation}\label{eq:depth}
      \begin{tikzcd}
        S'' \ar{r}{h''}\ar{d}{p'}
        & T'' \ar{d}{q'}
        \\
        S' \ar{r}{h'}
        & T'
      \end{tikzcd}
    \end{equation}
    with $T''$ affine, the base change transformation
    \begin{equation}
      q'^* h'_* 
      \xrightarrow{\mrm{unit}} q'^* h'_* p'_* p'^*
      \simeq q'^* q'_* h''_* p'^*
      \xrightarrow{\mrm{counit}} h''_* p'^*
    \end{equation}
    restricted to $\Perf(S')$ is invertible.
  \end{thmlist}
  Then $h$ is sharp.
\end{lem}
\begin{proof}
  If $h'_*$ preserves perfect complexes, then the formula
  \begin{equation}\label{eq:donkey}
    h'_\sharp(\sF) \simeq h'_*(\sF^\vee)^\vee
  \end{equation}
  determines a left adjoint $h'_\sharp : \Perf(S') \to \Perf(T')$ to
  $h'^* : \Perf(T') \to \Perf(S')$ (\remref{rem:sharpdual}).
  Note that 
  \[
    \Ind(\Perf(T')) \simeq \QCoh(T'),
    \quad \Ind(\Perf(S')) \simeq \QCoh(S')
  \]
  since $T'$ and $S'$ are perfect by assumption.
  Thus there is a unique colimit-preserving extension
  \begin{equation*}
    h'_\sharp : \QCoh(S') \to \QCoh(T')
  \end{equation*}
  which is still left adjoint to $h'^*$ on $\QCoh$.
  
  To conclude that $h$ is sharp, it remains to verify the base change formula $h''_\sharp p'^*(\sF) \to q'^* h'_\sharp(\sF)$ \eqref{eq:sharpbc} for every cartesian square \eqref{eq:depth} and every $\sF \in \QCoh(S')$.
  Since each functor commutes with colimits, we may assume that $\sF$ is a perfect complex.
  Then, using the formula \eqref{eq:donkey} and the fact that $p'^*$ and $q'^*$ commute with duals (since they are symmetric monoidal), the morphism in question is dual to the base change transformation
  \[ q'^*h'_*(\sF^\vee) \to h''_* p'^* (\sF^\vee) \]
  which is invertible by assumption.
\end{proof}

\begin{exam}\label{exam:stayship}
  If $h : S \to T$ is afp proper representable and of finite Tor-amplitude, then it is sharp.\footnote{%
    See also \cite[\S~6.4.5]{LurieSAG} for related discussion in the context of Deligne--Mumford stacks.
  }
  Indeed, let us verify the conditions of \lemref{lem:neurology} for the base change $h' : S' \to T'$ of $h$ to an affine scheme $T'$.
  Perfectness of $S'$ holds by \cite[Ex.~1.41]{kstack}, $h'_*$ preserves perfect complexes by \cite[Thm.~6.1.3.2]{LurieSAG}, and the base change formula $q'^* h'_* \simeq h''_* p'^*$ holds by \cite[Cor.~3.4.2.2]{LurieSAG}.
  Moreover, in this case $h_\sharp : \QCoh(S) \to \QCoh(T)$ is given by the formula
  \begin{equation}
    h_\sharp(-) \simeq h_*(- \otimes \omega_{S/T}),
  \end{equation}
  where $\omega_{S/T} := h^!(\sO_T) \in \QCoh(S)$ is the dualizing complex (see \cite[Prop.~6.4.5.3]{LurieSAG}).
\end{exam}

\begin{exam}\label{exam:sharpvcd}
  If $h : D \hook T$ is a virtual Cartier divisor, then it is sharp with
  \begin{equation}
    h_\sharp(-) \simeq h_*(- \otimes \sO_D(D)[-1]).
  \end{equation}
  This follows from \examref{exam:stayship}, since the relative dualizing complex is given by $\omega_{D/T} \simeq \sO_D(D)[-1]$ (see \cite[Prop.~4.2]{kblow}).
\end{exam}

\begin{exam}
  Let $S$ and $T$ be locally noetherian derived Artin stacks such that $T$ locally admits a dualizing complex.
  Let $h : S \to T$ be a morphism which is qcqs, locally afp, of finite Tor-amplitude, and satisfies the coherent push-forward property in the sense of \cite[Def.~2.4.1]{HalpernLeistnerPreygel}.
  Then $h$ behaves like a sharp morphism when tested on noetherian $T'$:
  \begin{defnlist}
    \item 
    For any morphism $T'\to T$, where $T'$ is a noetherian affine scheme, consider the base change $h' : S' \to T'$.
    Then $h'^*$ admits a left adjoint $h'_\sharp$ by \cite[Prop.~5.1.6]{HalpernLeistnerPreygel}.
    (Note that, by definition, the coherent push-forward property is stable under noetherian base change, so the assumptions on $h$ persist to $h'$.)
    
    \item
    For any morphism of noetherian affine schemes $q' : T'' \to T'$, $h'_\sharp$ commutes with $q'^*$ (i.e., the natural transformation \eqref{eq:sharpbc} is invertible) by \cite[Lem.~5.1.8]{HalpernLeistnerPreygel}.
    (Note that any morphism of affine schemes is qcqs, representable and universally of cohomological dimension $0$.)
  \end{defnlist}
  We recall that the coherent push-forward property is satisfied by formally proper morphisms in the sense of \cite[Def.~1.1.3]{HalpernLeistnerPreygel} (see \cite[Thm.~2.4.3]{HalpernLeistnerPreygel}).
\end{exam}

%!TEX root = ../weilres.tex

\section{Derived vector bundles}
\label{sec:vb}

Let $\sE \in \QCoh(S)$ be a quasi-coherent complex on a derived stack $S$.
The derived stack $\V_S(\sE)$ over $S$ is defined such that there are isomorphisms
\begin{equation}\label{eq:besnivel}
  \Maps_S(T, \V_S(\sE))
  \simeq \Maps_{\QCoh(T)}(t^*\sE, \sO_T),
\end{equation}
functorial in $(t : T \to S) \in \dStk_{/S}$.
By definition, sections of the projection $\pi : \V_S(\sE) \to S$ are \emph{cosections} $\sE \to \sO_S$.
When $\sE$ is perfect, we refer to $\V_S(\sE)$ as the \emph{derived vector bundle} associated with $\sE$; its sheaf of sections is the dual $\sE^\vee$.

If $\sE$ is connective, then $\V_S(\sE)$ is identified with the relative spectrum
\begin{equation}\label{eq:benzalhydrazine}
  \V_S(\sE) \simeq \uSpec_S(\Sym^*_{\sO_S}(\sE)),
\end{equation}
where $\Sym^*_{\sO_S}(\sE)$ denotes the derived symmetric algebra of $\sE$.

\begin{prop}\label{prop:unpiloted}
  If $\sE$ is perfect, then the projection $\pi : \V_S(\sE) \to S$ is locally hfp.
  If $\sE$ is $m$-pseudo-coherent for some $m\in\Z$, then $\pi$ is locally fp to order $m-1$.
  If $\sE$ is pseudo-coherent, then $\pi$ is locally afp.
\end{prop}
\begin{proof}
  We may assume $S=\Spec(R)$ is affine and consider the construction
  \begin{equation*}
    \Maps_{\Spec(R)}(\Spec(A), \V_S(\sE))
    \simeq \Maps_{\on{D}(R)}(M, A),
  \end{equation*}
  where $M = \Gamma(S, \sE)$, as a functor in $R$-algebras $A$.
  If $\sE$ (hence $M$) is perfect, this commutes with filtered colimits in $A$.
  If $\sE$ is $m$-pseudo-coherent, then it commutes with filtered colimits in $A$ when restricted to $(m-1)$-truncated $R$-algebras (see \cite[Rem.~2.7.0.5]{LurieSAG}).
  This shows the first two statements, and the third follows from the second by definition.
\end{proof}

\begin{prop}\label{prop:adance}
  If $\sE \in \QCoh(S)$ is locally eventually connective, then the projection $\pi : \V_S(\sE) \to S$ has relative cotangent complex $\pi^*(\sE)$.
  In particular, if $\sE$ is perfect of amplitude $\le 0$ (resp. $\le 1$), then $\pi$ is smooth (resp. quasi-smooth).
\end{prop}
\begin{proof}
  It is clear that $\pi^*(\sE)$ satisfies the universal property of $\sL_{\V_S(\sE)/S}$ \eqref{eq:derivations}.
  Since $\pi$ is locally hfp when $\sE$ is perfect (\propref{prop:unpiloted}), the second statement follows by definition.
\end{proof}

\begin{rem}
  Note that the construction $\sE \mapsto \V_S(\sE)$ sends colimits in $\QCoh(S)$ to limits in $\dStk_{/S}$.
  In particular, if $\sE \to \sF \to \sG$ is an exact triangle in $\QCoh(S)$, then there is a cartesian square
  \begin{equation*}
    \begin{tikzcd}
      \V_S(\sG) \ar{r}\ar{d}
      & \V_S(\sF) \ar{d}
      \\
      S \ar{r}{0}
      & \V_S(\sE),
    \end{tikzcd}
  \end{equation*}
  where the zero section $0 : S \to \V_S(\sE)$ corresponds to the zero morphism $0 : \sE \to \sO_S$ under \eqref{eq:besnivel}.
\end{rem}

It is well-known that $\V_S(\sE)$ is $n$-Artin if $\sE$ is a $(-n)$-connective perfect complex.
We show that perfectness can be relaxed to pseudo-coherence, or even $0$-pseudo-coherence.

\begin{thm}\label{thm:V}
  Let $\sE \in \QCoh(S)$ be a quasi-coherent complex on a derived stack $S$.
  \begin{thmlist}
    \item
    If $\sE$ is connective, then the projection $\V_S(\sE) \to S$ is affine.
    \item
    If $\sE$ is $0$-pseudo-coherent and $(-n)$-connective for some $n>0$, then the projection $\V_S(\sE) \to S$ is $n$-representable.
  \end{thmlist}
\end{thm}
\begin{proof}
  The first claim is immediate from \eqref{eq:benzalhydrazine}.
  For the second, we may assume that $S$ is affine and show that $\V_S(\sE)$ is $n$-Artin.
  Since $\sE$ is $0$-pseudo-coherent, there exists a perfect complex $\sE' \in \Perf(S)$ of Tor-amplitude $\le 0$ and a morphism $\sE' \to \sE$ with $0$-connective fibre $\sF$.
  Since $\sE$ is $(-n)$-connective and $\pi_{-i}(\sE') \simeq \pi_{-i}(\sE)$ for $i>0$, $\sE'$ is also $(-n)$-connective.
  Since $\sF$ is connective, $\V_S(\sF)$ is affine by the first claim.
  Thus the cartesian square
  \begin{equation*}
    \begin{tikzcd}
      \V_S(\sE) \ar{r}\ar{d}
      & \V_S(\sE') \ar{d}
      \\
      S \ar{r}{0}
      & \V_S(\sF)
    \end{tikzcd}
  \end{equation*}
  shows that if $\V_S(\sE')$ is $n$-Artin, then so is $\V_S(\sE)$.
  
  We replace $\sE$ by $\sE'$ and thereby assume $\sE$ is $(-n)$-connective and perfect of Tor-amplitude $\le 0$.
  By induction on $n$, we assume the claim is proven for $(-n+1)$-connective perfect complexes of Tor-amplitude $\le 0$.
  Since $S$ is affine and $\pi_{-n}(\sE)$ is of finite presentation as a $\pi_0(\sO_S)$-module (see \cite[Cor.~7.2.4.5]{LurieHA}), we may choose a morphism $\sE_{-n}[-n] \to \sE$ where $\sE_{-n}$ is of amplitude $[0,0]$ and whose fibre $\sG$ is $(-n)$-connective and of Tor-amplitude $\le -1$.
  We have a cartesian square
  \begin{equation*}
    \begin{tikzcd}
      \V_S(\sE) \ar{r}\ar{d}
      & \V_S(\sE_{-n}[-n]) \ar{d}
      \\
      S \ar{r}{0}
      & \V_S(\sG).
    \end{tikzcd}
  \end{equation*}
  It will suffice to show that $\V_S(\sE_{-n}[-n])$ and $\V_S(\sG)$ are $n$-Artin.
  Replacing $\sE$ by $\sG$ or $\sE_{-n}[-n]$, we may thus assume that $\sE$ is $(-n)$-connective where $n>0$, and of Tor-amplitude $\le -1$.
  In this case, every $S$-morphism $T \to \V_S(\sE)$ factors up to homotopy through the zero section.
  In particular, $T \fibprod_{\V_S(\sE)} T' \simeq \V_{T \fibprod_S T'}(\sE[1]|_{T \fibprod_S T'})$ for any two morphisms $T \to \V_S(\sE)$ and $T' \to \V_S(\sE)$ with $T,T'$ affine.
  By the induction hypothesis, $\V_{T \fibprod_S T'}(\sE[1]|_{T \fibprod_S T'}) \simeq \V_T(\sE[1]|_T) \fibprod_S T'$ is $(n-1)$-Artin.
  It follows that $\V_S(\sE)$ has $(n-1)$-representable diagonal.
  Moreover, the zero section $0 : S \to \V_S(\sE)$ is smooth in this case by \propref{prop:adance}.
  Putting everything together, we have proven that:
  \begin{enumerate}
    \item $\V_S(\sE)$ has $(n-1)$-representable diagonal.
    \item The zero section $0 : S \to \V_S(\sE)$ is a smooth morphism whose base change along any morphism $T \to \V_S(\sE)$ admits a section over $T$.
    \emph{A fortiori}, $0 : S \to \V_S(\sE)$ admits étale-local sections.
  \end{enumerate}
  It follows that $\V_S(\sE)$ is $n$-Artin.
\end{proof}

%%%%%%%%%%%%%%%%%%%%%%%%%%%%%%%%%%%%%%%%%%%%%%%%%%%%%%%%%%%%%%%%%%%%%%%%%%%

%!TEX root = weilres.tex

\bibliographystyle{halphanum}

\setlength{\parindent}{0em}

\end{document}